\documentclass[11pt]{amsart}
\usepackage{amsmath}
\usepackage{amsfonts}
\usepackage{amssymb}
\usepackage[all]{xy}           
\usepackage{bbding}
\usepackage{txfonts}
\usepackage{amscd}
\usepackage[shortlabels]{enumitem}
\usepackage{ifpdf}
\ifpdf
  \usepackage[colorlinks,final,
  hyperindex]{hyperref}
\else
  \usepackage[colorlinks,final,
  hyperindex]{hyperref}
\fi
\usepackage{tikz}
\usepackage[active]{srcltx}

\makeatletter

\numberwithin{equation}{section}

\newtheorem{thm}{Theorem}[section]
\newtheorem{lem}[thm]{Lemma}
\newtheorem{cor}[thm]{Corollary}
\newtheorem{pro}[thm]{Proposition}
\newtheorem{ex}[thm]{Example}
\newtheorem{rmk}[thm]{Remark}
\newtheorem{defi}[thm]{Definition}

\newcommand{\gl}{\mathrm{End}}
\newcommand{\kl}{\mathfrak{l}}
\newcommand{\kr}{\mathfrak{r}}
\newcommand{\ku}{\mathfrak{u}}

\newcommand{\fu}{\mathbf u}
\newcommand{\fd}{\mathbf{d}}
\newcommand{\fl}{\mathbf{l}}
\newcommand{\fr}{\mathbf{r}}

\newcommand{\DAYBE}{\mathrm{DAYBE}}
\newcommand{\NYBE}{\mathrm{NYBE}}
\newcommand{\PYBE}{\mathrm{PYBE}}
\newcommand{\DPYBE}{\mathrm{DPYBE}}
\newcommand{\DNYBE}{\mathrm{dNYBE}}
\newcommand{\CYBE}{\mathrm{CYBD}}

\newcommand{\Img}{\mathrm{Im}}

\def\id{\mathop {\fam0 id}\nolimits}

\begin{document}

\title[Differential perm bialgebras and related bialgebra structures]
{Differential perm bialgebras, differential perm Yang-Baxter equation, and their
applications in related bialgebra structures}

\author{Ahuan Dong}
\address{School of Mathematics and Statistics, Henan University, Kaifeng 475004,
P.R. China}
\email{dah@henu.edu.cn}

\author{Bo Hou}
\address{School of Mathematics and Statistics, Henan University, Kaifeng 475004,
China}
\email{houbo@henu.edu.cn, bohou1981@163.com}


\begin{abstract}
We introduce the differential perm bialgebras by generalizing the study of perm bialgebras
to the context of differential perm algebras. They are characterized by the Manin triple of
differential perm algebras as well as the matched pairs of differential perm algebras.
The quasi-triangular (resp. triangular, factorizable) differential perm bialgebra are
introduced. In particular, symmetric solutions of an analogue of perm Yang-Baxter equation
in differential perm algebras provide triangular differential perm bialgebras, whereas
in turn the notions of $\mathcal{O}$-operators of differential perm algebras. As an
application, the relationship between differential perm bialgebras and related bialgebra
structures is discussed in detail. We show that there is a natural differential
infinitesimal bialgebra structure on the tensor product of a differential perm bialgebra
and a quadratic Zinbiel algebra, and this differential infinitesimal bialgebra structure
is quasi-triangular (resp. triangular, factorizable) if the original differential perm
bialgebra is quasi-triangular (resp. triangular, factorizable). Starting from a
differential perm bialgebras with certain conditions, we can construct Novikov
bialgebras, and these Novikov bialgebras can also be implemented by using the tensor
product of diNovikov bialgebras and quadratic Zinbiel algebras.
\end{abstract}

\makeatletter
\@namedef{subjclassname@2020}{\textup{2020} Mathematics Subject Classification}
\makeatother

\maketitle
\vspace{-6mm}
\tableofcontents 
\vspace{-5mm}



\vspace{-4mm}

\section{Introduction}\label{sec:intr}

The aim of this paper is to study the bialgebra theory of differential perm algebras by
generalizing the bialgebra theory of perm algebras, to construct diNovikov bialgebras
by differential perm bialgebras. As a direct application of this construction, we present
the close connection between differential infinitesimal bialgebras,
differential perm bialgebras, Novikov bialgebras and diNovikov bialgebras.

Lie bialgebras, introduced by Drinfeld \cite{Dri}, are the infinitesimal versions of
Poisson-Lie groups, and play a crucial role in the study of quantum groups \cite{CP,Kas}.
For associative algebras, the notion of infinitesimal bialgebras was introduced by Joni and
Rota to provide an algebraic framework for the calculus of divided differences \cite{JR,Agu}.
Antisymmetric infinitesimal bialgebras, introduced in \cite{Bai}, regarded as the
associative analogues of Lie bialgebras. Generally speaking, a bialgebra structure
on a given algebra structure is obtained as a coalgebra structure together which
gives the same algebra structure on the dual space with a set of compatibility conditions.
In the past few years a similar approach has been taken to establish the bialgebra
theories for quite a few other algebraic structures, such as left-symmetric algebras
\cite{Bai1}, Poisson algebras \cite{NB}, Leibniz algebras \cite{TS}, Jordan algebras
\cite{Zhe}, Novikov algebras \cite{HBG}, perm algebras \cite{Hou,LZB}, Jacobi-Jordan
algebras \cite{BCHM}, differential algebras \cite{LLB}, transposed Poisson algebras
\cite{LB}, and so on.

As an important class of quasi-triangular Lie bialgebras, factorizable Lie bialgebras
are used to connect classical $r$-matrices with certain factorization problems, and have
various applications in integrable systems \cite{BGN,RS}. Recently, the factorizable Lie
bialgebras have received further research \cite{LS}, and the factorizable antisymmetric
infinitesimal bialgebras \cite{SW}, factorizable Leibniz bialgebras \cite{BLST},
factorizable perm bialgebras \cite{Lin}, factorizable Novikov bialgebras \cite{CH}
have been studied. In this paper, we introduce the differential perm bialgebras by
generalizing the study of perm bialgebras to the context of differential perm algebras,
and give some close connection between differential perm bialgebras, differential
infinitesimal bialgebras, Novikov bialgebras and diNovikov bialgebras.

\smallskip\noindent
1.1. {\bf Bialgebra theory of differential perm algebras.}
A (noncommutative) differential algebra is an associative algebra equipped
with a linear operator satisfying the Leibniz rule, which have applications in
mathematical physics such as Yang-Mills theory. A differential perm algebra is a
perm algebra equipped with a derivation, which could be regarded as a special differential
algebra. In this paper, we establish a bialgebra theory for differential perm algebras,
called differential perm bialgebras, by extending the study of perm bialgebras in
\cite{Hou,LZB} to the context of differential perm algebras. The derivations in a
differential perm algebra play an important role in such generalizations by introducing
an admissibility condition between the linear operators and the differential perm algebra,
which gives a reasonable bimodule over the differential perm algebra on the dual space.
Such an admissibility leads to the compatibility among the product, the coproduct and
the linear operators, so that the theory of perm bialgebras can be extended
to differential perm bialgebras.

Explicitly, we introduce the concepts of matched pairs and Manin triples of differential
perm algebras as the generalizations of matched pairs and Manin triples of perm algebras
respectively, and show that the differential perm bialgebras are
characterized equivalently by matched pairs of differential perm algebras and Manin
triples of differential perm algebras. The coboundary cases lead to the introduction
of the notion of differential perm Yang-Baxter equation ($\DPYBE$) whose solutions are
closely related to the quasi-triangular, triangular and factorizable differential perm
algebras. The notions of $\mathcal{O}$-operators of differential perm algebras is
introduced, and a one-to-one correspondence between $\mathcal{O}$-operators and
symmetric solutions of $\DPYBE$ is provided. We summarize these results in the
following diagram: \\[-8mm]

{\tiny
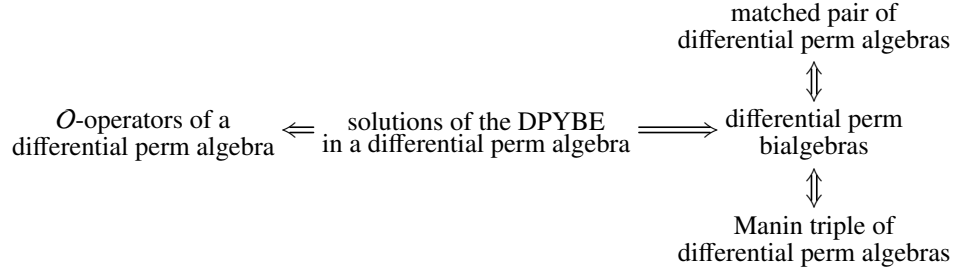
$$
\xymatrix@C=0.4cm@R=0.5cm{
&&\txt{{\small matched pair of}\\ {\small differential perm algebras}}\\
\txt{{\small $\mathcal{O}$-operators of a} \\ {\small differential perm algebra}}&
\txt{{\small solutions of the $\DPYBE$} \\ {\small in a differential perm algebra}}
\ar@2{->}[r] \ar@2{->}[l]&
\txt{{\small differential perm}\\ {\small bialgebras}}\ar@2{<->}[d]\ar@2{<->}[u]\\
&&\txt{{\small Manin triple of}\\ {\small differential perm algebras}}
}
$$}

\smallskip\noindent
1.2. {\bf The relationship between differential perm algebras and related bialgebra
structures.}
Recently, the relationship between different types of bialgebra structures
has received a lot of attention. Since the operad of (left) Novikov algebras
and the operad of right Novikov algebras are Koszul dual, in \cite{HBG}, Hong, Bai
and Guo have proposed a method for constructing infinite-dimensional Lie bialgebras
using the affinization of Novikov bialgebras. Similarly, Lin, Zhou and Bai have
constructed infinite-dimensional Lie bialgebras by using the pre-Lie bialgebras
and perm bialgebras, respectively \cite{LZB}. Hou and Lin have constructed Lie bialgebras
by using the Leibniz bialgebras and Zinbiel bialgebras \cite{HL}. In \cite{Hou1}, we have
provided a method for constructing infinite-dimensional antisymmetric infinitesimal
bialgebras using the affinization of dendriform $\mathrm{D}$-bialgebras (this affinization
independent of Koszul dual), and given some methods for constructing Lie bialgebras by
discussing in detail the relationship between dendriform $D$-bialgebras, pre-Lie bialgebras,
antisymmetric infinitesimal bialgebras and Lie bialgebras. Moreover, the relationship
between Novikov bialgebras and Lie conformal bialgebras, between diNovikov bialgebras
and Leibniz conformal bialgebras have been discussed in \cite{HBG1}, \cite{HBG2} and
\cite{XBH}. Here, we give some close connection between differential perm bialgebras,
differential infinitesimal bialgebras, diNovikov bialgebras andNovikov bialgebras as follows:
\begin{enumerate}\itemsep=0pt
\item[$\bullet$] There is a differential infinitesimal bialgebra structure on the tensor
product of a differential perm bialgebra and a quadratic Zinbiel algebra, and the
differential infinitesimal bialgebra is quasi-triangular (resp. triangular, factorizable)
if the original differential perm bialgebra is quasi-triangular (resp. triangular,
factorizable).
\item[$\bullet$] Each differential perm bialgebra $(P, \diamond, \vartheta, \fd, -\fd)$
induces a diNovikov bialgebra, and this induced diNovikov bialgebra is triangular if the
original differential perm bialgebra is triangular.
\item[$\bullet$] There is a Novikov bialgebra structure on the tensor product
of a diNovikov bialgebra and a quadratic Zinbiel algebra, and this induced Novikov bialgebra
is triangular if the original diNovikov bialgebra is triangular.
\item[$\bullet$] Each differential infinitesimal bialgebra $(A, \cdot, \Delta,
\fd, -\fd)$ induces a Novikov bialgebra, and the induced Novikov bialgebra is
quasi-triangular (resp. triangular, factorizable) if the original differential
infinitesimal bialgebra is quasi-triangular (resp. triangular, factorizable).
The specific proof of this part can be found in \cite{HBG1,CH}.
\end{enumerate}
Therefore, starting from the same differential perm bialgebra, there are two methods
to construct Novikov bialgebras. We can show that the Novikov bialgebras constructed
by these two methods are consistent. Thus, we obtain the following commutative
diagram (see Theorem \ref{thm:commdig}):
$$
\xymatrix@C=2cm@R=0.7cm{
\txt{$(P, \diamond, \vartheta, \fd, -\fd)$ \\ {\tiny a differential perm bialgebra}}
\ar[d]_{{\rm Thm.}~\ref{thm:ind-dinovbia}} \ar[r]^{{\rm Thm.}~\ref{thm:permbia-ASI}}
&\txt{$(P\otimes C, \cdot, \Delta, \hat{\fd}, -\hat{\fd})$\\
{\tiny a differential infinitesimal bialgebra}}\ar[d]^{{\rm Thm.}~\ref{thm:ind-novbia}} \\
\txt{$(P, \dashv_{\fd}, \vdash_{\fd}, \nu_{\dashv,\fd}, \nu_{\vdash,\fd})$ \\
{\tiny a diNovikov bialgebra}} \ar[r]^{{\rm Thm.}~\ref{thm:diNbia-Nbia}}
& \txt{$(P\otimes C, \ast_{\hat{\fd}}, \delta_{\hat{\fd}})$ \\ {\tiny a Novikov bialgebra}}}
$$
Moreover, by using the correspondence between Novikov bialgebras and Lie conformal bialgebras,
diNovikov bialgebras and Leibniz conformal bialgebras, these conclusions can also be applied
to the construction of Lie conformal bialgebras and Leibniz conformal bialgebras.

In \cite{Kup}, Kupershmidt found that the classical Yang-Baxter equation (or $\CYBE$)
in tensor form on Lie algebras can be converted into an $\mathcal{O}$-operator associated
to the coregular representation. This conclusion has been proven to be valid for various
algebra structures.
In this paper,
by discussing the connection between the solutions of the differential associative
Yang-Baxter equation in an admissible differential algebra $(A, \cdot, \fd, -\fd)$ and
the solutions of the Novikov Yang-Baxter equation in the induced Novikov algebra $(A,
\ast_{\fd})$, the solutions of the $\DPYBE$ in the induced admissible differential
perm algebra $(A\otimes P, \hat{\diamond}, \hat{\fd}, -\hat{\fd})$ (by a perm algebra
$(P, \diamond)$), as well as the solutions of the diNovikov Yang-Baxter equation in the
related diNovikov algebra $(A\otimes P, \dashv_{\hat{\fd}}, \vdash_{\hat{\fd}})$,
we present the close connections among triangular differential infinitesimal bialgebras
triangular Novikov bialgebras, triangular differential perm bialgebras and triangular
diNovikov bialgebras, as well as the connections among $\mathcal{O}$-operators on
differential algebras, Novikov algebras, differential perm algebras and diNovikov algebras.
These results, when put together, yield the following commutative diagram:\\[-6mm]

{\small
\begin{displaymath}
\xymatrix@R=0.4cm@C=-0.8cm{
&\txt{{\small $(P, \diamond, \vartheta_{r}, \fd, -\fd)$ }\\
{\tiny a triangular differential perm bialgebra}}
\ar@{->}[rr]|-{\txt{\tiny\textcolor{red}{Thm.}~\ref{thm:ind-dinovbia}}}
\ar@{->}[ld]|-{\txt{\tiny\textcolor{red}{Thm.}~\ref{thm:indu-asssdibia}}}
\ar@{<.}[dd]|-(0.8){\txt{\tiny\textcolor{red}{Pro.}~\ref{pro:quasi-dpba}}}&
&\txt{{\small $(P, \dashv_{\fd}, \vdash_{\fd}, \nu_{\dashv,r}, \nu_{\vdash,r})$} \\
{\tiny a triangular diNovikov bialgebra}}
\ar@{->}[ld]|-{\txt{\tiny\textcolor{red}{Thm.}~\ref{thm:diNbia-Nbia}}}
\ar@{<-}[dd]|-{\txt{\tiny\textcolor{red}{Pro.}~\ref{pro:tri-dinov}}}&\\
\txt{{\small $(P\otimes C, \cdot, \Delta_{\widehat{r}}, \hat{\fd}, -\hat{\fd})$} \\
{\tiny a triangular differential infinitesimal bialgebra}}
\ar@{<-}[dd]|-{\txt{\tiny\textcolor{red}{Pro.}~\ref{pro:diff-bia}}}
\ar@{->}[rr]|-(0.75){\txt{\tiny\textcolor{red}{Thm.}~\ref{thm:indu-spNbia}}}&
&\txt{{\small $(P\otimes C, \ast, \delta_{\hat{\fd}})$}\\
{\tiny a triangular Novikov bialgebra}}
\ar@{<-}[dd]|-(0.7){\txt{\tiny\textcolor{red}{Pro.}~\ref{pro:quasass-Nbia}}} \\
&\txt{{\small\bf $r$}\\ {\tiny\bf a symmetric solution} \\
{\tiny\bf of the $\DPYBE$ in $(P, \diamond, \fd, -\fd)$}}
\ar@{.>}[dd]|-(0.75){\txt{\tiny\textcolor{red}{Pro.}~\ref{pro:o-dperm}}}
\ar@{.>}[rr]|-(0.2){\rm\textcolor{red}{Pro.}~\ref{pro:ind-DPYBE}}&
&\txt{{\small\bf $r$} \\  {\tiny\bf a symmetric solution} \\
{\tiny\bf of the $\DNYBE$ in $(P, \dashv_{\fd}, \vdash_{\fd})$}}
\ar@{->}[dd]|-{\txt{\tiny\textcolor{red}{Pro.}~\ref{pro:o-dinov}}}\\
\txt{{\small\bf $\widehat{r}$}\\ {\tiny\bf a skew-symmetric solution} \\
{\tiny\bf of the $\DAYBE$ in $(P\otimes C, \cdot, \hat{\fd}, -\hat{\fd})$}}
\ar@{->}[dd]|-{\txt{\tiny\textcolor{red}{Pro.}~\ref{pro:o-dass}}}
\ar@{<.}[ru]|-{\txt{\tiny\textcolor{red}{Pro.}~\ref{pro:DPYBE-DAYBE}}}
\ar@{->}[rr]|-(0.2){\txt{\tiny\textcolor{red}{Pro.}~\ref{pro:indu-NYBE}}}&
&\txt{{\small\bf $\widehat{r}$}\\ {\tiny\bf a skew-symmetric solution} \\
{\tiny\bf of the $\NYBE$ in $(P\otimes C, \ast_{\hat{\fd}})$}}
\ar@{<-}[ru]|-{\txt{\tiny\textcolor{red}{Pro.}~\ref{pro:DNYBE-NYBE}}}
\ar@{->}[dd]|-(0.75){\txt{\tiny\textcolor{red}{Pro.}~\ref{pro:o-nov}}}\\
&\txt{{\small $r^{\sharp}$}\\{\tiny an $\mathcal{O}$-operator of $(P, \diamond, \fd)$} \\
{\tiny associated to the coregular bimodule}}
\ar@{.>}[ld]|-{\txt{\tiny\textcolor{red}{$-\otimes\kappa^{\sharp}$}}} &
&\txt{{\small $r^{\sharp}$}\\
{\tiny an $\mathcal{O}$-operator of $(P, \dashv_{\fd}, \vdash_{\fd})$ associated} \\
{\tiny to the coregular representation}}  \ar@{<.}[ll]\\
\txt{{\small $\widehat{r}^{\sharp}=r^{\sharp}\otimes\kappa^{\sharp}$}\\{\tiny an $\mathcal{O}$-operator of $(P\otimes C, \cdot, \hat{\fd})$} \\
{\tiny associated to the coregular bimodule}}&
&\txt{{\small $\widehat{r}^{\sharp}=r^{\sharp}\otimes\kappa^{\sharp}$}\\
{\tiny an $\mathcal{O}$-operator of $(P\otimes C, \ast_{\hat{\fd}})$} \\
{\tiny associated to the coregular representation}}
\ar@{<-}[ru]|-{\txt{\tiny\textcolor{red}{$-\otimes\kappa^{\sharp}$}}}\ar@{<-}[ll]&}
\end{displaymath}
}

\smallskip\noindent
1.3. {\bf Outline of the paper.}
This paper is organized as follows. In Section \ref{sec:prel} we recall the notions
of differential perm algebras and bimodules over differential perm algebras.
The relationships between differential perm algebras, differential algebras,
diNovikov algebras and Novikov algebras are given in a commutative diagram.
In Section \ref{sec:bialg} we introduce the notion of a differential perm bialgebra,
which is equivalent to a Manin triple of differential perm algebras, is equivalent to
a matched pair of differential perm algebras (see Theorem \ref{thm:equ}).
In Section \ref{sec:quasi} the study of some special differential perm bialgebras leads
to the introduction of the differential perm Yang-Baxter equation ($\DPYBE$) in a
differential perm algebra. By using the solution of $\DPYBE$, we introduce the concept
of quasi-triangular (resp. triangular) differential perm bialgebras, and provide a
one-to-one correspondence between the symmetric solutions of $\DPYBE$ in a
differential perm algebra and the $\mathcal{O}$-operators on this differential perm algebra.
In Section \ref{sec:app} we discuss in detail the relationship between differential
perm bialgebras, differential infinitesimal bialgebras, diNovikov bialgebras and
Novikov bialgebras. We show that there is a differential infinitesimal bialgebra structure
on the tensor product of a differential perm bialgebra and a quadratic Zinbiel
algebra in Theorem \ref{thm:permbia-ASI}, there is aNovikov bialgebra structure
on the tensor product of a diNovikov bialgebra and a quadratic Zinbiel algebra in Theorem
\ref{thm:diNbia-Nbia}. Each differential infinitesimal bialgebra induces a Novikov
bialgebra and each differential perm bialgebra induces a diNovikov bialgebra (see
Theorems \ref{thm:ind-novbia} and \ref{thm:ind-dinovbia}). Therefore, we provide
two methods to construct a Novikov bialgebra from a differential perm bialgebra, and
we show that the Novikov bialgebras obtained by these two methods are consistent.
In particular, by considering the correspondence between the symmetric (or
skew-symmetric) solutions of the Yang-Baxter equation, the triangular bialgebra structures,
and the $\mathcal{O}$-operators, we present the three-dimensional commutative diagram as above.

Throughout this paper, we fix $\Bbbk$ as a field of characteristic zero.
All the vector spaces, algebras are over $\Bbbk$ and are finite-dimensional
unless otherwise specified, and all tensor products are also over $\Bbbk$.
We denote the identity map by $\id$. For any finite-dimensional $\Bbbk$-vector
space $V$, we denote $V^{\ast}$ the dual space.

\section{Differential perm algebras and bimodules} \label{sec:prel}
In this section, we recall the notions of differential perm algebras and bimodules over
differential perm algebras. The relationships between differential algebras,
differential perm algebras, Novikov algebras and diNovikov algebras are also given.

\begin{defi}\label{def:anti}
Let $P$ be a vector space with a binary operation $\diamond: P\otimes P
\rightarrow P$, $p_{1}\otimes p_{2}\mapsto p_{1}\diamond p_{2}$ satisfying the
(left commutative and associative) identities:
$$
(p_{1}\diamond p_{2})\diamond p_{3}=(p_{2}\diamond p_{1})\diamond p_{3}
=p_{1}\diamond(p_{2}\diamond p_{3}),
$$
for any $p_{1}, p_{2}, p_{3}\in P$. We call $(P, \diamond)$ a {\bf perm algebra}.
\end{defi}

Recall that an {\bf associative algebra} $(A, \cdot)$ is a vector space $A$ with a binary
operation $\cdot: A\otimes A\rightarrow A$, $a_{1}\otimes a_{2}\mapsto a_{1}a_{2}:=
a_{1}\cdot a_{2}$ satisfying $(a_{1}a_{2})a_{3}=a_{1}(a_{2}a_{3})$ for any $a_{1},
a_{2}, a_{3}\in A$. An associative algebra $(A, \cdot)$ is called {\bf commutative}
if for any $a_{1}, a_{2}\in A$, $a_{1}a_{2}=a_{2}a_{1}$. Any commutative associative
algebra is a perm algebra. A perm algebra with identity is nothing but a unital
commutative associative algebra. The perm algebra we consider in this paper does not
require identity element. Let $(P, \diamond)$ be a perm algebra (resp. commutative
associative algebra) and $\fd: P\rightarrow P$ be a linear map. If for any $p_{1},
p_{2}\in P$, $\fd(p_{1}\diamond p_{2})=p_{1}\diamond\fd(p_{2})+\fd(p_{1})\diamond p_{2}$,
we call that $\fd$ is a {\bf derivation} on this perm algebra (resp. commutative
associative algebra). A commutative associative algebra $(A, \cdot)$ with a derivation
$\fd$ is called a (commutative) {\bf differential algebra}. A perm algebra with a
derivation $(P, \diamond, \fd)$ is called a {\bf differential perm algebra}.
Let $(P, \diamond, \fd)$ and $(Q, \diamond', \fd')$ be two differential perm algebras.
A linear map $f: P\rightarrow Q$ is called a {\bf differential perm algebra homomorphism}
if for any $p_{1}, p_{2}\in P$, $f(p_{1}\diamond p_{2})=f(p_{1})\diamond f(p_{2})$ and
$\fd'\circ f=f\circ\fd$. A differential perm algebra homomorphism $f$ is said to be an
isomorphism if $f$ is a bijection.


Let $D$ be a vector space with two binary operations $\vdash$ and $\dashv$. If for any
$d_{1}, d_{2}, d_{3}\in D$, they satisfy the following equalities:
\begin{align*}
&\qquad\qquad\qquad d_{1}\dashv(d_{2}\dashv d_{3})=d_{1}\dashv(d_{2}\vdash d_{3}),\\
&\qquad\qquad\qquad (d_{1} \dashv d_{2})\dashv d_{3}=(d_{1}\dashv d_{3})\dashv d_{2},\\
&\qquad\quad (d_{1}\vdash d_{2})\dashv d_{3}=(d_{1}\vdash d_{3})\vdash d_{2}
=(d_{1}\dashv d_{3})\vdash d_{2},\\
& (d_{1}\dashv d_{2})\dashv d_{3}-d_{1}\dashv(d_{2}\dashv d_{3})
=(d_{2}\vdash d_{1})\dashv d_{3}-d_{2}\vdash(d_{1}\dashv d_{3}),\\
& (d_{1}\vdash d_{2})\vdash d_{3}-d_{1}\vdash(d_{2}\vdash d_{3})
=(d_{2}\vdash d_{1})\vdash d_{3}-d_{2}\vdash(d_{1}\vdash d_{3}),
\end{align*}
then $(D, \dashv, \vdash)$ is called a (left) {\bf diNovikov algebra}.
In \cite{KS} and \cite{XBH}, the diNovikov algebra is also called Novikov dialgebra.

\begin{pro}[\cite{KS}]\label{pro:dperm-diN}
Let $(P, \diamond, \fd)$ be a differential perm algebra. Define two binary operations
$\vdash_{\fd}, \dashv_{\fd}: P\otimes P\rightarrow P$ by
\begin{align}
p_{1}\vdash_{\fd}p_{2}:=p_{1}\diamond\fd(p_{2}),\qquad\qquad
p_{1}\dashv_{\fd}p_{2}:=\fd(p_{2})\diamond p_{1},\label{ind-dinov}
\end{align}
for any $p_{1}, p_{2}\in P$. Then $(P, \dashv_{\fd}, \vdash_{\fd})$ is a diNovikov algebra.
\end{pro}

Recall that a (left)
{\bf Novikov algebra} $(B, \ast)$ is a vector space $B$ with a binary operation
$\ast: B\otimes B\rightarrow B$ satisfying the following condition:
\begin{align*}
& \qquad\qquad (b_{1}\ast b_{2})\ast b_{3}=(b_{1}\ast b_{3})\ast b_{2},\\
& (b_{1}\ast b_{2})\ast b_{3}-b_{1}\ast (b_{2}\ast b_{3})
=(b_{2}\ast b_{1})\ast b_{3}-b_{2}\ast (b_{1}\ast b_{3}),
\end{align*}
for any $b_{1}, b_{2}, b_{3}\in B$. Given a differential algebra $(A, \cdot, \fd)$,
the binary operation
\begin{align}
a_{1}\ast_{\fd}a_{2}:=a_{1}\fd(a_{2}),  \label{ind-nov}
\end{align}
for any $a_{1}, a_{2}\in A$, defines a Novikov algebra $(A, \ast_{\fd})$, which is called
{\bf the Novikov algebra induced by the differential algebra $(A, \cdot, \fd)$}.
Moreover, commutative associative algebras and perm algebras, Novikov algebras and
diNovikov algebras are also closely related, and we can use Zinbiel algebra to
connect them together. Recall that a {\bf Zinbiel algebra} $(C, \star)$ is a vector
space $C$ together with a bilinear operation $\star: C\otimes C\rightarrow C$ such that
$$
c_{1}\star(c_{2}\star c_{3})=(c_{1}\star c_{2})\star c_{3}+(c_{2}\star c_{1})\star c_{3},
$$
for any $c_{1}, c_{2}, c_{3}\in C$.

\begin{pro}\label{pro:tensor}
Let $(P, \diamond, \fd)$ be a differential perm algebra, $(D, \dashv, \vdash)$ be a
diNovikov algebra and $(C, \star)$ be a Zinbiel algebra.
\begin{enumerate}\itemsep=0pt
\item[$(i)$] Define a binary operation $\cdot: (P\otimes C)\otimes(P\otimes C)
    \rightarrow(P\otimes C)$ by
$$
(p_{1}\otimes c_{1})\cdot(p_{2}\otimes c_{2})
:=(p_{1}\diamond p_{2})\otimes(c_{1}\star c_{2})
+(p_{2}\diamond p_{1})\otimes(c_{2}\star c_{1}),
$$
for any $p_{1}, p_{2}\in P$ and $c_{1}, c_{2}\in C$. Then $(P\otimes C, \cdot,
\hat{\fd})$ is a differential algebra, which is called {\bf the differential
algebra induced from $(P, \diamond, \fd)$ by $(C, \star)$},
where $\hat{\fd}=\fd\otimes\id$.
\item[$(ii)$] Define a binary operation $\ast: (D\otimes C)\otimes(D\otimes C)
    \rightarrow(D\otimes C)$ by
$$
(d_{1}\otimes c_{1})\ast(d_{2}\otimes c_{2})
:=(d_{1}\vdash d_{2})\otimes(c_{1}\star c_{2})
+(d_{1}\dashv d_{2})\otimes(c_{2}\star c_{1}),
$$
for any $d_{1}, d_{2}\in D$ and $c_{1}, c_{2}\in C$. Then $(D\otimes C, \ast)$
is a Novikov algebra, which is called {\bf the Novikov algebra induced from
$(D, \dashv, \vdash)$ by $(C, \star)$}.
\end{enumerate}
\end{pro}

\begin{proof}
$(i)$ First, by \cite[Proposition 2.10]{GH}, we get that $(P\otimes C, \cdot)$ is a
commutative associative algebra. Second, by direct calculation, for any $p_{1},
p_{2}\in P$ and $c_{1}, c_{2}\in C$, we have
\begin{align*}
&\; \hat{\fd}(p_{1}\otimes c_{1})\cdot(p_{2}\otimes c_{2})
+(p_{1}\otimes c_{1})\cdot\hat{\fd}(p_{2}\otimes c_{2})\\
=&\; (\fd(p_{1})\diamond p_{2}+p_{1}\diamond\fd(p_{2}))\otimes(c_{1}\star c_{2})
+(\fd(p_{2})\diamond p_{1}+p_{2}\diamond\fd(p_{1}))\otimes(c_{2}\star c_{1})\\
=&\;\hat{\fd}\big((p_{1}\otimes c_{1})\cdot(p_{2}\otimes c_{2})\big),
\end{align*}
since $\fd$ is a derivation on $(P, \diamond)$. That is, $\hat{\fd}$ is a derivation
on $(P\otimes C, \cdot)$, and so that $(P\otimes C, \cdot, \hat{\fd})$ is a
differential algebra.

$(ii)$ For any $d_{1}, d_{2}, d_{3}\in D$ and $c_{1}, c_{2}, c_{3}\in C$,
since $(D, \dashv, \vdash)$ is a diNovikov algebra, we have
\begin{align*}
&\; \big((d_{1}\otimes c_{1})\ast(d_{2}\otimes c_{2})\big)\ast(d_{3}\otimes c_{3})\\
=&\; ((d_{1}\vdash d_{2})\vdash d_{3})\otimes((c_{1}\star c_{2})\star c_{3})
+((d_{1}\vdash d_{2})\dashv d_{3})\otimes((c_{1}\star c_{3})\star c_{2})\\[-1mm]
&\ \ +((d_{1}\vdash d_{2})\dashv d_{3})\otimes((c_{3}\star c_{1})\star c_{2})
+((d_{1}\dashv d_{2})\vdash d_{3})\otimes((c_{2}\star c_{1})\star c_{3})\\[-1mm]
&\ \ +((d_{1}\dashv d_{2})\dashv d_{3})\otimes((c_{3}\star c_{2})\star c_{1})
+((d_{1}\dashv d_{2})\dashv d_{3})\otimes((c_{2}\star c_{3})\star c_{1})\\
=&\; ((d_{1}\vdash d_{3})\vdash d_{2})\otimes((c_{1}\star c_{3})\star c_{2})
+((d_{1}\vdash d_{3})\dashv d_{2})\otimes((c_{2}\star c_{1})\star c_{3})\\[-1mm]
&\ \ +((d_{1}\vdash d_{3})\dashv d_{2})\otimes((c_{1}\star c_{2})\star c_{3})
+((d_{1}\dashv d_{3})\vdash d_{2})\otimes((c_{3}\star c_{1})\star c_{2})\\[-1mm]
&\ \ +((d_{1}\dashv d_{3})\dashv d_{2})\otimes((c_{2}\star c_{3})\star c_{1})
+((d_{1}\dashv d_{3})\dashv d_{2})\otimes((c_{3}\star c_{2})\star c_{1})\\
=&\; \big((d_{1}\otimes c_{1})\ast(d_{3}\otimes c_{3})\big)\ast(d_{2}\otimes c_{2}).
\end{align*}
Similarly, $((d_{1}\otimes c_{1})\ast(d_{2}\otimes c_{2}))\ast(d_{3}\otimes c_{3})
-(d_{1}\otimes c_{1})\ast((d_{2}\otimes c_{2})\ast(d_{3}\otimes c_{3}))
=((d_{2}\otimes c_{2})\ast(d_{1}\otimes c_{1}))\ast(d_{3}\otimes c_{3})
-(d_{2}\otimes c_{2})\ast((d_{1}\otimes c_{1})\ast(d_{3}\otimes c_{3}))$.
Thus, $(D\otimes C, \ast)$ is a Novikov algebra.
\end{proof}

Let $(P, \diamond, \fd)$ be a differential perm algebra, $(D, \dashv, \vdash)$ be a
diNovikov algebra and $(C, \star)$ be a Zinbiel algebra. By the differential algebra
$(P\otimes C, \cdot, \hat{\fd})$, we obtain a Novikov algebra $(P\otimes C,
\ast_{\hat{\fd}})$, where
$$
(p_{1}\otimes c_{1})\ast_{\hat{\fd}}(p_{2}\otimes c_{2})
=(p_{1}\diamond\fd(p_{2}))\otimes(c_{1}\star c_{2})
+(\fd(p_{2})\diamond p_{1})\otimes(c_{2}\star c_{1}),
$$
for any $p_{1}, p_{2}\in P$ and $c_{1}, c_{2}\in C$.
On the other hand, the Novikov algebra structure on the tensor product of $(D, \dashv_{\fd},
\vdash_{\fd})$ and $(C, \star)$ is given by
$$
(p_{1}\otimes c_{1})\ast(p_{2}\otimes c_{2})
=(p_{1}\diamond\fd(p_{2}))\otimes(c_{1}\star c_{2})
+(\fd(p_{2})\diamond p_{1})\otimes(c_{2}\star c_{1}).
$$
That is, $(P\otimes C, \ast)=(P\otimes C, \ast_{\hat{\fd}})$ as Novikov algebras,
and so that, we obtain the following commutative diagram:
$$
\xymatrix@C=2cm@R=0.5cm{
\txt{$(P, \diamond, \fd)$ \\ {\tiny a differential perm algebra}}\ar[d]\ar[r]
&\txt{$(P, \dashv_{\fd}, \vdash_{\fd})$\\ {\tiny a diNovikov algebra}}\ar[d] \\
\txt{$(P\otimes C, \cdot, \hat{\fd})$ \\ {\tiny a differential algebra}}\ar[r]
& \txt{$(P\otimes C, \ast_{\hat{\fd}})$ \\ {\tiny a Novikov algebra}}}
$$
One of the main conclusions of this paper is that this commutative diagram also holds
at the bialgebra level.

Let $(C, \star)$ be a Zinbiel algebra. Recall that a bimodule over $(C, \star)$
is a vector space $V$ with two linear maps $\tilde{\kl}, \tilde{\kr}:
C\rightarrow\gl(V)$ such that for any $c_{1}, c_{2}\in C$,
$$
\tilde{\kl}(c_{2})\circ\tilde{\kl}(c_{1})=\tilde{\kl}(c_{2}\star c_{1})
+\tilde{\kl}(c_{1}\star c_{2}),\qquad\qquad
\tilde{\kr}(c_{1}\star c_{2})=\tilde{\kr}(c_{2})\circ\tilde{\kr}(c_{1})
+\tilde{\kr}(c_{2})\circ\tilde{\kl}(c_{1})=\tilde{\kl}(c_{1})\circ\tilde{\kr}(c_{2}).
$$
In particular, if we define $\tilde{\fl}_{C}, \tilde{\fr}_{C}: C\rightarrow\gl(C)$ by
$\tilde{\fl}_{C}(c_{1})(c_{2})=c_{1}\star c_{2}=\tilde{\fr}_{C}(c_{2})(c_{1})$ for any
$c_{1}, c_{2}\in C$, then $(C, \tilde{\fl}_{C}, \tilde{\fr}_{C})$ is a bimodule over
$(C, \star)$. Next we introduce the definition of bimodule over a differential perm algebra.
Let $(P, \diamond)$ be a perm algebra. Recall that a {\bf bimodule $(V, \kl, \kr)$ over
perm algebra $(P, \diamond)$} is a vector space $V$ with two linear maps $\kl, \kr:
P\rightarrow\gl(V)$ such that for any $p_{1}, p_{2}\in p$,
$$
\kl(p_{1}\diamond p_{2})=\kl(p_{1})\circ\kl(p_{2})=\kl(p_{2})\circ\kl(p_{1}),\qquad\qquad
\kr(p_{1}\diamond p_{2})=\kr(p_{2})\circ\kr(p_{1})=\kr(p_{2})\circ\kl(p_{1})
=\kl(p_{1})\circ\kr(p_{2}).
$$

\begin{defi}\label{def:anti-mod}
Let $(P, \diamond, \fd)$ be a differential perm algebra and $V$ be a vector space.
If there exist linear maps $\kl, \kr: P\rightarrow\gl(V)$ and $\partial: V\rightarrow V$,
such that $(V, \kl, \kr)$ is a bimodule over perm algebra $(P, \diamond)$ and
$$
\partial(\kl(p)(v))=\kl(\fd(p))(v)+\kl(p)(\partial(v)),\qquad\quad
\partial(\kr(p)(v))=\kr(\fd(p))(v)+\kr(p)(\partial(v)),
$$
for any $p\in P$, $v\in V$. We call that $(V, \kl, \kr, \partial)$ is a
{\bf bimodule over $(P, \diamond, \fd)$}.
\end{defi}

By direct calculations, we can give an equivalent condition as follows.

\begin{pro}\label{pro:rep}
Let $(P, \diamond, \fd)$ be a differential perm algebra, $V$ be a vector space,
$\kl, \kr: P\rightarrow\gl(V)$ and $\partial: V\rightarrow V$ be linear maps.
Then $(V, \kl, \kr, \partial)$ is a bimodule over $(P, \diamond, \fd)$ if and only if
$(P\oplus V, \fd\oplus\partial)$ is a differential perm algebra
under the following operations:
$$
(p_{1}, v_{1})(p_{2}, v_{2})=\big(p_{1}\diamond p_{2},\ \
\kl(p_{1})(v_{2})+\kr(p_{2})(v_{1})\big),
$$
for all $p_{1}, p_{2}\in P$ and $v_{1}, v_{2}\in V$. This differential perm algebra is
called the {\bf semidirect product} of $(P, \diamond, \fd)$ by bimodule $(V, \kl, \kr,
\partial)$, denoted by $P\ltimes V$.
\end{pro}

For a differential perm algebra $(P, \diamond, \fd)$, we get linear maps $\fl_{P}, \fr_{P}:
P\rightarrow\gl(P)$, $\fl_{P}(p_{1})(p_{2})=p_{1}\diamond p_{2}=\fr_{P}(p_{2})(p_{1})$
for any $p_{1}, p_{2}\in P$. It is easy to see that $(P, \fl_{P}, \fr_{P}, \fd)$
is a bimodule over $(P, \diamond, \fd)$. We call this bimodule the {\bf regular bimodule}
of $(P, \diamond, \fd)$. Let $(V, \kl, \kr, \partial)$ and $(V', \kl', \kr', \partial')$
be two bimodules over a differential perm algebra $(P, \diamond, \fd)$. A linear map
$f: V\rightarrow V'$ is called a {\bf bimodule homomorphism} over $(P, \diamond, \fd)$
if for any $p\in P$, $f\circ\kl(p)=\kl'(p)\circ f$, $f\circ\kr(p)=\kr'(p)\circ f$ and
$\partial'\circ f=f\circ\partial$. If the homomorphism $f$ is a bijection, we call that
$f$ is an {\bf isomorphism} and the bimodules $(V, \kl, \kr, \partial)$ and $(V', \kl',
\kr', \partial')$ are isomorphic.

Let $V$ be a vector space. Denote the standard pairing between the dual space
$V^{\ast}$ and $V$ by
$$
\langle-,-\rangle:\quad V^{\ast}\otimes V\rightarrow \Bbbk, \qquad\quad
\langle \xi,\; v \rangle:=\xi(v),
$$
for any $\xi\in V^{\ast}$ and $v\in V$. Let $V$, $W$ be two vector spaces. For a linear
map $\varphi: V\rightarrow W$, the transpose map $\varphi^{\ast}: W^{\ast}\rightarrow
V^{\ast}$ is defined by
$$
\langle \varphi^{\ast}(\xi),\; v \rangle:=\langle\xi,\; \varphi(v)\rangle,
$$
for any $v\in V$ and $\xi\in W^{\ast}$. Let $(P, \diamond)$ be a perm algebra
and $V$ be a vector space. For a linear map $\psi: P\rightarrow\gl(V)$, the linear map
$\psi^{\ast}: P\rightarrow\gl(V^{\ast})$ is defined by
$$
\langle\psi^{\ast}(p)(\xi),\; v\rangle:=-\langle\xi,\; \psi(p)(v)\rangle,
$$
for any $p\in P$, $v\in V$, $\xi\in V^{\ast}$. That is, $\psi^{\ast}(p)=-\psi(p)^{\ast}$
for all $p\in P$. Let $(P, \diamond)$ be a perm algebra and $(V, \kl, \kr)$ be a
bimodule over $(P, \diamond)$. Then $(V^{\ast}, -\kl^{\ast}, \kr^{\ast}-\kl^{\ast})$ is
also a bimodule over $(P, \diamond)$, which is called the {\bf dual bimodule} of
$(V, \kl, \kr)$.

\begin{pro}\label{pro:dualmod}
Let $(P, \diamond, \fd)$ be a differential perm algebra, $(V, \kl, \kr)$ be a
bimodule over $(P, \diamond)$ and $\partial: V\rightarrow V$ be a linear map.
Then the quadruple $(V^{\ast}, -\kl^{\ast}, \kr^{\ast}-\kl^{\ast}, \partial^{\ast})$
is a bimodule over the differential perm algebra $(P, \diamond, \fd)$ if and only if
the following equations hold:
$$
\partial(\kl(p)(v))=\kl(p)(\partial(v))-\kl(\fd(p))(v), \qquad\qquad
\partial(\kr(p)(v))=\kr(p)(\partial(v))-\kr(\fd(p))(v),
$$
for any $p\in P$ and $v\in V$. Therefore, we get that $(V^{\ast}, -\kl^{\ast},
\kr^{\ast}-\kl^{\ast}, \partial^{\ast})$ is a bimodule over $(P, \diamond, \fd)$
if $(V, \kl, \kr, -\partial)$ is a bimodule over $(P, \diamond, \fd)$.
\end{pro}

\begin{proof}
Since for any $p\in P$, $v\in V$ and $\xi\in V^{\ast}$,
\begin{align*}
&\qquad\qquad\qquad\qquad\langle\partial^{\ast}(-\kl^{\ast}(p)(\xi)),\; v\rangle
=\langle\xi,\; \kl(p)(\partial(v))\rangle,\\
&\langle-\kl^{\ast}(\fd(p))(\xi),\; v\rangle
=\langle\xi,\; \kl(\fd(p))(v)\rangle,\qquad\quad
\langle-\kl^{\ast}(p)(\partial^{\ast}(\xi)),\; v\rangle
=\langle\xi,\; \partial(\kl(p)(v))\rangle,
\end{align*}
we get $\partial^{\ast}(-\kl^{\ast}(p)(\xi))=-\kl^{\ast}(\fd(p))(\xi)
-\kl^{\ast}(p)(\partial^{\ast}(\xi))$ if and only if
$\partial(\kl(p)(v))=\kl(p)(\partial(v))-\kl(\fd(p))(v)$.
Similarly, we have $\partial^{\ast}((\kr^{\ast}-\kl^{\ast})(p)(\xi))=
(\kr^{\ast}-\kl^{\ast})(\fd(p))(\xi)+(\kr^{\ast}-\kl^{\ast})(p)(\partial^{\ast}(\xi))$
if and only if $\partial(\kr(p)(v))=\kr(p)(\partial(v))-\kr(\fd(p))(v)$.
Thus, we get this proposition.
\end{proof}

In particular, we consider the regular bimodule $(P, \fl_{P}, \fr_{P}, \fd)$ over
a differential perm algebra $(P, \diamond, \fd)$. Then $(P^{\ast}, -\fl_{P}^{\ast},
\fr_{P}^{\ast}-\fl_{P}^{\ast}, -\fd^{\ast})$ is also a bimodule over $(P, \diamond, \fd)$,
which is called the {\bf coregular bimodule}. If there is a linear map $\partial:
P\rightarrow P$ such that $(P, \fl_{P}, \fr_{P}, -\partial)$ is a bimodule over
$(P, \diamond, \fd)$, $\partial$ is said to be {\bf admissible to $(P, \diamond, \fd)$}
and $(P, \diamond, \fd, \partial)$ is called {\bf an admissible differential
perm algebra}. Let $(P, \diamond, \fd)$ be a differential perm algebra and $\partial:
P\rightarrow P$ be a linear map. Then $(P, \diamond, \fd, \partial)$ is an admissible
differential perm algebra if and only if for any $p_{1}, p_{2}\in P$,
$$
\partial(p_{1}\diamond p_{2})=p_{1}\diamond\partial(p_{2})-\fd(p_{1})\diamond p_{2}
=\partial(p_{1})\diamond p_{2}-p_{1}\diamond\fd(p_{2}).
$$
In particular, $(P, \diamond, \fd, -\fd)$ is an admissible differential perm algebra.

\section{Differential perm bialgebras}\label{sec:bialg}
In this section, we introduce the notion of the differential perm bialgebras,
and give their equivalence in terms of matched pairs and Manin triples of averaging algebras.
We first recall the concept of a matched pair of perm algebras.

\begin{defi}[\cite{Hou,LZB}] \label{def:mat-perm}
A {\bf matched pair of perm algebras} consists of two perm algebras $(P, \diamond)$,
$(Q, \diamond')$, and linear maps $\kl_{P}, \kr_{P}: P\rightarrow\gl(Q)$
and $\kl_{Q}, \kr_{Q}: Q\rightarrow\gl(P)$, such that $(P\oplus Q, \bar{\diamond})$ is
also a perm algebra, where $\bar{\diamond}$ is defined by
\begin{align*}
(p_{1}, q_{1})\bar{\diamond}(p_{2}, q_{2})=\Big(p_{1}\diamond p_{2}+\kl_{Q}(q_{1})(p_{2})
+\kr_{Q}(q_{2})(p_{1}),\quad  q_{1}\diamond'q_{2}+\kl_{P}(p_{1})(q_{2})
+\kr_{P}(p_{2})(q_{1})\Big),
\end{align*}
for all $p_{1}, p_{2}\in P$ and $q_{1}, q_{2}\in Q$. The matched pair is denoted by
$(P, Q, \kl_{P}, \kr_{P}, \kl_{Q}, \kr_{Q})$ and the resulting algebra is denoted
by $P\bowtie Q:=(P\oplus Q, \bar{\diamond})$.
\end{defi}

For a matched pair of perm algebras $(P, Q, \kl_{P}, \kr_{P}, \kl_{Q}, \kr_{Q})$,
it is easy to see that $(P, \kl_{Q}, \kr_{Q})$ is a bimodule over $(Q, \diamond')$ and
$(Q, \kl_{P}, \kr_{P})$ is a bimodule over $(P, \diamond)$.

\begin{defi}\label{def:mat-diperm}
Let $(P, \diamond, \fd_{P})$ and $(Q, \diamond', \fd_{Q})$ be two differential perm algebras.
Suppose that $\kl_{P}, \kr_{P}: P\rightarrow\gl(Q)$ and $\kl_{Q}, \kr_{Q}:
Q\rightarrow\gl(P)$ are linear maps. If the following conditions are satisfied:
\begin{enumerate}\itemsep=0pt
\item[$(i)$] $(P, \kl_{Q}, \kr_{Q}, \fd_{P})$ is a bimodule over $(Q, \diamond', \fd_{Q})$;
\item[$(ii)$] $(Q, \kl_{P}, \kr_{P}, \fd_{Q})$ is a bimodule over $(P, \diamond, \fd_{P})$;
\item[$(iii)$] $(P, Q, \kl_{P}, \kr_{P}, \kl_{Q}, \kr_{Q})$ is a matched pair of
     perm algebras,
\end{enumerate}
then $((P, \fd_{P}), (Q, \fd_{Q}), \kl_{P}, \kr_{P}, \kl_{Q}, \kr_{Q})$ is called a
{\bf matched pair of differential perm algebras}.
\end{defi}

\begin{pro}\label{pro:matda}
Let $(P, \diamond, \fd_{P})$ and $(Q, \diamond', \fd_{Q})$ be two differential perm algebras.
Suppose that $(P, Q, \kl_{P}$, $\kr_{P}, \kl_{Q}, \kr_{Q})$ is a matched pair
of perm algebras. Then $(P\bowtie Q, \fd_{P}\oplus\fd_{Q})$ is a differential perm
algebra if and only if $((P, \fd_{P}), (Q, \fd_{Q}), \kl_{P}, \kr_{P}, \kl_{Q}, \kr_{Q})$
is a matched pair of differential perm algebras.
\end{pro}

\begin{proof}
It follows from a straightforward verification.
\end{proof}

We next recall the concept of the Manin triple of perm algebra.

\begin{defi}\label{def:form}
Let $V$ be a vector space and $\mathfrak{B}(-,-): V\otimes V\rightarrow\Bbbk$ be a
bilinear form on $V$.
\begin{enumerate}\itemsep=0pt
\item[-] $\mathfrak{B}(-,-)$ is called {\bf nondegenerate} if $\mathfrak{B}(v_{1},\;
        v_{2})=0$ for any $v_{2}\in V$, then $v_{1}=0$;
\item[-] $\mathfrak{B}(-,-)$ is called {\bf skew-symmetric} if $\mathfrak{B}(v_{1},\; v_{2})
        =-\mathfrak{B}(v_{2},\; v_{1})$, for any $v_{1}, v_{2}\in V$;
\item[-] $\mathfrak{B}(-,-)$ is called {\bf symmetric} if $\mathfrak{B}(v_{1},\; v_{2})
        =\mathfrak{B}(v_{2},\; v_{1})$, for any $v_{1}, v_{2}\in V$;
\end{enumerate}
A bilinear form $\mathfrak{B}(-,-)$ on a perm algebra $(P, \diamond)$ is called
{\bf invariant} if
$$
\mathfrak{B}(p_{1}\diamond p_{2},\; p_{3})=\mathfrak{B}(p_{1},\; p_{2}\diamond p_{3}
-p_{3}\diamond p_{2}),
$$
for any $p_{1}, p_{2}, p_{3}\in P$. A {\bf quadratic perm algebra} $(P, \diamond,
\mathfrak{B})$ is a perm algebra $(P, \diamond)$ with a nondegenerate skew-symmetric
invariant bilinear form $\mathfrak{B}(-,-)$.
\end{defi}

Let $(P, \diamond)$ be a perm algebra. Suppose that there is a perm algebra structure
$\diamond'$ on its dual space $P^{\ast}$ and a perm algebra structure $\bar{\diamond}$
on the direct sum $P\oplus P^{\ast}$ of the underlying vector spaces $P$ and $P^{\ast}$,
which contains both $(P, \diamond)$ and $(P^{\ast}, \diamond')$ as subalgebras.
Define a bilinear form on $P\oplus P^{\ast}$ by
\begin{align*}
\mathfrak{B}_{d}\big((p_{1}, \xi_{1}),\; (p_{2}, \xi_{2})\big)
:=\langle\xi_{2},\; p_{1}\rangle-\langle\xi_{1},\; p_{2}\rangle,
\end{align*}
for any $p_{1}, p_{2}\in P$ and $\xi_{1}, \xi_{2}\in P^{\ast}$.
If $(P\oplus P^{\ast}, \bar{\diamond}, \mathfrak{B}_{d})$ is a quadratic perm algebra,
then the triple $((P\oplus P^{\ast}, \mathfrak{B}_{d}), (P, \diamond), (P^{\ast},
\diamond'))$ is called a (standard) {\bf Manin triple} of $(P, \diamond)$ and
$(P^{\ast}, \diamond')$. 

\begin{defi}\label{def:dif-qua}
A {\bf quadratic differential perm algebra} is a quadruple $(P, \diamond, \fd, \mathfrak{B})$,
where $(P, \diamond, \fd)$ is a differential perm algebra and $\mathfrak{B}(-,-)$ is a
nondegenerate skew-symmetric invariant bilinear form on $(P, \diamond)$.
A linear map $\widetilde{\fd}: P\rightarrow P$ is called {\bf the adjoint linear operator
of $\fd$} under the nondegenerate bilinear form $\mathfrak{B}(-,-)$, if for
any $p_{1}, p_{2}\in P$,
$$
\mathfrak{B}(\fd(p_{1}),\; p_{2})=\mathfrak{B}(p_{1},\; \widetilde{\fd}(p_{2})).
$$
\end{defi}

For the adjoint linear operator, we have the following conclusion.

\begin{pro}\label{pro:difperm-mod}
Let $(P, \diamond, \fd, \mathfrak{B})$ be a quadratic differential perm algebra,
and $\widetilde{\fd}$ be the adjoint of $\fd$ with respect to $\mathfrak{B}(-,-)$.
Then, $(P^{\ast}, -\fl_{P}^{\ast}, \fr_{P}^{\ast}-\fl_{P}^{\ast}, \widetilde{\fd}^{\ast})$
is a bimodule over the differential perm algebra $(P, \diamond, \fd)$, and as bimodules
over $(P, \diamond, \fd)$, $(P, \fl_{P}, \fr_{P}, \fd)$ and $(P^{\ast}, -\fl_{P}^{\ast},
\fr_{P}^{\ast}-\fl_{P}^{\ast}, \widetilde{\fd}^{\ast})$ are isomorphic.
Moreover, let $(P, \diamond, \fd)$ be a differential perm algebra, and
$\partial: P\rightarrow P$ be a linear map. If $(P^{\ast}, -\fl_{P}^{\ast},
\fr_{P}^{\ast}-\fl_{P}^{\ast}, \partial^{\ast})$ is a bimodule over $(P, \diamond, \fd)$,
and it is isomorphic to the regular bimodule $(P, \fl_{P}, \fr_{P}, \fd)$, then
there exists a nondegenerate invariant bilinear form $\mathfrak{B}(-,-)$ on
$(P, \diamond)$ such that $\partial=\widetilde{\fd}$.
\end{pro}

\begin{proof}
First, suppose that $(P, \diamond, \fd, \mathfrak{B})$ is a quadratic differential
perm algebra. For any $p_{1}, p_{2}, p_{3}\in P$, since
\begin{align*}
\mathfrak{B}(-\widetilde{\fd}(\fl_{P}(p_{1})(p_{2})),\; p_{3})
&=-\mathfrak{B}(\fl_{P}(p_{1})(p_{2}),\; \fd(p_{3}))
=-\mathfrak{B}(p_{2},\; p_{1}\diamond\fd(p_{3})),\\
\mathfrak{B}(\fl_{P}(\fd(p_{1}))(p_{2}),\; p_{3})
&=\mathfrak{B}(p_{2},\; \fd(p_{1})\diamond p_{3}),\\
\mathfrak{B}(\fl_{P}(p_{1})(-\widetilde{\fd}(p_{2})),\; p_{3})
&=-\mathfrak{B}(\widetilde{\fd}(p_{2}),\; p_{1}\diamond p_{3})
=-\mathfrak{B}(p_{2},\; \fd(p_{1}\diamond p_{3})),
\end{align*}
and $\fd$ is a derivation, we get $-\widetilde{\fd}(\fl_{P}(p_{1})(p_{2}))=
\fl_{P}(\fd(p_{1}))(p_{2})+\fl_{P}(p_{1})(-\widetilde{\fd}(p_{2}))$. Similarly,
we also have $-\widetilde{\fd}(\fr_{P}(p_{1})(p_{2}))=\fr_{P}(\fd(p_{1}))(p_{2})
+\fr_{P}(p_{1})(-\widetilde{\fd}(p_{2}))$. Thus, $(P, \fl_{P}, \fr_{P},
-\widetilde{\fd})$ is a bimodule over $(P, \diamond, \fd)$. By Proposition \ref{pro:dualmod},
we get that $(P^{\ast}, -\fl_{P}^{\ast}, \fr_{P}^{\ast}-\fl_{P}^{\ast},
\widetilde{\fd}^{\ast})$ is a bimodule over $(P, \diamond, \fd)$.

Define a linear map $\varphi: P\rightarrow P^{\ast}$ by
$\varphi(p_{1})(p_{2})=\mathfrak{B}(p_{1}, p_{2})$ for any $p_{1}, p_{2}\in P$.
Then, $\varphi$ is a linear isomorphism and $\varphi\circ\fd=\widetilde{\fd}^{\ast}
\circ\varphi$. Moreover, for any $p_{1}, p_{2}, p_{3}\in P$, we have
\begin{align*}
\langle\varphi(\fl_{P}(p_{1})(p_{2})),\; p_{3}\rangle
&=\mathfrak{B}(p_{1}\diamond p_{2},\; p_{3})
=\langle\varphi(p_{2}),\; p_{1}\diamond p_{3}\rangle
=-\langle\fl_{P}^{\ast}(p_{1})(\varphi(p_{2})),\; p_{3}\rangle, \\
\langle\varphi(\fr_{P}(p_{1})(p_{2})),\; p_{3}\rangle
&=\mathfrak{B}(p_{2}\diamond p_{1},\; p_{3})
=\langle\varphi(p_{2}),\; p_{1}\diamond p_{3}-p_{3}\diamond p_{1}\rangle
=-\langle(\fl_{P}^{\ast}-\fr_{P}^{\ast})(p_{1})(\varphi(p_{2})),\; p_{3}\rangle.
\end{align*}
Hence, $\varphi$ is an isomorphism of bimodules over $(P, \diamond, \fd)$.

Second, suppose that $\varphi: P\rightarrow P^{\ast}$ is the isomorphism
from $(P, \fl_{P}, \fr_{P}, \alpha)$ to $(P^{\ast}, -\fl_{P}^{\ast},
\fr_{P}^{\ast}-\fl_{P}^{\ast}, \partial^{\ast})$.
Define a bilinear form $\mathfrak{B}(-,-)$ on $P$ by $\mathfrak{B}(p_{1},
p_{2}):=\langle\varphi(p_{1}),\; p_{2}\rangle$ for any $p_{1}, p_{2}\in P$.
Then by a similar argument as above, we show that $\mathfrak{B}(-,-)$ is a nondegenerate
invariant bilinear form on $(P, \diamond)$ such that $\partial=\widetilde{\fd}$.
\end{proof}

We now give the notion of Manin triple of differential perm algebras.

\begin{defi}\label{def:doucon}
Let $(P, \diamond, \fd)$ and $((P^{\ast}, \diamond', \fd'))$ be two differential perm
algebras and $P\bowtie P^{\ast}=(P\oplus P^{\ast}, \bar{\diamond})$ be the perm algebra
given in Definition \ref{def:mat-perm}. If $(P\bowtie P^{\ast}, \fd\oplus\fd',
\mathfrak{B}_{d})$ is a quadratic differential perm algebra, then the triple
$((P\bowtie P^{\ast}, \fd\oplus\fd', \mathfrak{B}_{d}), P, P^{\ast})$ is called a
(standard) {\bf Manin triple} of $(P, \diamond, \fd)$ and $(P^{\ast}, \diamond', \fd')$.
\end{defi}

\begin{lem}\label{lem:douadm}
Let $((P\bowtie P^{\ast}, \fd\oplus\partial^{\ast}, \mathfrak{B}_{d}), P, P^{\ast})$ be a
Manin triple of differential perm algebras $(P, \diamond, \fd)$ and $(P^{\ast}, \diamond',
\partial^{\ast}))$. Then,
\begin{enumerate}\itemsep=0pt
\item[$(i)$] The adjoint $\widetilde{\fd\oplus\partial^{\ast}}$ of $\fd\oplus\partial^{\ast}$
     with respect to $\mathfrak{B}_{d}(-,-)$ is $\partial\oplus\fd^{\ast}$, and $(P\oplus
     P^{\ast}, \fl_{P\bowtie P^{\ast}}, \fr_{P\bowtie P^{\ast}}, -\partial\oplus\fd^{\ast})$
     is a bimodule over $(P\bowtie P^{\ast}, \fd\oplus\partial^{\ast})$;
\item[$(ii)$] $(P, \diamond, \fd, \partial)$ and $(P^{\ast}, \diamond', \partial^{\ast},
     \fd^{\ast})$ are admissible.
\end{enumerate}
\end{lem}

\begin{proof}
$(i)$ For any $p_{1}, p_{2}\in P$ and $\xi_{1}, \xi_{2}\in P^{\ast}$, we have
\begin{align*}
\mathfrak{B}_{d}((\fd\oplus\partial^{\ast})(p_{1}, \xi_{1}),\; (p_{2}, \xi_{2}))
&=\langle\xi_{2},\; \fd(p_{1})\rangle-\langle\partial^{\ast}(\xi_{1}),\; p_{2}\rangle\\
&=\langle\fd^{\ast}(\xi_{2}),\; p_{1}\rangle-\langle\xi_{1},\; \partial(p_{2})\rangle
=\mathfrak{B}_{d}((p_{1}, \xi_{1}),\; (\partial\oplus\fd^{\ast})(p_{2}, \xi_{2})).
\end{align*}
That is to say, the adjoint of $\fd\oplus\partial^{\ast}$ with respect to
$\mathfrak{B}_{d}(-,-)$ is $\partial\oplus\fd^{\ast}$. Moreover, by the proof of
Proposition \ref{pro:difperm-mod}, we get $(P\oplus P^{\ast}, \fl_{P\bowtie P^{\ast}},
\fr_{P\bowtie P^{\ast}}, -(\partial\oplus\fd^{\ast}))$ is a bimodule over
$(P\bowtie P^{\ast}, \fd\oplus\partial^{\ast})$.

$(ii)$ Since $(P\oplus P^{\ast}, \fl_{P\bowtie P^{\ast}}, \fr_{P\bowtie P^{\ast}},
-\partial\oplus\fd^{\ast})$ is a bimodule over $(P\bowtie P^{\ast}, \fd\oplus
\partial^{\ast})$, we get for any $p_{1}, p_{2}\in P$ and $\xi_{1}, \xi_{2}\in P^{\ast}$,
we have
\begin{align*}
-(\partial\oplus\fd^{\ast})(\fl_{P\bowtie P^{\ast}}(p_{1}, \xi_{1})(p_{2}, \xi_{2}))
&=\fl_{P\bowtie P^{\ast}}((\fd\oplus\partial^{\ast})(p_{1}, \xi_{1}))(p_{2}, \xi_{2})
-\fl_{P\bowtie P^{\ast}}(p_{1}, \xi_{1})((\partial\oplus\fd^{\ast})(p_{2}, \xi_{1})),\\
-(\partial\oplus\fd^{\ast})(\fr_{P\bowtie P^{\ast}}(p_{1}, \xi_{1})(p_{2}, \xi_{2}))
&=\fr_{P\bowtie P^{\ast}}((\fd\oplus\partial^{\ast})(p_{1}, \xi_{1}))(p_{2}, \xi_{2})
-\fr_{P\bowtie P^{\ast}}(p_{1}, \xi_{1})((\partial\oplus\fd^{\ast})(p_{2}, \xi_{1})).
\end{align*}
Taking $\xi_{1}=\xi_{2}=0$ in the above equations, we get that $(P, \fl_{P}, \fr_{P},
-\partial)$ is a bimodule over $(P, \diamond, \fd)$, and taking $p_{1}=p_{2}=0$, we get
that $(P^{\ast}, \fl_{P^{\ast}}, \fr_{P^{\ast}}, -\fd^{\ast})$ is a bimodule over
$(P^{\ast}, \diamond', \partial^{\ast})$.
\end{proof}

\begin{pro}\label{pro:dou-mat}
Let $(P, \diamond, \fd)$ be a differential perm algebra and $\partial: P\rightarrow P$
be a linear map. Suppose that there is a bilinear map $\diamond': P^{\ast}\otimes P^{\ast}
\rightarrow P^{\ast}$ such that $(P^{\ast}, \diamond', \partial^{\ast})$ is a differential
perm algebra. Then $((P\bowtie P^{\ast}, \fd\oplus\partial^{\ast}, \mathfrak{B}_{d}), P,
P^{\ast})$ is a Manin triple of differential perm algebras $(P, \diamond, \fd)$ and
$(P^{\ast}, \diamond', \partial^{\ast})$ if and only if $((P, \fd), (P^{\ast},
\partial^{\ast}), -\fl_{P}^{\ast}, \fr_{P}^{\ast}-\fl_{P}^{\ast}, -\fl_{P^{\ast}}^{\ast},
\fr_{P^{\ast}}^{\ast}-\fl_{P^{\ast}}^{\ast})$ is a matched pair of differential perm algebras.
\end{pro}

\begin{proof}
By \cite{Hou,LZB}, we get $(P, P^{\ast}, -\fl_{P}^{\ast}, \fr_{P}^{\ast}-\fl_{P}^{\ast},
-\fl_{P^{\ast}}^{\ast}, \fr_{P^{\ast}}^{\ast}-\fl_{P^{\ast}}^{\ast})$ is a matched
pair of perm algebras if and only if $((P\bowtie P^{\ast}, \mathfrak{B}_{d}), P,
P^{\ast})$ is a Manin triple of perm algebras. Now, if $((P\bowtie P^{\ast},
\fd\oplus\partial^{\ast}, \mathfrak{B}_{d}), P, P^{\ast})$ is a Manin triple of
differential perm algebras, by Lemma \ref{lem:douadm}, $(P, \fl_{P}, \fr_{P}, -\partial)$
is a bimodule over $(P, \diamond, \fd)$, and so that $(P^{\ast}, -\fl_{P}^{\ast},
\fr_{P}^{\ast}-\fl_{P}^{\ast}, \partial^{\ast})$ is a bimodule over $(P, \diamond, \fd)$.
Similarly, $(P, -\fl_{P^{\ast}}^{\ast}, \fr_{P^{\ast}}^{\ast}-\fl_{P^{\ast}}^{\ast},
\fd)$ is a bimodule over $(P^{\ast}, \diamond', \partial^{\ast})$. Thus, $((P, \fd),
(P^{\ast}, \partial^{\ast}), -\fl_{P}^{\ast}, \fr_{P}^{\ast}-\fl_{P}^{\ast},
-\fl_{P^{\ast}}^{\ast}, \fr_{P^{\ast}}^{\ast}-\fl_{P^{\ast}}^{\ast})$ is a matched
pair of differential perm algebras.

Conversely, if $((P, \fd), (P^{\ast}, \partial^{\ast}), -\fl_{P}^{\ast}, \fr_{P}^{\ast}
-\fl_{P}^{\ast}, -\fl_{P^{\ast}}^{\ast}, \fr_{P^{\ast}}^{\ast}-\fl_{P^{\ast}}^{\ast})$
is a matched pair of differential perm algebras, by Proposition \ref{pro:matda},
we get $((P\bowtie P^{\ast}, \fd\oplus\partial^{\ast}, \mathfrak{B}_{d}), P,
P^{\ast})$ is a Manin triple of differential perm algebras.
\end{proof}

Next, we consider the bialgebras of differential perm algebras.
Recall that a {\bf perm coalgebra} $(P, \vartheta)$ is a vector space $P$
with a linear map $\vartheta: P\rightarrow P\otimes P$ satisfying:
$$
(\vartheta\otimes\id)\circ\vartheta=(\id\otimes\vartheta)\circ\vartheta
=((\tau\circ\vartheta)\otimes\id)\circ\vartheta,
$$
where $\tau: P\otimes P\rightarrow P\otimes P$ is the twist map defined by
$\tau(p_{1}\otimes p_{2}):=p_{2}\otimes p_{1}$ for all $p_{1}, p_{2}\in P$.

\begin{defi}[\cite{Hou,LZB}]\label{def:permbialgs}
A {\bf perm bialgebra} is a triple $(P, \diamond, \vartheta)$ consisting of a vector space
$P$ and two linear maps $\diamond: P\otimes P\rightarrow P$, $\vartheta: P\rightarrow
P\otimes P$, such that $(P, \diamond)$ is a perm algebra, $(P, \vartheta)$ is a perm
coalgebra and
\begin{align}
&\qquad\qquad (\fr_{P}(p_{1})\otimes\id)(\vartheta(p_{2}))
=\tau((\fr_{P}(p_{2})\otimes\id)(\vartheta(p_{1}))),                     \label{bialg1} \\
&\qquad \vartheta(p_{1}\diamond p_{2})=((\fl_{P}-\fr_{P})(p_{1})\otimes\id)(\vartheta(p_{2}))
+(\id\otimes\fr_{P}(p_{2}))(\vartheta(p_{1})),                          \label{bialg2} \\
& \vartheta(p_{1}\diamond p_{2})=(\id\otimes\fl_{P}(p_{1}))(\vartheta(p_{2}))
+((\fl_{P}-\fr_{P})(p_{2})\otimes\id)(\vartheta(p_{1})-\tau(\vartheta(p_{1}))),\label{bialg3}
\end{align}
for any $p_{1}, p_{2}\in P$.
\end{defi}

Let $(P, \vartheta)$ be a perm coalgebra. A linear map $\partial: P\rightarrow P$ is called
a {\bf coderivation} on $(P, \vartheta)$ if $\vartheta\circ\partial=(\id\otimes\partial
+\partial\otimes\id)\circ\vartheta$. A {\bf codifferential perm coalgebra} is a triple
$(P, \vartheta, \partial)$, consisting of a perm coalgebra $(P, \vartheta)$ and a
coderivation $\partial: P\rightarrow P$. Let $(P, \vartheta, \partial)$ be a codifferential
perm coalgebra and $\fd: P\rightarrow P$ be a linear map. If
$$
\vartheta\circ\fd=(\id\otimes\fd-\partial\otimes\id)\circ\vartheta
=(\fd\otimes\id-\id\otimes\partial)\circ\vartheta,
$$
then $(P, \vartheta, \partial, \fd)$ is said to be an {\bf admissible codifferential
perm coalgebra}. The notion of an admissible codifferential perm coalgebra
is the dualization of an admissible differential perm algebra, that is,
$(P, \vartheta, \partial, \fd)$ is an admissible codifferential perm coalgebra if
and only if $(P^{\ast}, \vartheta^{\ast}, \partial^{\ast}, \fd^{\ast})$ is an admissible
differential perm algebra.

\begin{defi}\label{def:avebialg}
A {\bf differential perm bialgebra} is a quintuple $(P, \diamond, \vartheta,
\fd, \partial)$ satisfying
\begin{enumerate}\itemsep=0pt
\item[-] $(P, \diamond, \vartheta)$ is a perm bialgebra;
\item[-] $(P, \diamond, \fd, \partial)$ is an admissible differential perm algebra;
\item[-] $(P, \vartheta, \partial, \fd)$ is an admissible codifferential perm coalgebra.
\end{enumerate}
\end{defi}

Let $(P, \diamond, \vartheta)$ be a perm bialgebra. A linear map $\fd: P\rightarrow P$
is called {\bf derivation on} $(P, \diamond, \vartheta)$ if $\fd$ is both a derivation on
$(P, \diamond)$ and a coderivation on $(P, \vartheta)$. Clearly, if $\fd$ is a derivation
on $(P, \diamond, \vartheta)$, then $(P, \diamond, \vartheta, \fd, -\fd)$ is a
differential perm bialgebra.

\begin{pro}\label{pro:bi-mat}
Let $(P, \diamond, \fd)$ be a differential perm algebra and $\partial: P\rightarrow P$
be a linear map. Suppose that there is a linear map $\vartheta: P\rightarrow P\otimes P$
such that $(P, \vartheta, \partial)$ is a codifferential perm coalgebra. Then the quintuple
$(P, \diamond, \vartheta, \fd, \partial)$ is a differential perm bialgebra if and
only if $((P, \fd), (P^{\ast}, \partial^{\ast}), -\fl_{P}^{\ast}, \fr_{P}^{\ast}
-\fl_{P}^{\ast}, -\fl_{P^{\ast}}^{\ast}, \fr_{P^{\ast}}^{\ast}-\fl_{P^{\ast}}^{\ast})$
is a matched pair of differential perm algebras, where $(P^{\ast}, \vartheta^{\ast},
\partial^{\ast})$ is the dual algebra of $(P, \vartheta, \partial)$.
\end{pro}

\begin{proof}
If the quintuple $(P, \diamond, \vartheta, \fd, \partial)$ is a differential perm
bialgebra, then $(P, \diamond, \vartheta)$ is a perm bialgebra, and so that $(P,
P^{\ast}, -\fl_{P}^{\ast}, \fr_{P}^{\ast}-\fl_{P}^{\ast}, -\fl_{P^{\ast}}^{\ast},
\fr_{P^{\ast}}^{\ast}-\fl_{P^{\ast}}^{\ast})$ is a matched pair of perm algebras.
Moreover, by the definition of differential perm bialgebra, we get $(P^{\ast},
-\fl_{P}^{\ast}, \fr_{P}^{\ast}-\fl_{P}^{\ast}, \partial^{\ast})$ is a bimodule
over $(P, \diamond, \fd)$, and $(P, -\fl_{P^{\ast}}^{\ast}, \fr_{P^{\ast}}^{\ast}
-\fl_{P^{\ast}}^{\ast}, \fd)$ is a bimodule over $(P^{\ast}, \vartheta^{\ast},
\partial^{\ast})$. Hence, $((P, \fd), (P^{\ast}, \partial^{\ast}), -\fl_{P}^{\ast},
\fr_{P}^{\ast}-\fl_{P}^{\ast}, -\fl_{P^{\ast}}^{\ast}, \fr_{P^{\ast}}^{\ast}
-\fl_{P^{\ast}}^{\ast})$ is a matched pair of differential perm algebras.

Conversely, if $((P, \fd), (P^{\ast}, \partial^{\ast}), -\fl_{P}^{\ast}, \fr_{P}^{\ast}
-\fl_{P}^{\ast}, -\fl_{P^{\ast}}^{\ast}, \fr_{P^{\ast}}^{\ast}-\fl_{P^{\ast}}^{\ast})$
is a matched pair of differential perm algebras, then $(P, P^{\ast}, -\fl_{P}^{\ast},
\fr_{P}^{\ast}-\fl_{P}^{\ast}, -\fl_{P^{\ast}}^{\ast}, \fr_{P^{\ast}}^{\ast}
-\fl_{P^{\ast}}^{\ast})$ is a matched pair of perm algebras, and so that,
$(P, \diamond, \vartheta)$ is a perm bialgebra. Moreover, by the definition of matched
pair of differential perm algebras, we get $(P^{\ast}, -\fl_{P}^{\ast}, \fr_{P}^{\ast}
-\fl_{P}^{\ast}, \partial^{\ast})$ is a bimodule over $(P, \diamond, \fd)$ and $(P,
-\fl_{P^{\ast}}^{\ast}, \fr_{P^{\ast}}^{\ast}-\fl_{P^{\ast}}^{\ast}, \fd)$ is a
bimodule over $(P^{\ast}, \vartheta^{\ast}, \partial^{\ast})$. Thus,
$(P, \diamond, \vartheta, \fd, \partial)$ is a differential perm bialgebra.
\end{proof}

Combining Propositions \ref{pro:dou-mat} and \ref{pro:bi-mat},
we have the following conclusion.

\begin{thm}\label{thm:equ}
Let $(P, \diamond, \fd)$ be a differential perm algebra and $\partial: P\rightarrow P$ be
a linear map. Suppose that there is a linear map $\vartheta: P\rightarrow
P\otimes P$ such that $(P^{\ast}, \vartheta^{\ast}, \partial^{\ast})$
is a differential perm algebra. Then the following conditions are equivalent:
\begin{enumerate}\itemsep=0pt
\item[$(i)$] $((P\bowtie P^{\ast}, \fd\oplus\partial^{\ast}, \mathfrak{B}_{d}), P,
     P^{\ast})$ is a Manin triple of differential perm algebras $(P, \diamond, \fd)$ and
     $(P^{\ast}, \vartheta^{\ast}, \partial^{\ast})$;
\item[$(ii)$] $((P, \fd), (P^{\ast}, \partial^{\ast}), -\fl_{P}^{\ast}, \fr_{P}^{\ast}
     -\fl_{P}^{\ast}, -\fl_{P^{\ast}}^{\ast}, \fr_{P^{\ast}}^{\ast}-\fl_{P^{\ast}}^{\ast})$
     is a matched pair of differential perm algebras;
\item[$(iii)$] $(P, \diamond, \vartheta, \fd, \partial)$ is a differential perm bialgebra.
\end{enumerate}
\end{thm}

\section{Quasi-triangular differential perm bialgebras and perm Yang-Baxter equation}
\label{sec:quasi}
In this section, we study the quasi-triangular differential perm bialgebras
and some special quasi-triangular differential perm bialgebras. Moreover,
introduce the notion of the differential perm Yang-Baxter equation in a differential
perm algebra. First, we give the notion of coboundary differential perm bialgebras.
A differential perm bialgebra is called coboundary if it as a perm bialgebra is coboundary.

\begin{defi}\label{def:cob}
A differential perm bialgebra $(P, \diamond, \vartheta, \fd, \partial)$ is called
{\bf coboundary} if there exists an element $r\in P\otimes P$, such that
\begin{align}
\vartheta(p)=\vartheta_{r}(p):=((\fl_{P}-\fr_{P})(p)\otimes\id
-\id\otimes\fr_{P}(p))(r),\label{cobo}
\end{align}
for any $p\in P$. We call $(P, \diamond, \vartheta_{r}, \fd, \partial)$ a differential
perm bialgebra induced by $r$.
\end{defi}

Let $(P, \diamond)$ be a perm algebra and $r=\sum_{i}x_{i}\otimes y_{i}\in P\otimes P$.
Then
$$
\mathbf{P}_{r}:=r_{12}\diamond r_{23}-r_{13}\diamond r_{23}
+r_{12}\diamond r_{13}-r_{13}\diamond r_{12}=0
$$
is called the {\bf perm Yang-Baxter equation} (or $\PYBE$) in $(P, \diamond)$,
where $r_{12}\diamond r_{23}:=\sum_{i,j}x_{i}\otimes(y_{i}\diamond x_{j})\otimes y_{j}$,
$r_{13}\diamond r_{23}:=\sum_{i,j}x_{i}\otimes x_{j}\otimes(y_{i}\diamond y_{j})$,
$r_{12}\diamond r_{13}:=\sum_{i,j}(x_{i}\diamond x_{j})\otimes y_{i}\otimes y_{j}$,
$r_{13}\diamond r_{12}:=\sum_{i,j}(x_{i}\diamond x_{j})\otimes y_{j}\otimes y_{i}$.
An element $r\in P\otimes P$ is called {\bf perm-invariant} if
$$
\vartheta_{r}(p):=((\fl_{P}-\fr_{P})(p)\otimes\id-\id\otimes\fr_{P}(p))(r)=0
$$
for any $p\in P$, $r$ is called {\bf symmetric} if $r=\tau(r)$.
Clearly, $r-\tau(r)=0$ is perm-invariant if $r$ is symmetric.

\begin{pro}[\cite{Lin}]\label{pro:quasi-ba}
Let $(P, \diamond)$ be a perm algebra and $r\in P\otimes P$.
If $r$ is a solution of the $\PYBE$ in $(P, \diamond)$ and $r-\tau(r)$ is perm-invariant,
then the $(P, \diamond, \vartheta_{r})$ with $\vartheta_{r}$ defined by Eq. \eqref{cobo} is
a perm bialgebra, which is called a {\bf quasi-triangular perm bialgebra}.
In particular, if $r$ is symmetric solution of the $\PYBE$ in $(P, \diamond)$, then
$(P, \diamond, \vartheta_{r})$ is called a {\bf triangular perm bialgebra}.
\end{pro}

We next consider the differential perm bialgebras.

\begin{lem}\label{lem:coba}
Let $(P, \diamond, \fd, \partial)$ be an admissible differential perm algebra and
$r\in P\otimes P$. If the linear map $\vartheta_{r}: P\rightarrow P\otimes P$ by Eq.
\eqref{cobo} defines a perm coalgebra structure on $(P, \diamond)$, then $\partial$ is a
coderivation on $(P, \vartheta_{r})$ if and only if for any $p\in P$,
\begin{align}
\big((\id\otimes\fr_{P}(p))(\id\otimes\fd-\partial\otimes\id)
+((\fr_{P}-\fl_{P})(p)\otimes\id)(\fd\otimes\id-\id\otimes\partial)\big)(r)=0, \label{coba1}
\end{align}
If $(P, \vartheta_{r}, \partial)$ is a codifferential perm coalgebra, $(P^{\ast},
\vartheta_{r}^{\ast}, \partial^{\ast}, \fd^{\ast})$ is admissible
if and only if for any $p\in P$,
\begin{align}
&\big(((\fr_{P}-\fl_{P})(p)\otimes\id+\id\otimes\fr_{P}(p))(\partial\otimes\id
-\id\otimes\fd)\big)(r)=0,                                    \label{coba2}\\
&\big(((\fr_{P}-\fl_{P})(p)\otimes\id+\id\otimes\fr_{P}(p))(\id\otimes\partial
-\fd\otimes\id)\big)(r)=0,                                    \label{coba3}
\end{align}
\end{lem}

\begin{proof}
First, since $(P, \fl_{P}, \fr_{P}, -\partial)$ is a bimodule over $(P, \diamond, \fd)$,
for any $p\in P$, we get $\fl_{P}(\partial(p))=\partial\circ\fl_{P}(p)+\fl_{P}(p)\circ\fd$
and $\fr_{P}(\partial(p))=\partial\circ\fr_{P}(p)+\fr_{P}(p)\circ\fd$, and so that,
\begin{align*}
&\;\vartheta_{r}(\partial(p))-(\id\otimes\partial+\partial\otimes\id)(\vartheta_{r}(p))\\
=&\;\Big(\id\otimes(\partial\circ\fr_{P}(p))+\id\otimes(\fr_{P}(p)\circ\fd)
+(\partial\circ\fr_{P}(p))\otimes\id+(\fr_{P}(p)\circ\fd)\otimes\id\\[-2mm]
&\qquad-(\partial\circ\fl_{P}(p))\otimes\id+(\fl_{P}(p)\circ\fd)\otimes\id
-\id\otimes(\partial\circ\fr_{P}(p))-\fr_{P}(p)\otimes\partial\\[-2mm]
&\qquad\qquad+\fl_{P}(p)\otimes\partial-\partial\otimes\fr_{P}(p)
-(\partial\circ\fr_{P}(p))\otimes\id+(\partial\circ\fl_{P}(p))\otimes\id\Big)(r)\\
=&\;\Big((\id\otimes\fr_{P}(p))(\id\otimes\fd-\partial\otimes\id)
+((\fr_{P}-\fl_{P})(p)\otimes\id)(\fd\otimes\id-\id\otimes\partial)\Big)(r).
\end{align*}
Thus, $\partial$ is a coderivation on $(P, \vartheta_{r})$ if and only if
Eq. \eqref{coba1} holds.

Next, since $(P^{\ast}, \fl_{P^{\ast}}, \fr_{P^{\ast}}, -\fd^{\ast})$ is a bimodule
over $(P^{\ast}, \vartheta_{r}^{\ast}, \partial^{\ast})$ if and only if
$$
(\id\otimes\fd)\circ\vartheta_{r}-(\partial\otimes\id)\circ\vartheta_{r}
=\vartheta_{r}\circ\fd=(\fd\otimes\id)\circ\vartheta_{r}-(\id\otimes\partial)\circ\vartheta_{r}.
$$
Similar to the calculation above, we can get that $\vartheta_{r}\circ\fd=
(\id\otimes\fd)\circ\vartheta_{r}-(\partial\otimes\id)\circ\vartheta_{r}$ if and only
if Eq. \eqref{coba2} holds, and $\vartheta_{r}\circ\fd=(\fd\otimes\id)\circ\vartheta_{r}
-(\id\otimes\partial)\circ\vartheta_{r}$ if and only if Eq.
\eqref{coba3} holds. The proof is completed.
\end{proof}

Let $(P, \diamond, \fd, \partial)$ be an admissible differential perm algebra and
$r\in P\otimes P$. Define a linear map $\vartheta_{r}$ by Eq. \eqref{cobo}.
Then $(P, \diamond, \vartheta_{r}, \fd, \partial)$ is a differential perm bialgebra if and
only if Eqs. \eqref{coba1}-\eqref{coba3} hold. In particular, $(P, \diamond, \vartheta_{r},
\fd, -\fd)$ is a differential perm bialgebra if and only if Eq. \eqref{coba1}
holds for $\partial=-\fd$.

\begin{defi}\label{def:dpYBE}
Let $(P, \diamond, \fd, \partial)$ be an admissible differential perm algebra and
$r\in P\otimes P$. Then equation $\mathbf{P}_{r}=0$ and
$$
(\fd\otimes\id-\id\otimes\partial)(r)=0,\qquad\qquad
(\partial\otimes\id-\id\otimes\fd)(r)=0,
$$
are called the {\bf differential perm Yang-Baxter equations (or $\DPYBE$) in
$(P, \diamond, \fd, \partial)$}.
\end{defi}

By Proposition \ref{pro:quasi-ba} and Lemma \ref{lem:coba}, we have:

\begin{pro}\label{pro:quasi-dpba}
Let $(P, \diamond, \fd, \partial)$ be an admissible differential perm algebra and
$r\in P\otimes P$. If $r$ is a solution of the $\DPYBE$ in $(P, \diamond, \fd, \partial)$
and $r-\tau(r)$ is perm-invariant, then the $(P, \diamond, \vartheta_{r}, \fd, \partial)$
with $\vartheta_{r}$ defined by Eq. \eqref{cobo} is a differential perm bialgebra,
which is called a {\bf quasi-triangular differential perm bialgebra}.
In particular, if $r$ is symmetric solution of the $\DPYBE$ in $(P, \diamond, \fd,
\partial)$, then $(P, \diamond, \vartheta_{r}, \fd, \partial)$ is called a {\bf
triangular differential perm bialgebra}.
\end{pro}

Let $(P, \diamond)$ be a perm algebra and $(V, \kl, \kr)$ be a bimodule over
$(P, \diamond)$. Recall that a linear map $T: V\rightarrow P$ is called an {\bf
$\mathcal{O}$-operator of $(P, \diamond)$ associated to $(V, \kl, \kr)$} if for
any $v_{1}, v_{2}\in V$,
$$
T(v_{1})\diamond T(v_{2})=T\big(\kl(T(v_{1}))(v_{2})+\kr(T(v_{2}))(v_{1})\big).
$$
Let $(P, \diamond, \fd)$ be a differential perm algebra and $(V, \kl, \kr, \partial)$
be a bimodule over $(P, \diamond, \fd)$. A linear map $T: V\rightarrow P$ is called an {\bf
$\mathcal{O}$-operator of $(P, \diamond, \fd)$ associated to $(V, \kl, \kr, \partial)$} if
$T$ is an $\mathcal{O}$-operator of $(P, \diamond)$ associated to $(V, \kl, \kr)$ and
$T\circ\partial=\fd\circ T$. Let $V$ be a linear space and $r\in V\otimes V$. We define
a linear map $r^{\sharp}: V^{\ast}\rightarrow V$ by
$$
\langle r^{\sharp}(\xi_{1}),\; \xi_{2}\rangle=\langle\xi_{1}\otimes\xi_{2},\; r\rangle,
$$
for any $\xi_{1}, \xi_{2}\in V^{\ast}$. Then we have the following proposition.

\begin{pro}\label{pro:o-dperm}
Let $(P, \diamond, \fd, \partial)$ be an admissible differential perm algebra and
$r\in P\otimes P$ be symmetric. Then $r$ is a solution of the $\DPYBE$ in $(P, \diamond,
\fd, \partial)$ if and only if $r^{\sharp}: P^{\ast}\rightarrow P$ is an
$\mathcal{O}$-operator of $(P, \diamond, \fd)$ associated to the coregular bimodule
$(P^{\ast}, -\fl_{P}^{\ast}, \fr_{P}^{\ast}-\fl_{P}^{\ast}, \partial^{\ast})$.
\end{pro}

\begin{proof}
First, since $r$ is symmetric, by Corollary 3.35 in \cite{LZB}, we get $r$ is a
solution of the $\PYBE$ in $(P, \diamond)$ if and only if $r^{\sharp}: P^{\ast}
\rightarrow P$ is an $\mathcal{O}$-operator of $(P, \diamond)$ associated to the
bimodule $(P^{\ast}, -\fl_{P}^{\ast}, \fr_{P}^{\ast}-\fl_{P}^{\ast})$.
Second, for any $\xi_{1}, \xi_{2}\in P^{\ast}$, note that
\begin{align*}
& \langle r^{\sharp}(\partial^{\ast}(\xi_{1})),\; \xi_{2}\rangle
=\langle\partial^{\ast}(\xi_{1})\otimes\xi_{2},\; r\rangle
=\langle\xi_{1}\otimes\xi_{2},\; (\partial\otimes\id)(r)\rangle,\\
& \langle\fd(r^{\sharp}(\xi_{1})),\; \xi_{2}\rangle
=\langle\xi_{1}\otimes\fd^{\ast}(\xi_{2}),\; r\rangle
=\langle\xi_{1}\otimes\xi_{2},\; (\id\otimes\fd)(r)\rangle.
\end{align*}
We get $T\circ\partial^{\ast}=\fd\circ T$ if and only if $(\partial\otimes\id)(r)
=(\id\otimes\fd)(r)$. Moreover, since $r$ is symmetric, $(\partial\otimes\id)(r)
=(\id\otimes\fd)(r)$$\Leftrightarrow$$(\id\otimes\partial)(r)=(\fd\otimes\id)(r)$.
Thus, we get the proof.
\end{proof}

The Drinfeld double plays a crucial role in the construction of Lie bialgebras.
In \cite{Lin}, Lin consider the double of perm bialgebras and show that any double
of a perm bialgebra is quasi-triangular. Next, we will elevate this result to
the differential perm bialgebras.

\begin{pro}\label{pro:doublebia}
Let $(P, \diamond, \vartheta, \fd, \partial)$ be a differential perm bialgebra.
Then, there is a quasi-triangular differential perm bialgebra structure on the direct sum
$P\oplus P^{\ast}$.
\end{pro}

\begin{proof}
Let $\{e_{1}, e_{2}, \cdots, e_{n}\}$ be a basis of $P$, $\{f_{1}, f_{2}, \cdots, f_{n}\}$
be the dual basis in $P^{\ast}$ and $r=\sum_{i=1}^{n}e_{i}\otimes f_{i}
\in P\otimes P^{\ast}\subset(P\oplus P^{\ast})\otimes(P\oplus P^{\ast})$.
Since $(P, \diamond, \vartheta, \fd, \partial)$ is a differential perm bialgebra, there
is a corresponding matched pair $((P, \fd), (P^{\ast}, \partial^{\ast}), -\fl_{P}^{\ast},
\fr_{P}^{\ast}-\fl_{P}^{\ast}, -\fl_{P^{\ast}}^{\ast}, \fr_{P^{\ast}}^{\ast}
-\fl_{P^{\ast}}^{\ast})$. Let $(P\bowtie P^{\ast}, \fd\oplus\partial^{\ast})$ be the
differential perm algebra obtained from this matched pair. By Lemma \ref{lem:douadm},
we obtain that $(P\bowtie P^{\ast}, \fd\oplus\partial^{\ast}, \partial\oplus\fd^{\ast})$
is an admissible differential perm algebra. Moreover, by \cite[Theorem 3.6]{Lin},
we get that $r$ is a solution of the $\PYBE$ in $P\bowtie P^{\ast}$ and $r-\tau(r)$ is
perm-invariant. For any $p\in P$ and $\xi\in P^{\ast}$, since
$$
\Big\langle\sum_{i}\fd(e_{i})\otimes f_{i},\; \xi\otimes p\Big\rangle
=\langle\xi,\; \fd(p)\rangle=\Big\langle\sum_{i}e_{i}\otimes\fd(f_{i}),\;
\xi\otimes p\Big\rangle,
$$
we get $\sum_{i}\fd(e_{i})\otimes f_{i}=\sum_{i}e_{i}\otimes\fd(f_{i})$, and so that
$((\fd\oplus\partial^{\ast})\otimes\id)(r)=(\id\otimes(\partial\oplus\fd^{\ast}))(r)$.
Similarly, $((\partial\oplus\fd^{\ast})\otimes\id)(r)=(\id\otimes(\fd\oplus
\partial^{\ast}))(r)$. Hence, $r$ is a solution of the $\DPYBE$ in $(P\bowtie P^{\ast},
\fd\oplus\partial^{\ast}, \partial\oplus\fd^{\ast})$ and $r-\tau(r)$ is perm-invariant.
By Proposition \ref{pro:quasi-dpba}, we obtain that $(P\bowtie P^{\ast},
\bar{\vartheta}_{r}, \fd\oplus\partial^{\ast}, \partial\oplus\fd^{\ast})$ is a
quasi-triangular differential perm bialgebra, where $\bar{\vartheta}_{r}$ is given by
Eq. \eqref{cobo}.
\end{proof}

Proposition \ref{pro:doublebia} provides a method for constructing quasi-triangular
differential perm bialgebras. More precisely, for any differential perm bialgebra
$(P, \diamond, \vartheta, \fd, \partial)$, we get a quasi-triangular differential perm
bialgebra $(P\bowtie P^{\ast}, \bar{\vartheta}_{r}, \fd\oplus\partial^{\ast},
\partial\oplus\fd^{\ast})$, which is called the {\bf double differential perm
bialgebra} of $(P, \diamond, \vartheta, \fd, \partial)$.

\begin{ex}\label{ex:YBE-bialg}
Let $(P, \diamond, \fd, -\fd)$ be a $2$-dimensional admissible differential perm algebra,
where vector space $P={\rm span}_{\Bbbk}\{x_{1}, x_{2}\}$, $\fd(x_{1})=x_{2}$,
$\fd(x_{2})=0$, and the nonzero products are given by $x_{1}\diamond x_{1}=x_{1}$ and
$x_{1}\diamond x_{2}=x_{2}$. Let $r=x_{2}\otimes x_{2}$. Then $r$ is a symmetric solution
of the $\DPYBE$ in $(P, \diamond, \fd, -\fd)$. Thus we get a triangular differential
perm bialgebra $(P, \diamond, \vartheta_{r}, \fd, -\fd)$, where $\vartheta_{r}(x_{1})
=x_{2}\otimes x_{2}$ and $\vartheta_{r}(x_{2})=0$. Consider the double differential
perm bialgebra $(P\oplus P^{\ast}, \bar{\diamond}, \bar{\vartheta}_{\tilde{r}},
\fd\oplus-\fd^{\ast}, -\fd\oplus\fd^{\ast})$, where $\tilde{r}=x_{1}\otimes y_{1}
+x_{2}\otimes y_{2}$, $\{y_{1}, y_{2}\}\subseteq P^{\ast}$ is the dual basis of
$\{x_{1}, x_{2}\}$. By direct calculation, we get that
\begin{align*}
& x_{1}\bar{\diamond}x_{1}=x_{1},\qquad x_{1}\bar{\diamond}x_{2}=x_{2},\qquad
y_{2}\bar{\diamond}y_{2}=y_{2}\bar{\diamond}x_{2}=-y_{1}=-x_{1}\bar{\diamond}y_{1},\qquad
x_{1}\bar{\diamond}y_{2}=y_{2}, \qquad y_{2}\bar{\diamond}x_{1}=x_{2}+y_{2},\\
&\qquad\qquad \bar{\vartheta}_{\tilde{r}}(x_{1})=-x_{2}\otimes x_{2},\qquad\quad
\bar{\vartheta}_{\tilde{r}}(y_{1})=-y_{1}\otimes y_{1},\qquad\quad
\bar{\vartheta}_{\tilde{r}}(y_{2})=-y_{1}\otimes y_{2}.
\end{align*}
Moreover, one can check that $\tilde{r}$ is a solution of the $\DPYBE$ in
$(P\oplus P^{\ast}, \bar{\diamond}, -\fd\oplus\fd^{\ast}, -\fd\oplus\fd^{\ast})$
and $\tilde{r}-\tau(\tilde{r})$ is perm-invariant. That is, $(P\oplus P^{\ast},
\bar{\diamond}, \bar{\vartheta}_{\tilde{r}}, \fd\oplus-\fd^{\ast}, -\fd\oplus\fd^{\ast})$
is a quasi-triangular differential perm bialgebra.
\end{ex}

Now, we introduce the definition of factorizable differential perm bialgebra.

\begin{defi}\label{def:fact-diffbia}
Let $(P, \diamond, \fd, \partial)$ be an admissible differential perm algebra,
$r\in P\otimes P$ and $(P, \diamond, \vartheta_{r}, \fd, \partial)$ be the quasi-triangular
differential perm bialgebra associated with $r$. If $\mathcal{I}=r^{\sharp}-\tau(r)^{\sharp}:
P^{\ast}\rightarrow P$ is a bijection and $\mathcal{I}\circ\partial^{\ast}
=\fd\circ\mathcal{I}$, then $(P, \diamond, \vartheta_{r}, \fd, \partial)$ is called
a {\bf factorizable differential perm bialgebra}.
\end{defi}

The factorizable perm bialgebras have been studied in \cite{Lin}.
The reason why it is called this name is any element in this perm bialgebra
has a unique decomposition. For the factorizable differential perm bialgebras
This unique decomposition still holds.

\begin{pro}\label{pro:fact-diff}
Let $(P, \diamond, \fd, \partial)$ be an admissible differential perm algebra and
$r\in P\otimes P$. Assume that the differential perm bialgebra $(P, \diamond, \vartheta_{r},
\fd, \theta)$ is factorizable. Then any $p\in P$ has a unique decomposition $p=p_{+}+p_{-}$,
where $p_{+}\in\Img(r^{\sharp})$ and $p_{-}\in\Img(\tau(r)^{\sharp})$.
\end{pro}

An important example of factorizable Lie bialgebra is implemented using the Drinfeld
double of Lie bialgebra \cite{LS}. Recently, Lin proved that the double of every
perm bialgebra is factorizable. Here we present that this conclusion is also
correct at the level of differential perm algebras.

\begin{pro}\label{pro:doufactdiff}
Let $(P, \diamond, \vartheta, \fd, \theta)$ be a differential perm bialgebra.
Then the double differential perm bialgebra $(P\bowtie P^{\ast}, \bar{\vartheta}_{r},
\fd\oplus\partial^{\ast}, \partial\oplus\fd^{\ast})$ of $(P, \diamond, \vartheta, \fd,
\partial)$ is factorizable, where $r=\sum_{i=1}^{n}e_{i}\otimes f_{i}$, $\{e_{1}, e_{2},
\cdots, e_{n}\}$ is a basis of $P$ and $\{f_{1}, f_{2}, \cdots, f_{n}\}$ is its dual
basis in $P^{\ast}$.
\end{pro}

\begin{proof}
For any $p\in P$ and $\xi\in P^{\ast}$, by direct calculation, we get that the linear maps
$r^{\sharp}, T(r)^{\sharp}: (P\oplus P^{\ast})^{\ast}=P^{\ast}\oplus P
\rightarrow P\oplus P^{\ast}$ are given by
$$
r^{\sharp}(\xi, p)=(0, \xi)\qquad\mbox{ and }\qquad \tau(r)^{\sharp}(\xi, p)=(p, 0).
$$
Thus, $\mathcal{I}=r^{\sharp}-\tau(r)^{\sharp}$ is given by $(\xi, p)\mapsto(-p, \xi)$,
which is a linear isomorphism and $\mathcal{I}\circ(\partial\oplus\fd^{\ast})^{\ast}
=(\fd\oplus\partial^{\ast})\circ\mathcal{I}$. Thus, the double differential perm bialgebra
$(P\bowtie P^{\ast}, \bar{\vartheta}_{r}, \fd\oplus\partial^{\ast},
\partial\oplus\fd^{\ast})$ is factorizable.
\end{proof}

Thus, we obtain that the quasi-triangular differential perm bialgebra $(P\oplus P^{\ast},
\bar{\diamond}, \bar{\vartheta}_{\tilde{r}}, \fd\oplus-\fd^{\ast}, -\fd\oplus\fd^{\ast})$
given in Example \ref{ex:YBE-bialg} is factorizable.
In \cite{Lin}, the notion of quadratic Rota-Baxter perm algebras of weights was
introduced, and a one-to-one correspondence between quadratic Rota-Baxter perm algebras
of nonzero weights and factorizable perm bialgebras has been given.
At the end of this chapter, we would like to point out that these conclusions can be
extended to the level of differential perm algebra by introducing Rota-Baxter operators
on differential perm algebra. This will not be provided here.

\section{Differential perm bialgebras and related bialgebra structures} \label{sec:app}
In this section, we will discuss in detail the relationship between differential
infinitesimal bialgebras, differential perm bialgebras, Novikov bialgebras and
diNovikov bialgebras. Specifically divided into four parts: we show that there is a
differential infinitesimal bialgebra structure on the tensor product of a differential
perm bialgebra and a quadratic Zinbiel algebra; each differential perm bialgebras
induces a diNovikov bialgebra; there is a Novikov bialgebra structure on
the tensor product of a diNovikov bialgebra and a quadratic Zinbiel algebra;
each differential infinitesimal bialgebras induces a Novikov bialgebra. Therefore, we provide
two methods to construct a Novikov bialgebra from a differential perm
bialgebra, and we show that the Novikov bialgebra obtained by
these two methods are consistent.

\subsection{From differential perm bialgebras to differential infinitesimal bialgebras}
\label{subsec:inf-difperm}
Let $(A, \cdot, \fd)$ be a differential algebra and $\partial: A\rightarrow A$ be a linear
map. If for any $a_{1}, a_{2}\in A$,
$$
\partial(a_{1}a_{2})=\partial(a_{1})a_{2}-a_{1}\fd(a_{2}),
$$
then we say that $\partial$ is {\bf admissible to $(A, \cdot, \fd)$}. We also say
that the quadruple $(A, \cdot, \fd, \partial)$ is an {\bf admissible differential algebra}.

\begin{pro}\label{pro:tensor-ad}
Let $(P, \diamond, \fd, \partial)$ be an admissible differential perm algebra and
$(C, \star)$ be a Zinbiel algebra. Define a binary operation $\cdot: (P\otimes C)\otimes
(P\otimes C)\rightarrow(P\otimes C)$ by
$$
(p_{1}\otimes c_{1})\cdot(p_{2}\otimes c_{2})
:=(p_{1}\diamond p_{2})\otimes(c_{1}\star c_{2})
+(p_{2}\diamond p_{1})\otimes(c_{2}\star c_{1}),
$$
for any $p_{1}, p_{2}\in P$ and $c_{1}, c_{2}\in C$. Then $(P\otimes C, \cdot,
\hat{\fd}, \hat{\partial})$ is an admissible differential algebra, which is
called {\bf the admissible differential algebra induced from $(P, \diamond, \fd, \partial)$
by $(C, \star)$}, where $\hat{\fd}=\fd\otimes\id$ and $\hat{\partial}=\partial\otimes\id$.
\end{pro}

\begin{proof}
First, by Proposition \ref{pro:tensor}, we get that $(P\otimes C, \cdot,
\hat{\fd})$ is a differential algebra. Second, since $\partial$ is
admissible to $(P, \diamond, \fd)$, for any $p_{1}, p_{2}\in P$ and $c_{1}, c_{2}\in C$,
\begin{align*}
&\; \hat{\partial}((p_{1}\otimes c_{1})\cdot(p_{2}\otimes c_{2}))
-(p_{1}\otimes c_{1})\cdot\hat{\partial}(p_{2}\otimes c_{2})
+\hat{\fd}(p_{1}\otimes c_{1})\cdot(p_{2}\otimes c_{2})\\
=&\; \big(\partial(p_{1}\diamond p_{2})-p_{1}\diamond\partial(p_{2})+\fd(p_{1})
\diamond p_{2}\big)\otimes(c_{1}\star c_{2})
+\big(\partial(p_{2}\diamond p_{1})-p_{2}\diamond\partial(p_{1})+\fd(p_{2})
\diamond p_{1}\big)\otimes(c_{2}\star c_{1})\\
=&\; 0.
\end{align*}
That is, $\hat{\partial}((p_{1}\otimes c_{1})\cdot(p_{2}\otimes c_{2}))=
(p_{1}\otimes c_{1})\cdot\hat{\partial}(p_{2}\otimes c_{2})
-\hat{\fd}(p_{1}\otimes c_{1})\cdot(p_{2}\otimes c_{2})$, $\hat{\partial}$ is
admissible to $(P\otimes C, \cdot, \hat{\fd})$.
\end{proof}

Recall that a {\bf coassociative coalgebra} is a vector space $A$ with a linear map
$\Delta: A\rightarrow A\otimes A$ satisfying $(\Delta\otimes\id)\circ\Delta
=(\id\otimes\Delta)\circ\Delta$. A coassociative coalgebra $(A, \Delta)$ is called
{\bf cocommutative} if $\Delta=\tau\circ\Delta$. For a coassociative coalgebra
$(A, \Delta)$, a linear map $\partial: A\rightarrow A$ is called a {\bf coderivation}
on $(A, \Delta)$ if $\Delta\circ\partial=(\id\otimes\partial+\partial\otimes\id)\circ\Delta$.
A {\bf codifferential coalgebra} is a triple $(A, \Delta, \partial)$, where
$(A, \Delta)$ is a cocommutative coassociative coalgebra and $\partial$ is a
coderivation. Obviously, $(A, \Delta, \partial)$ is a codifferential coalgebra if
and only if $(A^{\ast}, \Delta^{\ast}, \partial^{\ast})$ is a differential algebra.
Let $(A, \Delta, \partial)$ be a codifferential coalgebra. If $\fd: A\rightarrow A$
is a linear map satisfying
$$
(\fd\otimes\id-\id\otimes\partial)\circ\Delta=\Delta\circ\fd,
$$
then we call $(A, \Delta, \partial, \fd)$ an {\bf admissible codifferential
coalgebra}. One can check that $(A, \Delta, \partial, \fd)$ is an admissible
codifferential coalgebra if and only if $(A^{\ast}, \Delta^{\ast}, \partial^{\ast},
\fd^{\ast})$ is an admissible differential algebra. Recall that a {\bf Zinbiel coalgebra}
$(C, \theta)$ is a vector space $C$ with a linear map $\theta: C\rightarrow C\otimes C$
such that
$$
(\id\otimes\theta)\circ\theta=(\theta\otimes\id)\circ\theta
+(\tau\otimes\id)\circ(\theta\otimes\id)\circ\theta.
$$
Dual for Proposition \ref{pro:tensor-ad}, we have following conclusion.

\begin{pro}\label{pro:ind-coadass}
Let $(P, \vartheta, \partial, \fd)$ be an admissible codifferential perm coalgebra and
$(C, \theta)$ be a Zinbiel coalgebra. Define a linear map $\hat{\vartheta}: P\otimes C
\rightarrow(P\otimes C)\otimes(P\otimes C)$ by
\begin{align*}
\Delta(p\otimes c)&=\vartheta(p)\bullet\theta(c)+\tau(\vartheta(p))
\bullet\tau(\theta(c))\\
&=\sum_{(c)}\sum_{(p)}\Big((p_{(1)}\otimes c_{(1)})\otimes(p_{(2)}\otimes c_{(2)})
+(p_{(2)}\otimes c_{(2)})\otimes(p_{(1)}\otimes c_{(1)})\Big),
\end{align*}
for any $p\in P$ and $c\in C$, where $\theta(c)=\sum_{(c)}c_{(1)}\otimes c_{(2)}$
and $\vartheta(p)=\sum_{(p)}p_{(1)}\otimes p_{(2)}$ in the Sweedler notation.
Then $(P\otimes C, \cdot, \hat{\partial}, \hat{\fd})$ is an admissible
codifferential coalgebra, which is called the {\bf admissible codifferential
coalgebra induced from $(P, \vartheta, \partial, \fd)$ by $(C, \theta)$}, where
$\hat{\fd}=\fd\otimes\id$ and $\hat{\partial}=\partial\otimes\id$.
\end{pro}

\begin{proof}
First, by \cite[Proposition 3.6]{GH}, we get that $(P\otimes C, \Delta)$ is a
commutative coassociative coalgebra. Second, by direct calculation, for any $p\in P$
and $c\in C$, we have
\begin{align*}
\Delta(\hat{\fd}(p\otimes c))
&=\vartheta(\fd(p))\bullet\theta(c)
+\tau(\vartheta(\fd(p)))\bullet\tau(\theta(c))\\
&=(\id\otimes\fd+\fd\otimes\id)(\vartheta(p))\bullet\theta(c)
+(\id\otimes\fd+\fd\otimes\id)(\tau(\vartheta(p)))\bullet\tau(\theta(c))\\
&=(\id\otimes\hat{\fd}+\hat{\fd}\otimes\id)(\Delta(p\otimes c)).
\end{align*}
This means that $(P\otimes C, \cdot, \hat{\fd})$ is a codifferential coalgebra.
Finally, since $(P, \vartheta, \partial, \fd)$ is admissible, i.e., $(\fd\otimes\id-
\id\otimes\partial)\circ\vartheta=\vartheta\circ\fd=(\id\otimes\fd-\partial\otimes\id)
\circ\vartheta$, for any $p\in P$ and $c\in C$, we have
\begin{align*}
\Delta(\hat{\fd}(p\otimes c))
&=\vartheta(\fd(p))\bullet\theta(c)+\tau(\vartheta(\fd(p)))\bullet\tau(\theta(c))\\
&=(\fd\otimes\id-\id\otimes\partial)(\vartheta(p))\bullet\theta(c)
+\tau((\fd\otimes\id-\id\otimes\partial)(\vartheta(p))\big)\bullet\tau(\theta(c))\\
&=(\hat{\fd}\otimes\id-\id\otimes\hat{\partial})\big(\vartheta(p)\bullet\theta(c)
+\tau(\vartheta(p))\bullet\tau(\theta(c))\big)\\
&=(\hat{\fd}\otimes\id-\id\otimes\hat{\partial})(\Delta(p\otimes c)).
\end{align*}
Thus, $(P\otimes C, \cdot, \hat{\partial}, \hat{\fd})$ is an admissible
codifferential coalgebra.
\end{proof}

Let $(A, \cdot)$ be a commutative associative algebra and $(A, \Delta)$ be a
commutative coassociative coalgebra. If for any $a_{1}, a_{2}\in A$,
\begin{align}
\Delta(a_{1}a_{2})=(\id\otimes\fu_{A}(a_{1}))(\Delta(a_{2}))
+(\fu_{A}(a_{2})\otimes\id)(\Delta(a_{1})),   \label{comm-bia}
\end{align}
then $(A, \cdot, \Delta)$ is called an {\bf infinitesimal bialgebra}, where
$\fu_{A}(a_{1})(a_{2})=a_{1}a_{2}$ for any $a_{1}, a_{2}\in A$.

\begin{defi}[\cite{LLB}]\label{def:diff-bia}
A quintuple $(A, \cdot, \Delta, \fd, \partial)$ is called a {\bf differential
infinitesimal bialgebra} if
\begin{enumerate}
\item[(i)] $(A, \cdot, \Delta)$ is an infinitesimal bialgebra,
\item[(ii)]  $(A, \cdot, \fd, \partial)$ is an admissible differential algebra,
\item[(iii)]  $(A, \Delta, \partial, \fd)$ is an admissible codifferential coalgebra.
\end{enumerate}
\end{defi}

Recall that a bilinear form $\omega(-, -)$ on a Zinbiel algebra $(C, \star)$ is called
{\bf invariant} if
$$
\omega(c_{1}\star c_{2},\; c_{3})=\omega(c_{2},\; c_{1}\star c_{3}+c_{3}\star c_{1}),
$$
for any $c_{1}, c_{2}, c_{3}\in C$. A Zinbiel algebra $(C, \star)$ with a nondegenerate
skew-symmetric invariant bilinear form $\omega(-, -)$ is called a {\bf quadratic Zinbiel
algebra} and denoted by $(C, \star, \omega)$. In this case, the bilinear form
$\omega(-,-)$ can naturally expand to the tensor product $C\otimes C\otimes
\cdots\otimes C$, i.e.,
$$
\omega(-,-):\qquad (\underbrace{C\otimes\cdots\otimes C}_{\mbox{\tiny $k$-fold}})\otimes
(\underbrace{C\otimes\cdots\otimes C}_{\mbox{\tiny $k$-fold}})\longrightarrow\Bbbk,
$$
$\omega(c_{1}\otimes c_{2}\otimes\cdots\otimes c_{k},\ \ c'_{1}\otimes c'_{2}
\otimes\cdots\otimes c'_{k})=\prod_{i=1}^{k}\omega(c_{i}, c'_{i})$,
for any $c_{1}, c_{2},\cdots, c_{k}, c'_{1}, c'_{2},\cdots, c'_{k}\in C$.
Then $\omega(-,-)$ on $C\otimes C\otimes\cdots\otimes C$ is also a
nondegenerate bilinear form.

\begin{lem}\label{lem:Zinb-dual}
Let $(C, \star, \omega)$ be a quadratic Zinbiel algebra. Define a linear map
$\theta_{\omega}: C\rightarrow C\otimes C$ by $\omega(\theta_{\omega}
(c_{1}),\; c_{2}\otimes c_{3})=\omega(c_{1},\; c_{2}\star c_{3})$ for any
$c_{1}, c_{2}, c_{3}\in C$. Then $(C, \theta_{\omega})$ is a Zinbiel coalgebra.
\end{lem}

\begin{proof}
It is straightforward.
\end{proof}

\begin{ex}\label{ex:qu-zib}
Let $(C={\rm span}_{\Bbbk}\{e_{1}, e_{2}, e_{3}, e_{4}\}, \star)$ be the $4$-dimensional
Zinbiel algebra given by $e_{1}\star e_{1}=e_{2}$, $e_{4}\star e_{4}=e_{3}$,
$e_{1}\star e_{4}=2e_{3}-e_{2}$ and $e_{4}\star e_{1}=2e_{2}-e_{3}$. If we define
a bilinear form $\omega(-,-)$ on $(C, \star)$ by $\omega(e_{3}, e_{1})=\omega(e_{4}, e_{2})
=1=-\omega(e_{1}, e_{3})=-\omega(e_{2}, e_{4})$ and all others are zero, then
$(C, \star, \omega)$ is a quadratic Zinbiel algebra. Then by Lemma \ref{lem:Zinb-dual},
we get a Zinbiel coalgebra $(C, \theta_{\omega})$, where $\theta_{\omega}(e_{1})
=-e_{2}\otimes e_{2}-e_{2}\otimes e_{3}+2e_{3}\otimes e_{2}$, $\theta_{\omega}(e_{4})
=e_{3}\otimes e_{3}-2e_{2}\otimes e_{3}+e_{3}\otimes e_{2}$ and $\theta_{\omega}(e_{2})
=\theta_{\omega}(e_{3})=0$.
\end{ex}

Now we can construct a differential infinitesimal bialgebra by the tensor product of
a differential perm bialgebra and a quadratic Zinbiel algebra.

\begin{thm}\label{thm:permbia-ASI}
Let $(P, \diamond, \vartheta, \fd, \partial)$ be a differential perm bialgebra,
$(C, \star, \omega)$ be a quadratic Zinbiel algebra and $(P\otimes C, \cdot,
\hat{\fd}, \hat{\partial})$ be the induced admissible differential algebra by
$(P, \diamond, \fd, \partial)$ and $(C, \star)$ in Propositions \ref{pro:tensor-ad}.
Define a linear map $\theta: P\otimes C\rightarrow(P\otimes C)\otimes(P\otimes C)$ by
\begin{align*}
\Delta(p\otimes c)&=\vartheta(p)\bullet\theta_{\omega}(c)+\tau(\vartheta(p))
\bullet\tau(\theta_{\omega}(c))\\
&=\sum_{(c)}\sum_{(p)}\Big((p_{(1)}\otimes c_{(1)})\otimes(p_{(2)}\otimes c_{(2)})
+(p_{(2)}\otimes c_{(2)})\otimes(p_{(1)}\otimes c_{(1)})\Big),
\end{align*}
for any $p\in P$ and $c\in C$, where $\theta_{\omega}(c)=\sum_{(c)}c_{(1)}\otimes c_{(2)}$
and $\vartheta(p)=\sum_{(p)}p_{(1)}\otimes p_{(2)}$ in the Sweedler notation.
Then $(P\otimes C, \cdot, \Delta, \hat{\fd}, \hat{\partial})$ is a differential
infinitesimal bialgebra, which is called {\bf the differential infinitesimal bialgebra induced
from $(P, \diamond, \vartheta, \fd, \partial)$ by $(C, \star, \omega)$}.
\end{thm}

\begin{proof}
By Proposition \ref{pro:ind-coadass} and Lemma \ref{lem:Zinb-dual}, we get
that $(P\otimes C, \Delta, \hat{\fd}, \hat{\partial})$ is an admissible
codifferential coalgebra. Following, we show that $(P\otimes C, \cdot, \Delta)$
is an infinitesimal bialgebra. In fact, for any $e, f, c_{1}, c_{2}\in C$, we have
\begin{align*}
& \omega(\theta_{\omega}(c_{1}\star c_{2}),\; e\otimes f)
=-\omega(c_{1},\; (e\star f)\star c_{2}),\\
& \omega(\theta_{\omega}(c_{2}\star c_{1}),\; e\otimes f)
=\omega(c_{1},\; (c_{2}\star e)\star f+(e\star c_{2})\star f+(e\star f)\star c_{2}),\\
& \omega(\tau(\theta_{\omega}(c_{1}\star c_{2})),\; e\otimes f)
=-\omega(c_{1},\; (f\star e)\star c_{2}),\\
& \omega(\tau(\theta_{\omega}(c_{2}\star c_{1})),\; e\otimes f)
=\omega(c_{1},\; (c_{2}\star f)\star e+(f\star c_{2})\star e+(f\star e)\star c_{2}),\\
& \omega((\id\otimes\tilde{\fl}_{C}(c_{1}))(\theta_{\omega}(c_{2})),\; e\otimes f)
=-\omega(c_{1},\; (c_{2}\star f)\star e+(f\star c_{2})\star e+(e\star f)\star c_{2}
+(f\star e)\star c_{2}),\\
& \omega(\tau((\tilde{\fl}_{C}(c_{1})\otimes\id)(\theta_{\omega}(c_{2}))),\; e\otimes f)
=\omega(c_{1},\; (c_{2}\star f)\star e+(f\star c_{2})\star e),\\
& \omega((\id\otimes\tilde{\fr}_{C}(c_{1}))(\theta_{\omega}(c_{2})),\; e\otimes f)
=\omega(c_{1},\; (c_{2}\star f)\star e+(f\star c_{2})\star e+(e\star f)\star c_{2}\\[-1mm]
&\qquad\qquad\qquad\qquad\qquad\qquad\qquad\qquad
+(f\star e)\star c_{2}+(c_{2}\star e)\star f+(e\star c_{2})\star f),\\
& \omega(\tau((\tilde{\fr}_{C}(c_{1})\otimes\id)(\theta_{\omega}(c_{2}))),\; e\otimes f)
=-\omega(c_{1},\; (c_{2}\star f)\star e+(f\star c_{2})\star e+(c_{2}\star e)\star f),\\
& \omega((\tilde{\fl}_{C}(c_{2})\otimes\id)(\theta_{\omega}(c_{1})),\; e\otimes f)
=\omega(c_{1},\; (c_{2}\star e)\star f+(e\star c_{2})\star f),\\
& \omega(\tau((\id\otimes\tilde{\fl}_{C}(c_{2}))(\theta_{\omega}(c_{1}))),\; e\otimes f)
=\omega(c_{1},\; (c_{2}\star f)\star e+(f\star c_{2})\star e+(e\star f)\star c_{2}
+(f\star e)\star c_{2}),\\
& \omega((\tilde{\fr}_{C}(c_{2})\otimes\id)(\theta_{\omega}(c_{1})),\; e\otimes f)
=-\omega(c_{1},\; (e\star c_{2})\star f),\\
& \omega(\tau((\id\otimes\tilde{\fr}_{C}(c_{2}))(\theta_{\omega}(c_{1}))),\; e\otimes f)
=-\omega(c_{1},\; (e\star f)\star c_{2}+(f\star e)\star c_{2}).
\end{align*}
If we denote $\Phi_{c_{1}, c_{2}}, \Phi'_{c_{1}, c_{2}}, \Psi_{c_{1}, c_{2}},
\Psi'_{c_{1}, c_{2}}, \Omega_{c_{1}, c_{2}}\in C\otimes C$ for any $c_{1}, c_{2}\in C$ by
\begin{align*}
&\omega(\Phi_{c_{1}, c_{2}},\; e\otimes f)=\omega(c_{1},\; (e\star f)\star c_{2}),\qquad\qquad
\omega(\Phi'_{c_{1}, c_{2}},\; e\otimes f)=\omega(c_{1},\; (f\star e)\star c_{2}),\\
&\omega(\Psi_{c_{1}, c_{2}},\; e\otimes f)=\omega(c_{1},\; (c_{2}\star e)\star f),\qquad\qquad
\omega(\Psi'_{c_{1}, c_{2}},\; e\otimes f)=\omega(c_{1},\; (e\star c_{2})\star f),\\
&\omega(\Omega_{c_{1}, c_{2}},\; e\otimes f)=\omega(c_{1},\; (c_{2}\star f)\star e
+(f\star c_{2})\star e),
\end{align*}
then we obtain
\begin{align*}
& \Delta((p_{(1)}\otimes c_{(1)})\cdot(p_{(2)}\otimes c_{(2)}))
-(\id\otimes\fu_{P\otimes C}(p_{(1)}\otimes c_{(1)}))(\Delta(p_{(2)}\otimes c_{(2)}))\\[-1mm]
&\ \ -(\fu_{P\otimes C}(p_{(2)}\otimes c_{(2)})\otimes\id)(\Delta(p_{(1)}\otimes c_{(1)}))\\
=&\; \vartheta(p_{1}\diamond p_{2})\bullet\theta_{\omega}(c_{1}\star c_{2})
+\tau(\vartheta(p_{1}\diamond p_{2}))\bullet\tau(\theta_{\omega}(c_{1}\star c_{2}))\\[-1mm]
&\ \ +\vartheta(p_{2}\diamond p_{1})\bullet\theta_{\omega}(c_{2}\star c_{1})
+\tau(\vartheta(p_{2}\diamond p_{1}))\bullet\tau(\theta_{\omega}(c_{2}\star c_{1}))\\[-1mm]
&\ \ -(\id\otimes\fl_{P}(p_{1}))(\vartheta(p_{2}))\bullet
(\id\otimes\tilde{\fl}_{C}(c_{1}))(\theta_{\omega}(c_{2}))
-\tau((\fl_{P}(p_{1})\otimes\id)(\vartheta(p_{2})))\bullet
\tau((\tilde{\fl}_{C}(c_{1})\otimes\id)(\theta_{\omega}(c_{2})))\\[-1mm]
&\ \ -(\id\otimes\fr_{P}(p_{1}))(\vartheta(p_{2}))\bullet
(\id\otimes\tilde{\fr}_{C}(c_{1}))(\theta_{\omega}(c_{2}))
-\tau((\fr_{P}(p_{1})\otimes\id)(\vartheta(p_{2})))\bullet
\tau((\tilde{\fr}_{C}(c_{1})\otimes\id)(\theta_{\omega}(c_{2})))\\[-1mm]
&\ \ -(\fl_{P}(p_{2})\otimes\id)(\vartheta(p_{1}))\bullet
(\tilde{\fl}_{C}(c_{2})\otimes\id)(\theta_{\omega}(c_{1}))
-\tau(\id\otimes(\fl_{P}(p_{2}))(\vartheta(p_{1})))\bullet
\tau((\id\otimes\tilde{\fl}_{C}(c_{2}))(\theta_{\omega}(c_{1})))\\[-1mm]
&\ \ -(\fr_{P}(p_{2})\otimes\id)(\vartheta(p_{1}))\bullet
(\tilde{\fr}_{C}(c_{2})\otimes\id)(\theta_{\omega}(c_{1}))
-\tau(\id\otimes(\fr_{P}(p_{2}))(\vartheta(p_{1})))\bullet
\tau((\id\otimes\tilde{\fr}_{C}(c_{2}))(\theta_{\omega}(c_{1})))\\[-1mm]
=&\; \Big(\vartheta(p_{2}\diamond p_{1})-\vartheta(p_{1}\diamond p_{2})
+(\id\otimes\fl_{P}(p_{1}))(\vartheta(p_{2}))\\[-2mm]
&\quad -(\id\otimes\fr_{P}(p_{1}))(\vartheta(p_{2}))
-\tau(\id\otimes(\fl_{P}(p_{2}))(\vartheta(p_{1})))
+\tau(\id\otimes(\fr_{P}(p_{2}))(\vartheta(p_{1})))\Big)\bullet\Phi_{c_{1}, c_{2}}\\
&\; +\Big(\tau(\vartheta(p_{2}\diamond p_{1}))
-\tau(\vartheta(p_{1}\diamond p_{2}))
+(\id\otimes\fl_{P}(p_{1}))(\vartheta(p_{2}))\\[-2mm]
&\quad -(\id\otimes\fr_{P}(p_{1}))(\vartheta(p_{2}))
-\tau(\id\otimes(\fl_{P}(p_{2}))(\vartheta(p_{1})))
+\tau(\id\otimes(\fr_{P}(p_{2}))(\vartheta(p_{1})))\Big)\bullet\Phi'_{c_{1}, c_{2}}\\
&\; +\Big(\vartheta(p_{2}\diamond p_{1})-(\id\otimes\fr_{P}(p_{1}))(\vartheta(p_{2}))
+\tau((\fr_{P}(p_{1})\otimes\id)(\vartheta(p_{2})))
-(\fl_{P}(p_{2})\otimes\id)(\vartheta(p_{1}))\Big)\bullet\Psi_{c_{1}, c_{2}}\\
&\; +\Big(\vartheta(p_{2}\diamond p_{1})-(\id\otimes\fr_{P}(p_{1}))(\vartheta(p_{2}))
-(\fl_{P}(p_{2})\otimes\id)(\vartheta(p_{1}))
+(\fr_{P}(p_{2})\otimes\id)(\vartheta(p_{1}))\Big)\bullet\Psi'_{c_{1}, c_{2}}\\
&\; +\Big(\tau(\vartheta(p_{2}\diamond p_{1}))
+(\id\otimes\fl_{P}(p_{1}))(\vartheta(p_{2}))
-\tau((\fl_{P}(p_{1})\otimes\id)(\vartheta(p_{2})))\\[-2mm]
&\quad -(\id\otimes\fr_{P}(p_{1}))(\vartheta(p_{2}))
+\tau((\fr_{P}(p_{1})\otimes\id)(\vartheta(p_{2})))
-\tau(\id\otimes(\fl_{P}(p_{2}))(\vartheta(p_{1})))\Big)\bullet\Omega_{c_{1}, c_{2}}.
\end{align*}
Suppose $(P, \diamond, \vartheta, \fd, \partial)$ is a differential perm bialgebra.
Then by Eqs. \eqref{bialg2} and \eqref{bialg3}, we get that the terms for
$\Phi_{c_{1}, c_{2}}$ and $\Phi'_{c_{1}, c_{2}}$ in the above equation are both zero.
By Eq. \eqref{bialg2}, we get that the term for $\Psi'_{c_{1}, c_{2}}$ in the above
equation is zero. By Eqs. \eqref{bialg2} and \eqref{bialg1}, we get that the term
for $\Psi_{c_{1}, c_{2}}$ in the above equation is zero. By Eq. \eqref{bialg3}, we
get that the term for $\Omega_{c_{1}, c_{2}}$ in the above equation is zero.
That is to say, $(P\otimes C, \cdot, \Delta)$ is an infinitesimal bialgebra and
$(P\otimes C, \cdot, \Delta, \hat{\fd}, \hat{\partial})$ is a differential
infinitesimal bialgebra if $(P, \diamond, \vartheta, \fd, \partial)$ is a
differential perm bialgebra.
\end{proof}

Let $(A, \cdot, \fd, \partial)$ be an admissible differential algebra. If there exists
an element $r\in A\otimes A$ such that $(A, \cdot, \Delta_{r}, \fd, \partial)$ is a
differential infinitesimal bialgebra, where $\Delta_{r}: A\rightarrow A\otimes A$ is given by
\begin{align}
\Delta_{r}(a)=\big(\id\otimes\fu_{A}(a)-\fu_{A}(a)\otimes\id\big)(r),   \label{cobass}
\end{align}
for any $a\in A$, then $(A, \cdot, \Delta_{r}, \fd, \partial)$ is called a {\bf coboundary
differential infinitesimal bialgebra}. An element $r=\sum_{i}x_{i}\otimes y_{i}\in
A\otimes A$ is said to be {\bf skew-symmetric} if $r=-\tau(r)$; $r$ is said to be {\bf
ass-invariant} if $(\id\otimes\fu_{A}(a)-\fu_{A}(a)\otimes\id)(r)=0$. The equations
\begin{align*}
&\qquad \mathbf{A}_{r}:=r_{12}r_{13}+r_{13}r_{23}-r_{23}r_{12}=0,\\
& (\fd\otimes\id-\id\otimes\partial)(r)=0=(\id\otimes\fd-\partial\otimes\id)(r),
\end{align*}
is called the {\bf differential associative Yang-Baxter equation} (or $\DAYBE$) in
$(A, \cdot, \fd, \partial)$, where $r_{12}r_{13}=\sum_{i,j}(x_{i}x_{j})$ $\otimes y_{i}
\otimes y_{j}$, $r_{13}r_{23}=\sum_{i,j}x_{i}\otimes x_{j}\otimes(y_{i}y_{j})$ and
$r_{23}r_{12}=\sum_{i,j}x_{i}\otimes(x_{j}y_{i})\otimes y_{j}$.

\begin{pro}[\cite{CH}]\label{pro:diff-bia}
Let $(A, \cdot, \fd, \partial)$ be an admissible differential algebra,
$r\in A\otimes A$ and $\Delta_{r}: A\rightarrow A\otimes A$ defined by Eq. \eqref{cobass}.
\begin{enumerate}
\item[$(i)$] If $r$ is a solution of the $\DAYBE$ in $(A, \cdot, \fd, \partial)$
     and $r+\tau(r)$ is ass-invariant, then $(A, \cdot, \Delta_{r}, \fd, \partial)$
     is a differential infinitesimal bialgebra, which is called a {\bf quasi-triangular
     differential infinitesimal bialgebra} associated with $r$.
\item[$(ii)$] If $r$ is a skew-symmetric solution of the $\DAYBE$ in $(A, \cdot, \fd,
     \partial)$, then $(A, \cdot, \Delta_{r}, \fd, \partial)$ is a differential
     infinitesimal bialgebra, which is called a {\bf triangular differential infinitesimal
     bialgebra} associated with $r$.
\end{enumerate}
\end{pro}

A quasi-triangular differential infinitesimal bialgebra $(A, \cdot, \Delta_{r}, \fd,
\partial)$ is called a {\bf factorizable differential infinitesimal bialgebra} if
$\mathcal{I}=r^{\sharp}+\tau(r)^{\sharp}: A^{\ast}\rightarrow A$
is an isomorphism of vector spaces and $\mathcal{I}\circ\partial^{\ast}
=\fd\circ\mathcal{I}$ \cite{CH}.
We now consider the relationship between the solutions of the $\DPYBE$ in an admissible
differential perm algebra and the solutions of the $\DAYBE$ in the induced admissible
differential algebra. Let $(C, \star, \omega)$ be a quadratic Zinbiel algebra
and $\{e_{1}, e_{2},\cdots, e_{n}\}$ be a basis of $C$. Since $\omega(-,-)$ is
skew-symmetric nondegenerate, we get a basis $\{f_{1}, f_{2},\cdots, f_{n}\}$ of $C$,
which is called the dual basis of $\{e_{1}, e_{2},\cdots, e_{n}\}$ with respect to
$\omega(-,-)$, i.e., $\omega(f_{i}, e_{j})=\delta_{ij}$, where $\delta_{ij}$
is the Kronecker delta. For the solutions of the $\DAYBE$ in the induced admissible
differential algebra, we have

\begin{pro}\label{pro:DPYBE-DAYBE}
Let $(P, \diamond, \vartheta, \fd, \partial)$ be an admissible differential perm algebra
$(C, \star, \omega)$ be a quadratic Zinbiel algebra and $(P\otimes C, \cdot,
\hat{\fd}, \hat{\partial})$ be the admissible differential algebra induced
from $(P, \diamond, \fd, \partial)$ by $(C, \star)$. Suppose that $r=\sum_{i}x_{i}
\otimes y_{i}\in P\otimes P$ is a solution of the $\DPYBE$ in
$(P, \diamond, \fd, \partial)$ and $r-\tau(r)$ is perm-invariant. Then
\begin{align}
\widehat{r}=\sum_{i, j}(x_{i}\otimes e_{j})\otimes(y_{i}\otimes f_{j})
\in(P\otimes C)\otimes(P\otimes C)   \label{assrmax}
\end{align}
is a solution of the $\DAYBE$ in $(P\otimes C, \cdot, \hat{\fd}, \hat{\partial})$ and
$\widehat{r}+\tau(\widehat{r})$ is ass-invariant, where $\{e_{1}, e_{2},\cdots, e_{n}\}$
is a basis of $C$ and $\{f_{1}, f_{2},\cdots, f_{n}\}$ is the dual basis of $\{e_{1},
e_{2},\cdots, e_{n}\}$ with respect to $\omega(-,-)$.

In particular, $\widehat{r}$ is a skew-symmetric solution of the $\DAYBE$ in
$(P\otimes C, \cdot, \hat{\fd}, \hat{\partial})$ if $r$ is a symmetric solution of
the $\DPYBE$ in $(P, \diamond, \fd, \partial)$.
\end{pro}

\begin{proof}
First, since for any $1\leq s, u, v\leq n$, since
$$
\omega\Big(\sum_{k,l}e_{k}\otimes(f_{k}\star e_{l})\otimes f_{l},\ \
e_{s}\otimes e_{u}\otimes e_{v}\Big)
=\omega(e_{u},\ \ e_{s}\star e_{v})
=-\omega\Big(\sum_{k,l}(e_{k}\star e_{l})\otimes f_{k}\otimes f_{l},\ \
e_{s}\otimes e_{u}\otimes e_{v}\Big),
$$
and the nondegeneracy of $\omega(-,-)$, we get $\sum_{k,l}e_{k}\otimes(f_{k}\star e_{l})
\otimes f_{l}=-\sum_{k,l}(e_{k}\star e_{l})\otimes f_{k}\otimes f_{l}$. Similarly,
$\sum_{k,l}e_{k}\otimes(e_{l}\star f_{k})\otimes f_{l}=-\sum_{k,l}e_{k}\otimes e_{l}\otimes
(f_{l}\star f_{k})$ and $\sum_{k,l}e_{k}\otimes e_{l}\otimes(f_{k}\star f_{l})
=\sum_{k,l}(e_{l}\star e_{k})\otimes f_{k}\otimes f_{l}=\sum_{k,l}e_{k}\otimes(e_{l}
\star f_{k})\otimes f_{l}+\sum_{k,l}e_{k}\otimes(f_{k}\star e_{l})\otimes f_{l}$.
Thus, we obtain
\begin{align*}
&\; \widehat{r}_{12}\cdot\widehat{r}_{13}+\widehat{r}_{13}\cdot
\widehat{r}_{23}-\widehat{r}_{23}\cdot\widehat{r}_{12}\\
=&\;\sum_{i,j}\sum_{k,l}\Big(\big((x_{i}\diamond x_{j})\otimes y_{i}\otimes y_{j}\big)
\bullet\big((e_{k}\star e_{l})\otimes f_{k}\otimes f_{l}\big)
+\big((x_{j}\diamond x_{i})\otimes y_{i}\otimes y_{j}\big)
\bullet\big((e_{l}\star e_{k})\otimes f_{k}\otimes f_{l}\big)\\[-4mm]
&\qquad\quad +\big(x_{i}\otimes x_{j}\otimes(y_{i}\diamond y_{j})\big)\bullet
\big(e_{k}\otimes e_{l}\otimes(f_{k}\star f_{l})\big)
+\big(x_{i}\otimes x_{j}\otimes(y_{j}\diamond y_{i})\big)\bullet
\big(e_{k}\otimes e_{l}\otimes(f_{l}\star f_{k})\big)\\[-1mm]
&\qquad\quad -\big(x_{i}\otimes(x_{j}\diamond y_{i})\otimes y_{j}\big)
\bullet\big(e_{k}\otimes(e_{l}\star f_{k})\otimes f_{l}\big)
-\big(x_{i}\otimes(y_{i}\diamond x_{j})\otimes y_{j}\big)
\bullet\big(e_{k}\otimes(f_{k}\star e_{l})\otimes f_{l}\big)\Big)\\
=&\;\Big(\sum_{i,j}\big((x_{j}\diamond x_{i})\otimes y_{i}\otimes y_{j}
-(x_{i}\diamond x_{j})\otimes y_{i}\otimes y_{j}
+x_{i}\otimes x_{j}\otimes(y_{i}\diamond y_{j})\\[-5mm]
&\qquad\qquad\qquad\qquad\qquad\qquad\qquad\qquad\qquad
-x_{i}\otimes(y_{i}\diamond x_{j})\otimes y_{j}\big)\Big)\bullet
\Big(\sum_{k,l}e_{k}\otimes(f_{k}\star e_{l})\otimes f_{l}\Big)\\[-2mm]
&\ \ +\Big(\sum_{i,j}\big((x_{j}\diamond x_{i})\otimes y_{i}\otimes y_{j}
+x_{i}\otimes x_{j}\otimes(y_{i}\diamond y_{j})
-x_{i}\otimes x_{j}\otimes(y_{j}\diamond y_{i})\\[-5mm]
&\qquad\qquad\qquad\qquad\qquad\qquad\qquad\qquad\qquad
-x_{i}\otimes(x_{j}\diamond y_{i})\otimes y_{j}\big)\Big)\bullet
\Big(\sum_{k,l}e_{k}\otimes(e_{l}\star f_{k})\otimes f_{l}\Big).
\end{align*}
Since $r-\tau(r)$ is perm-invariant, we get $\sum_{i,j}\big(x_{i}\otimes(y_{i}\diamond x_{j})
\otimes y_{j}+(x_{i}\diamond x_{j})\otimes y_{i}\otimes y_{j}-(x_{j}\diamond x_{i})
\otimes y_{i}\otimes y_{j}\big)=\sum_{i,j}\big(y_{i}\otimes(x_{i}\diamond x_{j})\otimes y_{j}
+(y_{i}\diamond x_{j})\otimes x_{i}\otimes y_{j}-(x_{j}\diamond y_{i})\otimes x_{i}
\otimes y_{j}\big)$ and $\sum_{i,j}\big(x_{i}\otimes x_{j}\otimes(y_{i}\diamond y_{j})
+(x_{i}\diamond y_{j})\otimes x_{j}\otimes y_{i}-(y_{j}\diamond x_{i})\otimes x_{j}
\otimes y_{i}\big)=\sum_{i,j}\big(y_{i}\otimes x_{j}\otimes(x_{i}\diamond y_{j})
+(y_{i}\diamond y_{j})\otimes x_{j}\otimes x_{i}-(y_{j}\diamond y_{i})\otimes x_{j}
\otimes x_{i}\big)$, and so that
\begin{align*}
\mathbf{P}_{r}
&=\sum_{i,j}\Big(x_{i}\otimes(y_{i}\diamond x_{j})\otimes y_{j}
-x_{i}\otimes x_{j}\otimes(y_{i}\diamond y_{j})
+(x_{i}\diamond x_{j})\otimes y_{i}\otimes y_{j}
-(x_{i}\diamond x_{j})\otimes y_{j}\otimes y_{i}\Big)\\[-2mm]
&=\sum_{i,j}\Big(y_{i}\otimes(x_{i}\diamond x_{j})\otimes y_{j}
+(y_{i}\diamond x_{j})\otimes x_{i}\otimes y_{j}
-(x_{j}\diamond y_{i})\otimes x_{i}\otimes y_{j}
-x_{i}\otimes x_{j}\otimes(y_{i}\diamond y_{j})\Big)\\[-2mm]
&=\sum_{i,j}\Big(y_{i}\otimes(x_{i}\diamond x_{j})\otimes y_{j}
-y_{i}\otimes x_{j}\otimes(x_{i}\diamond y_{j})
-(y_{i}\diamond y_{j})\otimes x_{j}\otimes x_{i}
+(y_{j}\diamond y_{i})\otimes x_{j}\otimes x_{i}\Big).
\end{align*}
Thus, we obtain
$$
\widehat{r}_{12}\cdot\widehat{r}_{13}+\widehat{r}_{13}\cdot
\widehat{r}_{23}-\widehat{r}_{23}\cdot\widehat{r}_{12}
=\mathbf{P}_{r}\bullet\Big(\sum_{k,l}e_{k}\otimes(f_{k}\star e_{l})\otimes f_{l}\Big)
+(\id\otimes\tau)((\tau\otimes\id)(\mathbf{P}_{r}))\bullet
\Big(\sum_{k,l}e_{k}\otimes(e_{l}\star f_{k})\otimes f_{l}\Big).
$$
That is, $\mathbf{A}_{\widehat{r}}=0$ if $\mathbf{P}_{r}=0$ and $r-\tau(r)$ is perm-invariant.
Moreover, note that $(\hat{\fd}\otimes\id-\id\otimes\hat{\partial})(\widehat{r})
=(\fd\otimes\id-\id\otimes\partial)(r)\bullet(\sum_{j}e_{j}\otimes f_{j})$ and
$(\hat{\partial}\otimes\id-\id\otimes\hat{\fd})(\widehat{r})
=(\partial\otimes\id-\id\otimes\fd)(r)\bullet(\sum_{j}e_{j}\otimes f_{j})$.
We get that $\widehat{r}$ is a solution of the $\DAYBE$ in $(P\otimes C, \cdot,
\hat{\fd}, \hat{\partial})$ if $r$ is a solution of the $\DPYBE$ in
$(P, \diamond, \fd, \partial)$ and $r-\tau(r)$ is perm-invariant.

Second, for any $1\leq s, t\leq n$ and $c\in C$, since $\omega(-,-)$ is nondegenerate and
$$
\omega\Big(\sum_{j}e_{j}\otimes(c\star f_{j}),\ \ e_{s}\otimes e_{t}\Big)
=\omega(c,\ \ e_{t}\star e_{s})
=-\omega\Big(\sum_{j}f_{j}\otimes(c\star e_{j}),\ \ e_{s}\otimes e_{t}\Big),
$$
we obtain $\sum_{j}e_{j}\otimes(c\star f_{j})=-\sum_{j}f_{j}\otimes(c\star e_{j})$.
Similarly, $\sum_{j}(c\star f_{j})\otimes e_{j}=-\sum_{j}(c\star e_{j})\otimes f_{j}$
and $\sum_{j}f_{j}\otimes(e_{j}\star c)=\sum_{j}(e_{j}\star c)\otimes f_{j}
=-\sum_{j}e_{j}\otimes(f_{j}\star c)=-\sum_{j}(f_{j}\star c)\otimes e_{j}=
\sum_{j}e_{j}\otimes(c\star f_{j})+\sum_{j}(c\star f_{j})\otimes e_{j}$.
Thus, for any $c\in C$ and $p\in P$, we have
\begin{align*}
&\;(\id\otimes\fu_{P\otimes C}(p\otimes c)-\fu_{P\otimes C}(p\otimes c)\otimes\id)
(\widehat{r}+\tau(\widehat{r}))\\
=&\; \sum_{i,j}\Big(
\big(x_{i}\otimes(p\diamond y_{i})\big)\bullet\big(e_{j}\otimes(c\star f_{j})\big)
+\big(x_{i}\otimes(y_{i}\diamond p)\big)\bullet\big(e_{j}\otimes(f_{j}\star c)\big)\\[-4mm]
&\qquad\quad
-\big((p\diamond x_{i})\otimes y_{i}\big)\bullet\big((c\star e_{j})\otimes f_{j}\big)
-\big((x_{i}\diamond p)\otimes y_{i}\big)\bullet\big((e_{j}\star c)\otimes f_{j}\big)\\
&\qquad\quad
+\big(y_{i}\otimes(p\diamond x_{i})\big)\bullet\big(f_{j}\otimes(c\star e_{j})\big)
+\big(y_{i}\otimes(x_{i}\diamond p)\big)\bullet\big(f_{j}\otimes(e_{j}\star c)\big)\\[-1mm]
&\qquad\quad
-\big((p\diamond y_{i})\otimes x_{i}\big)\bullet\big((c\star f_{j})\otimes e_{j}\big)
-\big((y_{i}\diamond p)\otimes x_{i}\big)\bullet\big((f_{j}\star c)\otimes e_{j}\big)\Big)\\
=&\;\Big(\sum_{i}\big((p\diamond x_{i})\otimes y_{i}-x_{i}\otimes(y_{i}\diamond p)
-(x_{i}\diamond p)\otimes y_{i}+y_{i}\otimes(x_{i}\diamond p)\\[-4mm]
&\qquad\qquad\qquad\qquad\qquad\qquad
-(p\diamond y_{i})\otimes x_{i}+(y_{i}\diamond p)\otimes x_{i}\big)
\Big)\bullet\Big(\sum_{j}(c\star f_{j})\otimes e_{j}\Big)\\[-2mm]
&\; +\Big(\sum_{i}\big(x_{i}\otimes(p\diamond y_{i})-x_{i}\otimes(y_{i}\diamond p)
-(x_{i}\diamond p)\otimes y_{i}-y_{i}\otimes(p\diamond x_{i})\\[-4mm]
&\qquad\qquad\qquad\qquad\qquad\qquad
+y_{i}\otimes(x_{i}\diamond p)+(y_{i}\diamond p)\otimes x_{i}\big)
\Big)\bullet\Big(\sum_{j}e_{j}\otimes(c\star f_{j})\Big).
\end{align*}
If $r-\tau(r)$ is perm-invariant, i.e., $\sum_{i}\big((p\diamond x_{i})\otimes y_{i}
-x_{i}\otimes(y_{i}\diamond p)-(x_{i}\diamond p)\otimes y_{i}+y_{i}\otimes(x_{i}\diamond p)
-(p\diamond y_{i})\otimes x_{i}+(y_{i}\diamond p)\otimes x_{i}\big)=0$ for any $p\in P$,
then $(\id\otimes\fu_{P\otimes C}(p\otimes c)-\fu_{P\otimes C}(p\otimes c)\otimes\id)
(\widehat{r}+\tau(\widehat{r}))=0$, and so that $\widehat{r}+\tau(\widehat{r})$
is ass-invariant.

Finally, for any $s, t\in\{1, 2,\cdots, n\}$, note that
$$
\omega\Big(e_{s}\otimes e_{t},\; \sum_{j}e_{j}\otimes f_{j}\Big)
=-\omega(e_{s}, e_{t})
=-\omega\Big(e_{s}\otimes e_{t},\; \sum_{j}f_{j}\otimes e_{j}\Big).
$$
The nondegeneracy of $\omega(-,-)$ yields that $\sum_{j}e_{j}\otimes f_{j}
=-\sum_{j}f_{j}\otimes e_{j}$. Since $r$ is symmetric, we get $\widehat{r}$
is skew-symmetric. The proof is finished.
\end{proof}

We can now use the relationship between the solution of the $\DPYBE$ in a differential
perm algebra and the solution of the $\DAYBE$ in the induced differential algebra given
in Proposition \ref{pro:DPYBE-DAYBE} to provide another main conclusion of this section.

\begin{thm}\label{thm:indu-asssdibia}
Let $(P, \diamond, \vartheta, \fd, \partial)$ be a differential perm bialgebra,
$(C, \star, \omega)$ be a quadratic Zinbiel algebra and $(P\otimes C, \cdot, \Delta,
\hat{\fd}, \hat{\partial})$ be the induced differential infinitesimal bialgebra from
$(P, \diamond, \vartheta, \fd, \partial)$ by $(C, \star, \omega)$. If $\vartheta
=\vartheta_{r}$ is defined by Eq. \eqref{cobo} for some $r\in P\otimes P$ and
$r-\tau(r)$ is perm-invariant, then $(P\otimes C, \cdot, \Delta, \hat{\fd},
\hat{\partial})=(P\otimes C, \cdot, \Delta_{\widehat{r}}, \hat{\fd}, \hat{\partial})$
as differential infinitesimal bialgebras, where $\Delta_{\widehat{r}}$ is defined by
Eq. \eqref{cobass} and $\widehat{r}$ is defined by Eq. \eqref{assrmax}. Therefore, we obtain
\begin{enumerate}\itemsep=0pt
\item[$(i)$] $(P\otimes C, \cdot, \Delta, \hat{\fd}, \hat{\partial})$ is
     quasi-triangular if $(P, \diamond, \vartheta, \fd, \partial)$ is quasi-triangular;
\item[$(ii)$] $(P\otimes C, \cdot, \Delta, \hat{\fd}, \hat{\partial})$ is
     triangular if $(P, \diamond, \vartheta, \fd, \partial)$ is triangular;
\item[$(iii)$] $(P\otimes C, \cdot, \Delta, \hat{\fd}, \hat{\partial})$ is
     factorizable if $(P, \diamond, \vartheta, \fd, \partial)$ is factorizable.
\end{enumerate}
\end{thm}

\begin{proof}
Let $r=\sum_{i}x_{i}\otimes y_{i}\in P\otimes P$. First, for any $c\in C$ and
$p\in P$, we have
\begin{align*}
\Delta(p\otimes c)&=\Big(\sum_{i}\big((p\diamond x_{i}-x_{i}\diamond p)\otimes y_{i}
-x_{i}\otimes(y_{i}\diamond p)\big)\Big)
\bullet\Big(\sum_{(c)}c_{(1)}\otimes c_{(2)}\Big)\\[-2mm]
&\quad +\Big(\sum_{i}\big(y_{i}\otimes(p\diamond x_{i}-x_{i}\diamond p)
-(y_{i}\diamond p)\otimes x_{i}\big)\Big)
\bullet\Big(\sum_{(c)}c_{(2)}\otimes c_{(1)}\Big),
\end{align*}
where $\theta_{\omega}(c)=\sum_{(c)}c_{(1)}\otimes c_{(2)}$ and $\vartheta_{r}(p)=
((\fl_{P}-\fr_{P})(p)\otimes\id-\id\otimes\fr_{P}(p))(r)=\sum_{i}\big((p\diamond x_{i}
-x_{i}\diamond p)\otimes y_{i}-x_{i}\otimes(y_{i}\diamond p)\big)$. On the other hand,
\begin{align*}
\Delta_{\widehat{r}}(p\otimes c)&=\big(\id\otimes\fu_{P\otimes C}(p\otimes c)
-\fu_{P\otimes C}(p\otimes c)\otimes\id\big)(\widehat{r})\\
&=\sum_{i,j}\Big(\big(x_{i}\otimes(p\diamond y_{i})\big)\bullet
\big(e_{j}\otimes(c\star f_{j})\big)+\big(x_{i}\otimes(y_{i}\diamond p)\big)\bullet
\big(e_{j}\otimes(f_{j}\star c)\big)\\[-5mm]
&\qquad\quad-\big((p\diamond x_{i})\otimes y_{i}\big)\bullet
\big((c\star e_{j})\otimes f_{j}\big)-\big((x_{i}\diamond p)\otimes y_{i}\big)\bullet
\big((e_{j}\star c)\otimes f_{j}\big)\Big),
\end{align*}
where $\{e_{1}, e_{2},\cdots, e_{n}\}$ is a basis of $C$ and $\{f_{1}, f_{2},\cdots,
f_{n}\}$ is the dual basis of $\{e_{1}, e_{2}, \cdots, e_{n}\}$ with respect to
$\omega(-,-)$. For any basis elements $e_{s}, e_{t}\in C$, since
$$
\omega\Big(\sum_{(c)}c_{(2)}\otimes c_{(1)},\ \ e_{s}\otimes e_{t}\Big)
=\omega(c,\; e_{t}\star e_{s})
=\omega\Big(\sum_{j}e_{j}\otimes(c\star f_{j}),\ \ e_{s}\otimes e_{t}\Big),
$$
we get that $\sum_{(c)}c_{(2)}\otimes c_{(1)}=\sum_{j}e_{j}\otimes(c\star f_{j})$.
Similarly, $\sum_{(c)}c_{(1)}\otimes c_{(2)}=-\sum_{j}(c\star e_{j})\otimes f_{j}$
and $\sum_{j}(e_{j}\star c)\otimes f_{j}=-\sum_{j}e_{j}\otimes(f_{j}\star c)=
\sum_{(c)}c_{(1)}\otimes c_{(2)}+\sum_{(c)}c_{(2)}\otimes c_{(1)}$. Thus, we obtain
\begin{align*}
\Delta_{\widehat{r}}(p\otimes c)&=\Big(\sum_{i}\big((p\diamond x_{i})\otimes y_{i}
-x_{i}\otimes(y_{i}\diamond p)-(x_{i}\diamond p)\otimes y_{i}\big)\Big)
\bullet\Big(\sum_{(c)}c_{(1)}\otimes c_{(2)}\Big)\\[-2mm]
&\quad +\Big(\sum_{i}\big(x_{i}\otimes(p\diamond y_{i})
-x_{i}\otimes(y_{i}\diamond p)-(x_{i}\diamond p)\otimes y_{i}\big)\Big)
\bullet\Big(\sum_{(c)}c_{(2)}\otimes c_{(1)}\Big).
\end{align*}
Since $r-\tau(r)$ is perm-invariant, i.e., $\sum_{i}\big(x_{i}\otimes(p\diamond y_{i})
-x_{i}\otimes(y_{i}\diamond p)-(x_{i}\diamond p)\otimes y_{i}\big)=\sum_{i}\big(y_{i}
\otimes(p\diamond x_{i}-x_{i}\diamond p)-(y_{i}\diamond p)\otimes x_{i}\big)$,
we get $\Delta=\Delta_{\widehat{r}}$, and so that $(P\otimes C, \cdot, \Delta,
\hat{\fd}, \hat{\partial})=(P\otimes C, \cdot, \Delta_{\widehat{r}}, \hat{\fd},
\hat{\partial})$ as differential infinitesimal bialgebras. Thus, Proposition
\ref{pro:DPYBE-DAYBE}, we get the conclusions $(i)$ and $(ii)$

Suppose now $(P, \diamond, \vartheta, \fd, \partial)$ is factorizable.
Then $\mathcal{I}=r^{\sharp}-\tau(r)^{\sharp}: P^{\ast}\rightarrow P$ is an
isomorphism of vector spaces and $\mathcal{I}\circ\partial^{\ast}=\fd\circ\mathcal{I}$.
We need to show that $\widehat{\mathcal{I}}:=\widehat{r}^{\sharp}
+\tau(\widehat{r})^{\sharp}: (P\otimes C)^{\ast}\rightarrow P\otimes C$ is an
isomorphism of vector spaces and $\widehat{\mathcal{I}}\circ\hat{\partial}^{\ast}
=\hat{\fd}\circ\widehat{\mathcal{I}}$. Denote $\kappa:=\sum_{j}e_{j}\otimes
f_{j}\in C\otimes C$. Define $\kappa^{\sharp}: C^{\ast}\rightarrow C$ by
$\langle\kappa^{\sharp}(\xi_{1}),\; \xi_{2}\rangle=\langle\kappa,\;
\xi_{1}\otimes\xi_{2}\rangle$, for any $\xi_{1}, \xi_{2}\in C^{\ast}$.
Then, one can check that $\kappa^{\sharp}$ is a linear isomorphism and
$\langle\kappa^{\sharp}(\xi_{1}),\; \xi_{2}\rangle=\langle\kappa,\;
\xi_{1}\otimes\xi_{2}\rangle=-\langle\kappa,\; \xi_{2}\otimes\xi_{1}\rangle
=\langle\kappa^{\sharp}(\xi_{2}),\; \xi_{1}\rangle$.
Therefore, for any $\xi_{1}, \xi_{2}\in C^{\ast}$ and $\eta_{1}, \eta_{2}\in P^{\ast}$,
\begin{align*}
\langle\widehat{r}^{\sharp}(\eta_{1}\otimes\xi_{1}),\; \eta_{2}\otimes\xi_{2}\rangle
&=\sum_{i,j}\langle(\eta_{1}\otimes\xi_{1})\otimes(\eta_{2}\otimes\xi_{2}),\ \
(x_{i}\otimes e_{j})\otimes(y_{i}\otimes f_{j})\rangle\\[-2mm]
&=\Big(\sum_{j}\langle\eta_{1}, x_{i}\rangle\langle\eta_{2}, y_{i}\rangle\Big)
\Big(\sum_{i}\langle\xi_{1}, e_{j}\rangle\langle\xi_{2}, f_{j}\rangle\Big)\\[-2mm]
&=\langle r^{\sharp}(\eta_{1}),\; \eta_{2}\rangle
\langle\kappa^{\sharp}(\xi_{1}),\; \xi_{2}\rangle\\
&=\langle r^{\sharp}(\eta_{1})\otimes\kappa^{\sharp}(\xi_{1}),\;
\eta_{2}\otimes\xi_{2}\rangle.
\end{align*}
That is, $\widehat{r}^{\sharp}=r^{\sharp}\otimes\kappa^{\sharp}$, Similarly,
$\tau(\widehat{r})^{\sharp}=-\tau(r)^{\sharp}\otimes\kappa^{\sharp}$. Thus,
$\widehat{\mathcal{I}}=\mathcal{I}\otimes\kappa^{\sharp}$ is an isomorphism of
vector spaces and $\widehat{\mathcal{I}}\circ\hat{\partial}^{\ast}=(\mathcal{I}
\circ\partial)\otimes\kappa^{\sharp}=(\fd\circ\mathcal{I})\otimes\kappa^{\sharp}
=\hat{\fd}\circ\widehat{\mathcal{I}}$. The proof is completed.
\end{proof}

By Theorem \ref{thm:indu-asssdibia}, we have the following commutative diagram:
$$
\xymatrix@C=2cm@R=0.6cm{
\txt{$r$ \\ {\tiny a solution of the $\DPYBE$ in $(P, \diamond,
\fd, \partial)$}\\ {\tiny such that $r-\tau(r)$
is perm-invariant}} \ar[r]^{\qquad{\rm Pro.}~\ref{pro:quasi-dpba}}
\ar[d]_{{\rm Pro.}~\ref{pro:DPYBE-DAYBE}} &
\txt{$(P, \diamond, \vartheta_{r}, \fd, \partial)$ \\
{\tiny a quasi-triangular differential}\\ {\tiny prem bialgebra}}
\ar[d]^{{\rm Thm.}~\ref{thm:indu-asssdibia}}\\
\txt{$\widehat{r}$ \\ {\tiny a solution of the $\DAYBE$ in $(P\otimes C, \cdot,
\hat{\fd}, \hat{\partial})$}\\
{\tiny such that $\widehat{r}+\tau(\widehat{r})$ is ass-invariant}}
\ar[r]^{\quad{\rm Pro.}~\ref{pro:diff-bia}} &
\txt{$(P\otimes C, \cdot, \Delta_{\widehat{r}}, \hat{\fd}, \hat{\partial})$ \\
{\tiny a quasi-triangular differential}\\ {\tiny infinitesimal bialgebra}} }
$$

\begin{ex}\label{ex:ind-triasbi}
Let $(P={\rm span}_{\Bbbk}\{x_{1}, x_{2}\}, \diamond, \vartheta, \fd, -\fd)$ be the
$2$-dimensional triangular differential perm bialgebra associated with $r=x_{2}
\otimes x_{2}$ given in Example \ref{ex:YBE-bialg} and $(C={\rm span}_{\Bbbk}\{e_{1},
e_{2}, e_{3}, e_{4}\}, \star, \omega)$ be the $4$-dimensional quadratic Zinbiel algebra
given in Example \ref{ex:qu-zib}. Then we get a $8$-dimensional differential infinitesimal
bialgebra $(P\otimes C, \cdot, \Delta, \hat{\fd}, -\hat{\fd})$, where the nonzero
commutative products and cocommutative products are given by
\begin{align*}
& (x_{1}\otimes e_{1})\cdot(x_{1}\otimes e_{1})=2x_{1}\otimes e_{2},\qquad
(x_{1}\otimes e_{4})\cdot(x_{1}\otimes e_{4})=2x_{1}\otimes e_{3},\qquad
(x_{1}\otimes e_{1})\cdot(x_{2}\otimes e_{1})=x_{2}\otimes e_{2},\\
&\qquad (x_{1}\otimes e_{1})\cdot(x_{2}\otimes e_{4})
=2x_{2}\otimes e_{3}-x_{2}\otimes e_{2},\qquad\quad
(x_{1}\otimes e_{1})\cdot(x_{1}\otimes e_{4})=x_{1}\otimes e_{2}+x_{1}\otimes e_{3},\\
&\qquad (x_{1}\otimes e_{4})\cdot(x_{2}\otimes e_{1})
=2x_{2}\otimes e_{2}-x_{2}\otimes e_{3},\qquad\quad
(x_{1}\otimes e_{4})\cdot(x_{2}\otimes e_{4})=x_{2}\otimes e_{3},\\
&\qquad \Delta(x_{1}\otimes e_{1})=(x_{2}\otimes e_{2})\otimes(x_{2}\otimes e_{3})
+(x_{2}\otimes e_{3})\otimes(x_{2}\otimes e_{2})
-2(x_{2}\otimes e_{2})\otimes(x_{2}\otimes e_{2}),\\
&\qquad \Delta(x_{1}\otimes e_{4})=2(x_{2}\otimes e_{3})\otimes(x_{2}\otimes e_{3})
-(x_{2}\otimes e_{2})\otimes(x_{2}\otimes e_{3})
-(x_{2}\otimes e_{3})\otimes(x_{2}\otimes e_{2}).
\end{align*}
Moreover, since the dual basis of $\{e_{1}, e_{2}, e_{3}, e_{4}\}$ with respect to
$\omega(-,-)$ is given by $\{e_{3}, e_{4}, -e_{1}, -e_{2}\}$, we get that
$$
\widehat{r}=(x_{2}\otimes e_{1})\otimes(x_{2}\otimes e_{3})
+(x_{2}\otimes e_{2})\otimes(x_{2}\otimes e_{4})
-(x_{2}\otimes e_{3})\otimes(x_{2}\otimes e_{1})
-(x_{2}\otimes e_{4})\otimes(x_{2}\otimes e_{2})
$$
is a skew-symmetric solution of the $\DAYBE$ in $(P\otimes C, \cdot, \hat{\fd}, -\hat{\fd})$.
By direct calculation, one can show that the triangular differential infinitesimal
bialgebra associated with $\widehat{r}$ is exactly the differential infinitesimal
bialgebra $(P\otimes C, \cdot, \Delta, \hat{\fd}, -\hat{\fd})$ given above.
\end{ex}

Let $(A, \cdot)$ be a commutative associative algebra, $V$ be a vector space and $\ku:
A\rightarrow\gl(V)$ be a linear map. Then $(V, \ku)$ is called a {\bf module over
$(A, \cdot)$} if for any $a_{1}, a_{2}\in A$, $\ku(a_{1})\circ\ku(a_{2})=\ku(a_{1}a_{2})$.
Clearly, $(A, \fu_{A})$ is a module over $(A, \cdot)$, which is called the {\bf regular
module}. Let $(A, \cdot, \fd)$ be a differential algebra, $(V, \ku)$ be a module over
$(A, \cdot)$ and $\partial: V\rightarrow V$ be a linear map. Then $(V, \ku, \partial)$
is called a {\bf module over $(A, \cdot, \fd)$} if for any $a\in A$ and $v\in V$,
$\partial(\fu(a)(v))=\fu(\fd(a))(v)+\fu(a)(\partial(v))$. The {\bf regular module and
coregular module} of $(A, \cdot, \fd)$ are given by $(A, \fu_{A}, \fd)$ and $(A^{\ast},
-\fu_{A}^{\ast}, -\fd^{\ast})$.  Let $(A, \cdot, \fd)$ be a differential algebra and
$(V, \ku, \partial)$ be a module over it. Recall that a linear map $T: V\rightarrow A$
is called an {\bf $\mathcal{O}$-operator of $(A, \cdot, \fd)$ associated to $(V, \ku,
\partial)$} if $T\circ\partial=\fd\circ T$ and for any $v_{1}, v_{2}\in V$,
$$
T(v_{1})\cdot T(v_{2})=T\big(\ku(T(v_{1}))(v_{2})+\ku(T(v_{2}))(v_{1})\big).
$$
For the solution of the $\DAYBE$ in an admissible differential algebra and
$\mathcal{O}$-operator, we have

\begin{pro}[\cite{LLB}]\label{pro:o-dass}
Let $(A, \cdot, \fd, \partial)$ be an admissible differential algebra and
$r\in A\otimes A$ be skew-symmetric. Then $r$ is a solution of the $\DAYBE$ in
$(A, \cdot, \fd, \partial)$ if and only if $r^{\sharp}: A^{\ast}\rightarrow A$ is an
$\mathcal{O}$-operator of $(A, \cdot, \fd)$ associated to the coregular module $(A^{\ast},
-\fu_{A}^{\ast}, \partial^{\ast})$.
\end{pro}

Give an admissible differential perm algebra, we have construct a differential algebra
by the tensor product of this admissible differential perm algebra and a Zinbiel algebra.
What is the relationship between the $\mathcal{O}$-operator of a differential perm algebra
and the $\mathcal{O}$-operator of the induced differential algebra? The proof process
of Theorem \ref{thm:indu-asssdibia} also provides the answer.

\begin{cor}\label{cor:o-dass-dperm}
Let $(A, \cdot, \fd, \partial)$ be an admissible differential algebra,
$(P, \diamond, \mathfrak{B})$ be a quadratic perm algebra and $(A\otimes P, \hat{\diamond},
\hat{\fd}, \hat{\partial})$ be the admissible differential perm algebra induced
from $(A, \cdot, \fd, \partial)$ by $(P, \diamond)$. Suppose $r$ is a skew-symmetric
solution of the $\DAYBE$ in $(A, \cdot, \fd, \partial)$, $\widehat{r}$ is a symmetric
solution of the $\DPYBE$ in $(A\otimes P, \hat{\diamond}, \hat{\fd}, \hat{\partial})$
induced by $r$ and $\kappa:=\sum_{j}e_{j}\otimes f_{j}\in P\otimes P$ is given in the
proof of Theorem \ref{thm:indu-asssdibia}. Then we have the following
commutative diagram:\\[-6mm]
$$
\xymatrix@C=3cm@R=0.5cm{
\txt{$r$ \\ {\tiny a symmetric solution} \\ {\tiny of the $\DPYBE$ in
$(P, \diamond, \fd, \partial)$}} \ar[r]^-{{\rm Pro.}~\ref{pro:o-dperm}}
\ar[d]_-{{\rm Pro.}~\ref{pro:DPYBE-DAYBE}}
& \txt{$r^{\sharp}$ \\ {\tiny an $\mathcal{O}$-operator of $(P, \diamond, \fd)$ } \\
{\tiny associated to coregular bimodule}}
\ar[d]^-{\mbox{$-\otimes\kappa^{\sharp}$}}\\
\txt{$\widehat{r}$ \\ {\tiny a skew-symmetric solution} \\ {\tiny of the $\DAYBE$ in
$(P\otimes C, \cdot, \hat{\fd}, \hat{\partial})$}}
\ar[r]^-{{\rm Pro.}~\ref{pro:o-dass}} &
\txt{$\widehat{r}^{\sharp}=r^{\sharp}\otimes\kappa^{\sharp}$\\
{\tiny an $\mathcal{O}$-operator of $(P\otimes C, \cdot, \hat{\fd})$} \\
{\tiny associated to coregular bimodule}}}
$$
\end{cor}

\subsection{From differential perm bialgebras to diNovikov bialgebras} \label{subsec:dperm-diN}
In this subsection, we consider the diNovikov bialgebra induced by a special differential
perm bialgebra $(P, \diamond, \vartheta, \fd, -\fd)$, where $\fd$ is a derivation on perm
bialgebra $(P, \diamond, \vartheta)$, i.e., $\vartheta\circ\fd=(\fd\otimes\id+
\id\otimes\fd)\circ\vartheta$ and $\fd(p_{1}\diamond p_{2})=\fd(p_{1})\diamond p_{2}
+p_{1}\diamond\fd(p_{2})$ for any $p_{1}, p_{2}\in P$.
Let $(D, \dashv, \vdash)$ be a diNovikov algebra, $V$ be a vector space and
$\kl_{\dashv}, \kr_{\dashv}, \kl_{\vdash}, \kr_{\vdash}: D\rightarrow\gl(V)$ be
four linear maps. If for any $d_{1}, d_{2}\in D$,
\begin{align*}
& \kl_{\dashv}(d_{1}\dashv d_{2})=\kr_{\dashv}(d_{2})\circ\kl_{\dashv}(d_{1}),\qquad
\kl_{\dashv}(d_{1}\dashv d_{2})-\kl_{\dashv}(d_{1})\circ\kl_{\dashv}(d_{2})
=\kl_{\dashv}(d_{2}\vdash d_{1})-\kl_{\vdash}(d_{2})\circ\kl_{\dashv}(d_{1}),\\
& \kr_{\dashv}(d_{2})\circ\kr_{\dashv}(d_{1})
=\kr_{\dashv}(d_{1})\circ\kr_{\dashv}(d_{2}),\qquad
\kr_{\dashv}(d_{2})\circ\kl_{\dashv}(d_{1})-\kl_{\dashv}(d_{1})\circ\kr_{\dashv}(d_{2})
=\kr_{\dashv}(d_{2})\circ\kr_{\vdash}(d_{1})-\kr_{\vdash}(d_{1}\dashv d_{2}),\\
& \kl_{\dashv}(d_{1})\circ\kl_{\dashv}(d_{2})
=\kl_{\dashv}(d_{1})\circ\kl_{\vdash}(d_{2}),\qquad
\kr_{\dashv}(d_{2})\circ\kr_{\dashv}(d_{1})-\kr_{\dashv}(d_{1}\dashv d_{2})
=\kr_{\dashv}(d_{2})\circ\kl_{\vdash}(d_{1})-\kl_{\vdash}(d_{1})\circ\kr_{\dashv}(d_{2}),\\
& \kl_{\dashv}(d_{1})\circ\kr_{\dashv}(d_{2})
=\kl_{\dashv}(d_{1})\circ\kr_{\vdash}(d_{2}),\qquad
\kl_{\vdash}(d_{1}\dashv d_{2})-\kl_{\vdash}(d_{1})\circ\kl_{\vdash}(d_{2})
=\kl_{\vdash}(d_{2}\vdash d_{1})-\kl_{\vdash}(d_{2})\circ\kl_{\vdash}(d_{1}),\\
& \kr_{\dashv}(d_{1}\dashv d_{2})=\kr_{\dashv}(d_{1}\vdash d_{2}),\qquad
\kr_{\vdash}(d_{2})\circ\kl_{\dashv}(d_{1})-\kl_{\vdash}(d_{1})\circ\kr_{\vdash}(d_{2})
=\kr_{\vdash}(d_{2})\circ\kr_{\vdash}(d_{1})-\kr_{\vdash}(d_{1}\vdash d_{2}),\\
& \kl_{\dashv}(d_{1}\vdash d_{2})=\kr_{\vdash}(d_{2})\circ\kl_{\vdash}(d_{1})
=\kr_{\vdash}(d_{2})\circ\kl_{\dashv}(d_{1}), \qquad
\kr_{\dashv}(d_{2})\circ\kl_{\vdash}(d_{1})
=\kl_{\vdash}(d_{1}\vdash d_{2})=\kl_{\vdash}(d_{1}\dashv d_{2}),\\
&\qquad\qquad\qquad\qquad \kr_{\dashv}(d_{2})\circ\kr_{\vdash}(d_{1})
=\kr_{\vdash}(d_{1})\circ\kr_{\vdash}(d_{2})
=\kr_{\vdash}(d_{1})\circ\kr_{\dashv}(d_{2}),
\end{align*}
then $(V, \kl_{\dashv}, \kr_{\dashv}, \kl_{\vdash}, \kr_{\vdash})$ is called a
{\bf representation of $(D, \dashv, \vdash)$}. In particular, if we define
$\fl_{\dashv}, \fr_{\dashv}, \fl_{\vdash}, \fr_{\vdash}: D\rightarrow\gl(D)$ by
$\fl_{\dashv}(d_{1})(d_{2})=d_{1}\dashv d_{2}=\fr_{\dashv}(d_{2})(d_{1})$,
$\fl_{\vdash}(d_{1})(d_{2})=d_{1}\vdash d_{2}=\fr_{\vdash}(d_{2})(d_{1})$ for any
$d_{1}, d_{2}\in D$, then $(D, \fl_{\dashv}, \fr_{\dashv}, \fl_{\vdash},
\fr_{\vdash})$ is a representation of $(D, \dashv, \vdash)$, which is called the
{\bf regular representation}. The {\bf coregular representation} of $(D, \dashv,
\vdash)$ is given by $(D^{\ast}, \fr_{\dashv}^{\ast}-\fl_{\dashv}^{\ast}
+\fl_{\vdash}^{\ast}-\fr_{\vdash}^{\ast}, -\fr_{\dashv}^{\ast}, \fl_{\vdash}^{\ast}
+\fr_{\dashv}^{\ast}, \fr_{\vdash}^{\ast}-\fr_{\dashv}^{\ast})$.
Next, we recall the definition of diNovikov coalgebra.

\begin{defi}\label{def:diNco}
A {\bf diNovikov coalgebra} is a triple $(D, \nu_{\dashv}, \nu_{\vdash})$, where
$\nu_{\dashv}, \nu_{\vdash}: D\rightarrow D\otimes D$ are linear maps satisfying
\begin{align}
&\qquad\qquad\qquad\quad (\id\otimes\nu_{\dashv})\circ\nu_{\dashv}
=(\id\otimes\nu_{\vdash})\circ\nu_{\dashv},           \label{ncdi1}\\
&\qquad\qquad\quad (\nu_{\dashv}\otimes\id)\circ\nu_{\dashv}
=(\id\otimes\tau)\circ(\nu_{\dashv}\otimes\id)\circ\nu_{\dashv},   \label{ncdi2}\\
& (\nu_{\vdash}\otimes\id)\circ\nu_{\dashv}
=(\id\otimes\tau)\circ(\nu_{\vdash}\otimes\id)\circ\nu_{\vdash}
=(\id\otimes\tau)\circ(\nu_{\dashv}\otimes\id)\circ\nu_{\vdash},    \label{ncdi3}
\end{align}
\begin{align}
& (\nu_{\dashv}\otimes \id-\id\otimes\nu_{\dashv})\circ\nu_{\dashv}
=(\tau\otimes\id)\circ\big((\nu_{\vdash}\otimes\id)\circ\nu_{\dashv}
-(\id\otimes\nu_{\dashv})\circ\nu_{\vdash}\big),      \label{ncdi4}\\
&\qquad (\nu_{\vdash}\otimes\id-\id\otimes\nu_{\vdash})\circ\nu_{\vdash}
=(\tau\otimes\id)\circ(\nu_{\vdash}\otimes\id-\id\otimes
\nu_{\vdash})\circ\nu_{\vdash}.            \label{ncdi5}
\end{align}
\end{defi}

In Proposition \ref{pro:dperm-diN}, we have shown that each differential perm algebra
induced a diNovikov algebra. Dual, for diNovikov coalgebras, we have the following conclusion.

\begin{pro}\label{pro:ind-codinov}
Let $(P, \vartheta, -\fd)$ be a codifferential perm coalgebra.
Define two coproducts $\nu_{\dashv,\fd}, \nu_{\vdash,\fd}: P\rightarrow P\otimes P$ by
\begin{align}
\nu_{\dashv,\fd}=\tau\circ(\fd\otimes\id)\circ\vartheta,  \qquad\qquad
\nu_{\vdash,\fd}=(\id\otimes\fd)\circ\vartheta.  \label{ind-codinov}
\end{align}
Then $(P, \nu_{\dashv,\fd}, \nu_{\vdash,\fd})$ is a diNovikov coalgebra, which is
called the {\bf diNovikov coalgebra induced from $(P, \vartheta, -\fd)$}.
\end{pro}

\begin{proof}
Since $(P, \vartheta, -\fd)$ be a codifferential perm coalgebra, i.e.,
$(\vartheta\otimes\id)\circ\vartheta=(\id\otimes\vartheta)\circ\vartheta
=((\tau\circ\vartheta)\otimes\id)\circ\vartheta$ and
$\vartheta\circ\fd=(\id\otimes\fd+\fd\otimes\id)\circ\vartheta$, by direct calculation,
we have
\begin{align*}
(\id\otimes\nu_{\dashv,\fd})\circ\nu_{\dashv,\fd}
=&\;(\id\otimes(\tau\circ(\fd\otimes\id)\circ\vartheta))\circ\tau
\circ(\fd\otimes\id)\circ\vartheta\\
=&\;(\id\otimes\tau)\circ(\tau\otimes\id)\circ(\id\otimes\tau)\circ(\fd\otimes\id\otimes\id)
\circ((\vartheta\circ\fd)\otimes\id)\circ\vartheta\\
=&\;(\tau\otimes\id)\circ(\id\otimes\tau)\circ(\fd\otimes\id\otimes\id)
\circ((\vartheta\circ\fd)\otimes\id)\circ\vartheta\\
=&\;(\id\otimes((\fd\otimes\id)\circ\vartheta))\circ\tau
\circ(\fd\otimes\id)\circ\vartheta\\
=&\;(\id\otimes\nu_{\vdash,\fd})\circ\nu_{\dashv,\fd}.
\end{align*}
That is, Eq. \eqref{ncdi1} holds. Similarly, Eqs. \eqref{ncdi2}-\eqref{ncdi5} hold.
Hence $(P, \nu_{\dashv,\fd}, \nu_{\vdash,\fd})$ is a diNovikov coalgebra.
\end{proof}

Recently, Xu, Bai and Hong provided the definition of diNovikov bialgebras and
studied the connection between diNovikov bialgebras and Leibniz conformal
bialgebras \cite{XBH}.

\begin{defi}[\cite{XBH}]\label{def:diNbi}
Let $(D, \dashv, \vdash)$ be a diNovikov algebra and $(D, \nu_{\dashv}, \nu_{\vdash})$
be a diNovikov coalgebra. If for all $d_{1}, d_{2}\in D$,
\begin{align}
&\qquad\qquad\qquad\quad (\id\otimes\fl_{\dashv}(d_{2}))(\nu_{\dashv}(d_{1}))
=(\fl_{\dashv}(d_{1})\otimes\id)(\tau(\nu_{\dashv}(d_{2}))),  \label{ndba1}\\
&\qquad\qquad\qquad (\fl_{\dashv}(d_{1})\otimes\id)(\nu_{\vdash}(d_{2}))
=-(\id\otimes(\fl_{\dashv}+\fr_{\vdash})(d_{2}))(\nu_{\dashv}(d_{1})), \label{ndba2}\\
&\qquad\qquad (\id\otimes(\fl_{\dashv}+\fr_{\vdash})(d_{1}))\big(\nu_{\dashv}(d_{2})
+\tau(\nu_{\vdash}(d_{2}))\big)=(\fr_{\vdash}(d_{2})\otimes\id)
(\nu_{\vdash}(d_{1})),         \label{ndba3}\\
&\quad \nu_{\dashv}(d_{1}\vdash d_{2})=(\fr_{\vdash}(d_{2})\otimes \id)(\nu_{\dashv}(d_{1})
-\nu_{\vdash}(d_{1}))+(\id\otimes(\fl_{\vdash}
+\fr_{\dashv})(d_{1}))(\nu_{\dashv}(d_{2})+\tau(\nu_{\vdash}(d_{2}))),  \label{ndba4}\\
&\quad \nu_{\vdash}(d_{1}\dashv d_{2})=(\fr_{\dashv}(d_{2})\otimes\id)(\nu_{\vdash}(d_{1}))
+(\id\otimes(\fl_{\dashv}+\fr_{\vdash})(d_{1}))\big((\id\otimes\id-\tau)
(\nu_{\vdash}(d_{2})-\nu_{\dashv}(d_{2}))\big),                   \label{ndba5}\\
&\; \nu_{\vdash}(d_{1}\vdash d_{2})=((\fr_{\vdash}-\fr_{\dashv})(d_{2})\otimes\id)
(\nu_{\dashv}(d_{1})-\nu_{\vdash}(d_{1}))+(\id\otimes(\fl_{\vdash}+\fr_{\dashv})(d_{1}))
(\tau(\nu_{\dashv}(d_{2}))+\nu_{\vdash}(d_{2})),       \label{ndba6}\\
& \nu_{\dashv}(d_{1}\dashv d_{2})=(\fr_{\dashv}(d_{2})\otimes\id)
\nu_{\dashv}(d_{1})+(\id\otimes(\fr_{\dashv}-\fl_{\dashv}+\fl_{\vdash}-\fr_{\vdash})(d_{1}))
\big((\id\otimes\id-\tau)(\nu_{\dashv}(d_{2})-\nu_{\vdash}(d_{2}))\big),   \label{ndba7}\\
&\quad((\fl_{\dashv}+\fr_{\vdash})(d_{2})\otimes \id)(\nu_{\dashv}(d_{1})-\nu_{\vdash}(d_{1}))
+(\id\otimes(\fr_{\dashv}-\fl_{\dashv}+\fl_{\vdash}-\fr_{\vdash})(d_{2}))
\big(\tau((\nu_{\dashv}(d_{1})-\nu_{\vdash}(d_{1})))\big) \label{ndba8}\\[-1mm]
&=((\fl_{\vdash}+\fr_{\dashv})(d_{1})\otimes\id)(\nu_{\dashv}(d_{2}))
-(\id\otimes(\fl_{\vdash}+\fr_{\dashv})(d_{1}))(\tau(\nu_{\vdash}(d_{2}))),\nonumber\\
&\quad\; (\id\otimes\fr_{\vdash}(d_{1}))\big((\id\otimes\id-\tau)(\nu_{\dashv}(d_{2})
-\nu_{\vdash}(d_{2}))\big)-(\fr_{\dashv}(d_{2})\otimes\id)(\nu_{\vdash}(d_{1})
+\tau(\nu_{\dashv}(d_{1})))                   \label{ndba9}\\[-1mm]
&\; =-(\id\otimes\fr_{\dashv}(d_{2}))(\nu_{\vdash}(d_{1})+\tau(\nu_{\dashv}(d_{1})))
+((\fr_{\vdash}-\fr_{\dashv})(d_{1})\otimes\id)\big((\id\otimes\id-\tau)
(\nu_{\vdash}(d_{2})-\nu_{\dashv}(d_{2}))\big),    \nonumber
\end{align}
then $(D, \dashv, \vdash, \nu_{\dashv}, \nu_{\vdash})$ is called a {\bf diNovikov bialgebra}.
\end{defi}

We now give a construction of diNovikov bialgebra from a differential perm bialgebra.

\begin{thm}\label{thm:ind-dinovbia}
Let $(P, \diamond, \vartheta, \fd, -\fd)$ be a differential perm bialgebra.
Define binary operations $\dashv_{\fd}, \vdash_{\fd}$ and coproducts $\nu_{\dashv,\fd},
\nu_{\vdash,\fd}$ on $P$ by Eqs. \eqref{ind-dinov} and \eqref{ind-codinov} respectively.
Then $(P, \dashv_{\fd}, \vdash_{\fd}, \nu_{\dashv,\fd}, \nu_{\vdash,\fd})$ is a diNovikov
bialgebra, which is called {\bf the diNovikov bialgebra induced by the
differential perm bialgebra $(P, \diamond, \vartheta, \fd, -\fd)$}.
\end{thm}

\begin{proof}
By Proposition \ref{pro:dperm-diN}, we get that $(P, \dashv_{\fd}, \vdash_{\fd})$ is a
diNovikov algebra, and by Proposition \ref{pro:ind-codinov}, we get
that $(P, \nu_{\dashv,\fd}, \nu_{\vdash,\fd})$ is a diNovikov coalgebra.
Following, we show that Eqs. \eqref{ndba1}-\eqref{ndba9} hold. Since $(P, \diamond,
\vartheta, \fd, -\fd)$ is a differential perm bialgebra, we get that $\vartheta\circ\fd=
(\id\otimes\fd+\fd\otimes\id)\circ\vartheta$ and $(\fr_{P}(p_{1})\otimes\id)
(\vartheta(p_{2}))=\tau((\fr_{P}(p_{2})\otimes\id)(\vartheta(p_{1})))$ for any
$p_{1}, p_{2}\in P$. If denote $\nu_{\dashv,\fd}(p)=\sum_{(p)}p_{(1)}\otimes p_{(2)}$,
then for any $p, p'\in P$, we obtain
\begin{align*}
&\quad (\id\otimes\fl_{\dashv_{\fd}}(p'))(\nu_{\dashv,\fd}(p))\\
&=\sum_{(p)}p_{(2)}\otimes(\fd^{2}(p_{(1)})\diamond p')
\qquad\qquad\qquad\qquad\qquad\qquad\qquad\qquad\qquad\qquad\qquad\qquad\qquad\quad
\end{align*}
\begin{align*}
&=\sum_{(p)}p_{(2)}\otimes\fd(\fd(p_{(1)})\diamond p')
-p_{(2)}\otimes\fd(p_{(1)}\diamond\fd(p'))+p_{(2)}\otimes(p_{(1)}\diamond\fd^{2}(p'))\\[-2mm]
&=(\id\otimes\fd^{2})((\fr_{P}(p)\otimes\id)(\vartheta(p')))
-2(\id\otimes\fd)((\fr_{P}(p)\otimes\id)(\vartheta(\fd(p'))))
+(\fr_{P}(p)\otimes\id)(\vartheta(\fd^{2}(p')))\\
&=(\id\otimes\fd^{2})((\fr_{P}(p)\otimes\id)(\vartheta(p')))
-2(\id\otimes\fd)((\fr_{P}(p)\otimes\id)((\id\otimes\fd)(\vartheta(p'))
+(\fd\otimes\id)(\vartheta(p'))))\\
&\qquad+(\fr_{P}(p)\otimes\id)((\id\otimes\fd)(\vartheta(\fd(p')))
+(\fd\otimes\id)(\vartheta(\fd(p'))))\\
&=(\fr_{P}(p)\otimes\id)((\fd\otimes\id)(\vartheta(\fd(p'))))
-(\id\otimes\fd)((\fr_{P}(p)\otimes\id)((\fd\otimes\id)(\vartheta(p'))))\\
&=(\fr_{P}(p)\otimes\id)((\fd\otimes\id)((\id\otimes\fd)(\vartheta(p'))
+(\fd\otimes\id)(\vartheta(p'))))
-(\id\otimes\fd)((\fr_{P}(p)\otimes\id)((\fd\otimes\id)(\vartheta(p'))))\\
&=(\fr_{P}(p)\otimes\id)((\fd^{2}\otimes\id)(\vartheta(p')))\\
&=\sum_{(p')}(\fd^{2}(p'_{(1)})\diamond p)\otimes p'_{(2)}\\[-2mm]
&=(\fl_{\dashv_{\fd}}(p)\otimes\id)(\tau(\nu_{\dashv,\fd}(p'))).
\end{align*}
That is, Eq. \eqref{ndba1} holds. Similarly, we also have Eqs.
\eqref{ndba2}-\eqref{ndba9} hold. Hence, $(P, \dashv_{\fd}, \vdash_{\fd},
\nu_{\dashv,\fd}, \nu_{\vdash,\fd})$ is a diNovikov bialgebra.
\end{proof}

Let $(D, \dashv, \vdash)$ be a diNovikov algebra and $r=\sum_{i}x_{i}\otimes y_{i}
\in D\otimes D$. We define
\begin{align*}
\mathbf{DN}_{r}&=r_{12}\vdash r_{13}+r_{12}\vdash r_{23}
-r_{12}\dashv r_{13}+r_{13}\vdash r_{23}+r_{23}\dashv r_{13},
\end{align*}
where $r_{12}\vdash r_{13}=\sum_{i,j}(x_{i}\vdash x_{j})\otimes y_{i}\otimes y_{j}$,
$r_{12}\vdash r_{23}=\sum_{i,j}x_{i}\otimes(y_{i}\vdash x_{j})\otimes y_{j}$,
$r_{12}\dashv r_{13}=\sum_{i,j}(x_{i}\dashv x_{j})\otimes y_{i}\otimes y_{j}$,
$r_{13}\vdash r_{23}=\sum_{i,j}x_{i}\otimes x_{j}\otimes(y_{i}\vdash y_{j})$ and
$r_{23}\dashv r_{13}=\sum_{i,j}x_{i}\otimes x_{j}\otimes(y_{j}\dashv y_{i})$.
The equation $\mathbf{DN}_{r}=0$ is called the {\bf diNovikov Yang-Baxter equation}
(or $\DNYBE$) in $(D, \dashv, \vdash)$. We define linear maps
\begin{align}
&\nu_{\vdash, r}: D\rightarrow D\otimes D,\qquad
\nu_{\vdash, r}(d)=((\fl_{\dashv}-\fl_{\vdash})(d)\otimes\id
-\id\otimes(\fl_{\dashv}+\fr_{\vdash})(d))(r),\label{ndadr1}\\
&\nu_{\dashv, r}: D\rightarrow D\otimes D,\qquad
\nu_{\dashv, r}(d)=(\id\otimes(\fr_{\dashv}-\fl_{\dashv}+\fl_{\vdash}-\fr_{\vdash})(d)
+\fl_{\dashv}(d)\otimes\id)(\tau(r)), \label{ndadr2}
\end{align}
for any $d\in D$. If $(D, \dashv, \vdash, \nu_{\vdash, r}, \nu_{\dashv, r})$ is a
diNovikov bialgebra, then it will be called a {\bf coboundary diNovikov bialgebra}.
In particular, we have

\begin{pro}[\cite{XBH}]\label{pro:tri-dinov}
Let $(D, \dashv, \vdash)$ be a diNovikov algebra and $r=\sum_{i}x_{i}\otimes y_{i}
\in D\otimes D$. If $r$ is a symmetric solution of the $\DNYBE$ in $(D, \dashv, \vdash)$,
then $(D, \dashv, \vdash, \nu_{\vdash, r}, \nu_{\dashv, r})$ is a diNovikov bialgebra,
which is called a {\bf triangular diNovikov bialgebra} associated with $r$, where
$\nu_{\vdash, r}$ and $\nu_{\dashv, r}$ are given by Eqs. \eqref{ndadr1} and
\eqref{ndadr2} respectively.
\end{pro}

\begin{rmk}\label{rmk:DNYBE}
In paper \cite{XBH}, the $\DNYBE$ in a diNovikov algebra $(D, \dashv, \vdash)$
is given by $\mathbf{DN}'_{r}:=r_{12}\vdash r_{13}+r_{12}\vdash r_{23}-r_{12}\dashv r_{23}
+r_{23}\vdash r_{13}+r_{13}\dashv r_{23}=0$, which is different from $\mathbf{DN}_{r}=0$
we present here. But it is not difficult to see that $r\in D\otimes D$ is a
symmetric solution of $\mathbf{DN}_{r}=0$ if and only if $r$ is a symmetric solution
of $\mathbf{DN}'_{r}=0$.
\end{rmk}

Each admissible differential perm algebra induces a diNovikov algebra. Furthermore,
each solution of the $\DPYBE$ in an admissible differential perm algebra is also
a solution of the $\DNYBE$ in the induced diNovikov algebra.

\begin{pro}\label{pro:ind-DPYBE}
Let $(P, \diamond, \fd, -\fd)$ be an admissible differential perm algebra, $r=\sum_{i}x_{i}
\otimes y_{i}\in P\otimes P$ and $(P, \dashv_{\fd}, \vdash_{\fd})$ be the induced
diNovikov algebra. Then $r$ is a solution of the $\DNYBE$ in $(P, \dashv_{\fd},
\vdash_{\fd})$ if $r$ is a solution of the $\DPYBE$ in $(P, \diamond, \fd, -\fd)$.
In particular, each symmetric solution of the $\DPYBE$ in $(P, \diamond, \fd, -\fd)$ is
also a symmetric solution of the $\DNYBE$ in the induced diNovikov algebra
$(P, \dashv_{\fd}, \vdash_{\fd})$.
\end{pro}

\begin{proof}
If $r$ is a solution of the $\DPYBE$ in $(P, \diamond, \fd, -\fd)$, we get that
$(\fd\otimes\id+\id\otimes\fd)(r)=0$ and $\mathbf{P}_{r}=0$. By direct calculation,
we have
\begin{align*}
\mathbf{DN}_{r}&=r_{12}\vdash_{\fd}r_{23}+r_{12}\vdash_{\fd}r_{13}
-r_{12}\dashv_{\fd}r_{13}+r_{13}\vdash_{\fd}r_{23}+r_{23}\dashv_{\fd}r_{13}\\
&=\sum_{i,j}\Big(x_{i}\otimes(y_{i}\diamond\fd(x_{j}))\otimes y_{j}
+(x_{i}\diamond\fd(x_{j}))\otimes y_{i}\otimes y_{j}
-(\fd(x_{j})\diamond x_{i})\otimes y_{i}\otimes y_{j}\\[-5mm]
&\qquad\qquad+x_{i}\otimes x_{j}\otimes(y_{i}\diamond\fd(y_{j}))
+x_{i}\otimes x_{j}\otimes(\fd(y_{i})\diamond y_{j})\Big)\\
&=\sum_{i,j}\Big(-x_{i}\otimes(y_{i}\diamond x_{j})\otimes\fd(y_{j})
-(x_{i}\diamond x_{j})\otimes y_{i}\otimes\fd(y_{j})
+(x_{j}\diamond x_{i})\otimes y_{i}\otimes\fd(y_{j})\\[-5mm]
&\qquad\qquad+x_{i}\otimes x_{j}\otimes(y_{i}\diamond\fd(y_{j}))
+x_{i}\otimes x_{j}\otimes(\fd(y_{i})\diamond y_{j})\Big)\\
&=-(\id\otimes\id\otimes\fd)(\mathbf{P}_{r})\\
&= 0.
\end{align*}
That is, $r$ is a solution of the $\DNYBE$ in $(P, \dashv_{\fd}, \vdash_{\fd})$.
\end{proof}

Now we can show that the the induced diNovikov bialgebra given in Theorem
\ref{thm:ind-dinovbia} by a triangular differential perm bialgebra is also triangular.

\begin{thm}\label{thm:indu-spdiNbia}
Let $(P, \diamond, \vartheta_{r}, \fd, -\fd)$ be a triangular differential perm
bialgebra associated with a symmetric solution $r$ of the $\DPYBE$ in $(P, \diamond,
\fd, -\fd)$, where $\vartheta_{r}$ is given by Eq. \eqref{cobo}. Suppose
$(P, \dashv_{\fd}, \vdash_{\fd}, \nu_{\dashv,\fd}, \nu_{\vdash,\fd})$ is the induced
diNovikov bialgebra by the differential perm bialgebra $(P, \diamond, \vartheta_{r},
\fd, -\fd)$. Then $\nu_{\dashv,\fd}=\nu_{\dashv,r}$ and $\nu_{\vdash,\fd}=\nu_{\vdash,r}$,
and so that $(P, \dashv_{\fd}, \vdash_{\fd}, \nu_{\dashv,\fd}, \nu_{\vdash,\fd})$ is a
triangular diNovikov bialgebra associated with $r$, where $\nu_{\dashv,\fd}$ and
$\nu_{\vdash,\fd}$ are given by Eq. \eqref{ind-codinov}. Therefore, we obtain the following
commutative diagram:
$$
\xymatrix@C=2cm@R=0.6cm{
\txt{$r$ \\ {\tiny a symmetric solution of}\\
{\tiny the $\DPYBE$ in $(P, \diamond, \fd, -\fd)$}}
\ar[d]_{{\rm Pro.}~\ref{pro:ind-DPYBE}}\ar[r]^{{\rm Pro.}~\ref{pro:quasi-dpba}} &
\txt{$(P, \diamond, \vartheta_{r}, \fd, -\fd)$ \\ {\tiny a triangular differential}\\
{\tiny perm bialgebra}} \ar[d]^{{\rm Thm.}~\ref{thm:ind-dinovbia}} \\
\txt{$r$ \\ {\tiny a symmetric solution of}\\
{\tiny the $\DNYBE$ in $(P, \dashv_{\fd}, \vdash_{\fd})$}}
\ar[r]^{{\rm Pro.}~\ref{pro:tri-dinov}} &
\txt{$(P, \dashv_{\fd}, \vdash_{\fd}, \nu_{\dashv,r}, \nu_{\vdash,r})$ \\
{\tiny a triangular}\\ {\tiny diNovikov bialgebra}}}
$$
\end{thm}

\begin{proof}
Since $r$ is a solution of the $\DPYBE$ in $(P, \diamond, \fd, -\fd)$, we have
$(\fd\otimes\id+\id\otimes\fd)(r)=0$. Thus, by direct calculation, for any $p\in P$, we have
\begin{align*}
\nu_{\vdash, r}(p)&=((\fl_{\dashv}-\fl_{\vdash})(p)\otimes\id
-\id\otimes(\fl_{\dashv}+\fr_{\vdash})(p))(r)\\
&=\sum_{i}\big((\fd(x_{i})\diamond p)\otimes y_{i}-(p\diamond\fd(x_{i}))\otimes y_{i}
-x_{i}\otimes(\fd(y_{i})\diamond p)-x_{i}\otimes(y_{i}\diamond\fd(p))\big)\\[-2mm]
&=\sum_{i}\big((p\diamond x_{i})\otimes\fd(y_{i})-x_{i}\otimes\fd(y_{i}\diamond p)
-(x_{i}\diamond p)\otimes\fd(y_{i})\big)\\[-2mm]
&=(\id\otimes\fd)\Big(((\fl_{P}-\fr_{P})(p)\otimes\id-\id\otimes\fr_{P}(p))(r)\Big)\\
&=\nu_{\vdash,\fd}(p).
\end{align*}
Similarly, $\nu_{\dashv, r}=\nu_{\dashv,\fd}$. Thus, $(P, \dashv_{\fd}, \vdash_{\fd},
\nu_{\dashv,\fd}, \nu_{\vdash,\fd})$ is a triangular diNovikov bialgebra,
and the diagram is commutative.
\end{proof}

\begin{ex}\label{ex:ind-tridN}
Let $(P={\rm span}_{\Bbbk}\{x_{1}, x_{2}\}, \diamond, \vartheta, \fd, -\fd)$ be the
$2$-dimensional triangular differential perm bialgebra associated with $r=x_{2}
\otimes x_{2}$ given in Example \ref{ex:YBE-bialg}. Then this differential perm bialgebra
induces a diNovikov bialgebra, where the nonzero product is given by $x_{1}\vdash_{\fd}
x_{1}=x_{2}$ and the coproducts $\nu_{\dashv,\fd}$, $\nu_{\vdash,\fd}$ are zero.
This diNovikov bialgebra is triangular since it is exactly the triangular diNovikov
bialgebra associated with the symmetric solution $r=x_{2}\otimes x_{2}$ of the
$\DNYBE$ in $(P, \dashv_{\fd}, \vdash_{\fd})$.
\end{ex}

Next, we consider the operators forms of solution of the $\DNYBE$ in a diNovikov algebra.

\begin{defi}\label{def:o-diN}
Let $(D, \dashv, \vdash)$ be a diNovikov algebra and $(V, \kl_{\dashv}, \kr_{\dashv},
\kl_{\vdash}, \kr_{\vdash})$ be a representation of it. A linear map $T: V\rightarrow D$
is called an {\bf $\mathcal{O}$-operator of $(D, \dashv, \vdash)$ associated to
$(V, \kl_{\dashv}, \kr_{\dashv}, \kl_{\vdash}, \kr_{\vdash})$} if for any $v_{1}, v_{2}\in V$,
\begin{align*}
T(v_{1})\dashv T(v_{2})&=T\big(\kl_{\dashv}(T(v_{1}))(v_{2})
+\kr_{\dashv}(T(v_{2}))(v_{1})\big),\\
T(v_{1})\vdash T(v_{2})&=T\big(\kl_{\vdash}(T(v_{1}))(v_{2})
+\kr_{\vdash}(T(v_{2}))(v_{1})\big).
\end{align*}
\end{defi}

For the symmetric solution of the $\DNYBE$ in a diNovikov algebra and
$\mathcal{O}$-operator, we have:

\begin{pro}\label{pro:o-dinov}
Let $(D, \dashv, \vdash)$ be a diNovikov algebra and $r\in D\otimes D$ be symmetric.
Then $r$ is a solution of the $\DNYBE$ in $(D, \dashv, \vdash)$ if and only if
$r^{\sharp}: D^{\ast}\rightarrow D$ is an $\mathcal{O}$-operator of $(D, \dashv, \vdash)$
associated to the coregular representation $(D^{\ast}, \fr_{\dashv}^{\ast}-\fl_{\dashv}^{\ast}
+\fl_{\vdash}^{\ast}-\fr_{\vdash}^{\ast}, -\fr_{\dashv}^{\ast}, \fl_{\vdash}^{\ast}
+\fr_{\dashv}^{\ast}, \fr_{\vdash}^{\ast}-\fr_{\dashv}^{\ast})$.
\end{pro}

\begin{proof}
Let $r=\sum_{i}x_{i}\otimes y_{i}$ be a symmetric element in $D\otimes D$.
Then for any $\xi_{1}, \xi_{2}\in D^{\ast}$, we get $\langle r^{\sharp}(\xi_{1}),\;
\xi_{2}\rangle=\langle\xi_{1},\; r^{\sharp}(\xi_{2})\rangle$. Thus, for any
$\xi_{1}, \xi_{2}, \xi_{3}\in D^{\ast}$, we have
\begin{align*}
& \langle\xi_{1}\otimes\xi_{2}\otimes\xi_{3},\; (x_{i}\vdash x_{j})\otimes y_{i}
\otimes y_{j}\rangle
=\langle\xi_{1},\; r^{\sharp}(\xi_{2})\vdash r^{\sharp}(\xi_{3})\rangle
=-\langle\xi_{2},\; r^{\sharp}(\fr_{\vdash}^{\ast}(r^{\sharp}(\xi_{3}))(\xi_{1}))\rangle,\\
& \langle\xi_{1}\otimes\xi_{2}\otimes\xi_{3},\; x_{i}\otimes(y_{i}\vdash x_{j})
\otimes y_{j}\rangle
=\langle\xi_{2},\; r^{\sharp}(\xi_{1})\vdash r^{\sharp}(\xi_{3})\rangle,\\
& \langle\xi_{1}\otimes\xi_{2}\otimes\xi_{3},\; (x_{i}\dashv x_{j})\otimes y_{i}
\otimes y_{j}\rangle
=\langle\xi_{1},\; r^{\sharp}(\xi_{2})\dashv r^{\sharp}(\xi_{3})\rangle
=-\langle\xi_{2},\; r^{\sharp}(\fr_{\dashv}^{\ast}(r^{\sharp}(\xi_{3}))(\xi_{1}))\rangle,\\
& \langle\xi_{1}\otimes\xi_{2}\otimes\xi_{3},\; x_{i}\otimes x_{j}\otimes(y_{i}
\vdash y_{j})\rangle
=\langle\xi_{3},\; r^{\sharp}(\xi_{1})\vdash r^{\sharp}(\xi_{2})\rangle
=-\langle\xi_{2},\; r^{\sharp}(\fl_{\vdash}^{\ast}(r^{\sharp}(\xi_{1}))(\xi_{3}))\rangle,\\
& \langle\xi_{1}\otimes\xi_{2}\otimes\xi_{3},\; x_{i}\otimes x_{j}\otimes(y_{j}
\dashv y_{i})\rangle
=\langle\xi_{3},\; r^{\sharp}(\xi_{2})\dashv r^{\sharp}(\xi_{1})\rangle
=-\langle\xi_{2},\; r^{\sharp}(\fr_{\dashv}^{\ast}(r^{\sharp}(\xi_{1}))(\xi_{3}))\rangle.
\end{align*}
That is, $\langle\xi_{1}\otimes\xi_{2}\otimes\xi_{3},\; \mathbf{DN}_{r}\rangle=
\langle\xi_{1}\otimes\xi_{2}\otimes\xi_{3},\; r^{\sharp}(\xi_{1})\vdash r^{\sharp}(\xi_{3})
-r^{\sharp}\big(\fl_{\vdash}^{\ast}(r^{\sharp}(\xi_{1}))(\xi_{3})
+\fr_{\dashv}^{\ast}(r^{\sharp}(\xi_{1}))(\xi_{3})
+\fr_{\vdash}^{\ast}(r^{\sharp}(\xi_{3}))(\xi_{1})
-\fr_{\dashv}^{\ast}(r^{\sharp}(\xi_{3}))(\xi_{1})\big)\rangle$. Thus, we get that
$r^{\sharp}(\xi_{1})\vdash r^{\sharp}(\xi_{3})=r^{\sharp}\big(\fl_{\vdash}^{\ast}(r^{\sharp}
(\xi_{1}))(\xi_{3})+\fr_{\dashv}^{\ast}(r^{\sharp}(\xi_{1}))(\xi_{3})
+\fr_{\vdash}^{\ast}(r^{\sharp}(\xi_{3}))(\xi_{1})
-\fr_{\dashv}^{\ast}(r^{\sharp}(\xi_{3}))(\xi_{1})\big)$ if and only if $\mathbf{DN}_{r}=0$.
Similarly, we also have $r^{\sharp}(\xi_{1})\dashv r^{\sharp}(\xi_{3})=r^{\sharp}\big(
\fr_{\dashv}^{\ast}(r^{\sharp}(\xi_{1}))(\xi_{3})
-\fl_{\dashv}^{\ast}(r^{\sharp}(\xi_{1}))(\xi_{3})
+\fl_{\vdash}^{\ast}(r^{\sharp}(\xi_{1}))(\xi_{3})
-\fr_{\vdash}^{\ast}(r^{\sharp}(\xi_{1}))(\xi_{3})
-\fr_{\dashv}^{\ast}(r^{\sharp}(\xi_{3}))(\xi_{1})\big)$ if and only if $\mathbf{DN}_{r}=0$.
Therefore, we obtain that $r$ is a solution of the $\DNYBE$ in $(D, \dashv, \vdash)$
if and only if $r^{\sharp}: D^{\ast}\rightarrow D$ is an $\mathcal{O}$-operator of
$(D, \dashv, \vdash)$ associated to the coregular representation.
\end{proof}

Thus, for the $\mathcal{O}$-operator of a admissible differential perm algebra and
the $\mathcal{O}$-operator of the induced diNovikov algebra, by Proposition
\ref{pro:o-dperm} and \ref{pro:o-dinov}, we have:

\begin{cor}\label{cor:o-dass-dN}
Let $(P, \diamond, \fd, -\fd)$ be an admissible differential perm algebra
and $(P, \vdash_{\fd}, \dashv_{\fd})$ be the induced diNovikov algebra by $(A, \cdot, \fd)$.
If $r$ is a symmetric solution of the $\DPYBE$ in $(P, \diamond, \fd, -\fd)$,
Then we have the following commutative diagram:
$$
\xymatrix@C=3cm@R=0.5cm{
\txt{$r$ \\ {\tiny a symmetric solution} \\ {\tiny of the $\DPYBE$ in
$(P, \diamond, \fd, -\fd)$}}
\ar[d]_-{{\rm Pro.}~\ref{pro:ind-DPYBE}}\ar[r]^-{{\rm Pro.}~\ref{pro:o-dperm}} &
\txt{$r^{\sharp}$\\ {\tiny an $\mathcal{O}$-operator of $(P, \diamond, \fd)$} \\
{\tiny associated to $(P^{\ast}, -\fl_{P}^{\ast}, \fr_{P}^{\ast}-\fl_{P}^{\ast}
-\fd^{\ast})$}} \ar[d] \\
\txt{$r$ \\ {\tiny a symmetric solution} \\ {\tiny of the $\DNYBE$ in
$(P, \vdash_{\fd}, \dashv_{\fd})$}} \ar[r]^-{{\rm Pro.}~\ref{pro:o-dinov}}
& \txt{$r^{\sharp}$ \\ {\tiny an $\mathcal{O}$-operator of $(P, \vdash_{\fd},
\dashv_{\fd})$ associated to } \\
{\tiny $(P^{\ast}, \fr_{\dashv_{\fd}}^{\ast}-\fl_{\dashv_{\fd}}^{\ast}
+\fl_{\vdash_{\fd}}^{\ast}-\fr_{\vdash_{\fd}}^{\ast}, -\fr_{\dashv_{\fd}}^{\ast},
\fl_{\vdash_{\fd}}^{\ast}+\fr_{\dashv_{\fd}}^{\ast},
\fr_{\vdash_{\fd}}^{\ast}-\fr_{\dashv_{\fd}}^{\ast})$}}}
$$
\end{cor}

\subsection{From diNovikov bialgebras to Novikov bialgebras} \label{subsec:diNov-Nov}
In this subsection, we show that there is naturally a Novikov bialgebra structure on
the tensor product of a diNovikov bialgebra and a quadratic Zinbiel algebra.
We have shown that there is a Novikov algebra structure on the tensor product of
a diNovikov algebra and a Zinbiel algebra. Dual, we can construct a Novikov coalgebra from
the tensor product of a diNovikov coalgebra and a Zinbiel coalgebra. Recall that
a {\bf Novikov coalgebra} $(B, \delta)$ is a vector space
$B$ with a linear map $\delta: B\rightarrow B\otimes B$ such that
\begin{align*}
(\id\otimes\tau)\circ(\delta\otimes\id)\circ\delta&=(\delta\otimes\id)\circ\delta,\\
(\id\otimes\delta)\circ\delta-(\tau\otimes\id)\circ(\id\otimes\delta)\circ\delta
&=(\delta\otimes\id)\circ\delta-(\tau\otimes\id)\circ(\delta\otimes\id)\circ\delta.
\end{align*}

\begin{pro}\label{pro:nco-dinco}
Let $(D, \nu_{\dashv}, \nu_{\vdash})$ be a diNovikov coalgebra and $(C, \theta)$ be a
Zinbiel coalgebra. We define a linear map $\delta: D\otimes C\rightarrow(D\otimes C)
\otimes(D\otimes C)$ by
\begin{align*}
\delta(d\otimes c)&=\nu_{\vdash}(d)\bullet\theta(c)+\nu_{\dashv}(d)\bullet\tau(\theta(c))\\
&=\sum_{(d)}\sum_{(c)}(d_{(1)}\otimes c_{(1)})\otimes(d_{(2)}\otimes c_{(2)})
+\sum_{[d]}\sum_{(c)}(d_{[1]}\otimes c_{(2)})\otimes(d_{[2]}\otimes c_{(1)}),
\end{align*}
for any $d\in D$ and $c\in C$, where $\nu_{\vdash}(d)=\sum_{(d)}d_{(1)}\otimes d_{(2)}$,
$\nu_{\dashv}(d)=\sum_{[d]}d_{[1]}\otimes d_{[2]}$ and $\theta(c)=\sum_{(c)}c_{(1)}
\otimes c_{(2)}$ in the Sweedler notation. Then $(D\otimes C, \delta)$ is a Novikov coalgebra.
\end{pro}

\begin{proof}
For any $d\in D$ and $c\in C$, since $(C, \theta)$ is a Zinbiel coalgebra and
$(D, \nu_{\dashv}, \nu_{\vdash})$ is a diNovikov coalgebra, we have
\begin{align*}
(\id\otimes\tau)((\delta\otimes\id)(\delta(d\otimes c)))
&= (\id\otimes\tau)((\nu_{\vdash}\otimes\id)(\nu_{\vdash}(d)))
\bullet(\id\otimes\tau)((\theta\otimes\id)(\theta(c)))\\[-1mm]
&\quad +(\id\otimes\tau)((\nu_{\vdash}\otimes\id)(\nu_{\dashv}(d)))
\bullet(\theta\otimes\id)(\theta(c))\\[-1mm]
&\quad +(\id\otimes\tau)((\nu_{\vdash}\otimes\id)(\nu_{\dashv}(d)))
\bullet(\tau\otimes\id)((\theta\otimes\id)(\theta(c)))\\[-1mm]
&\quad +(\id\otimes\tau)((\nu_{\dashv}\otimes\id)(\nu_{\vdash}(d)))
\bullet(\id\otimes\tau)((\tau\otimes\id)((\theta\otimes\id)(\theta(c))))\\[-1mm]
&\quad +(\id\otimes\tau)((\nu_{\dashv}\otimes\id)(\nu_{\dashv}(d)))
\bullet(\tau\otimes\id)((\id\otimes\tau)((\theta\otimes\id)(\theta(c))))\\
&\quad +(\id\otimes\tau)((\nu_{\dashv}\otimes\id)(\nu_{\dashv}(d)))
\bullet(\tau\otimes\id)((\id\otimes\tau)((\tau\otimes\id)((\theta\otimes\id)(\theta(c)))))\\
&= (\nu_{\vdash}\otimes\id)(\nu_{\vdash}(d))
\bullet(\theta\otimes\id)(\theta(c))\\[-1mm]
&\quad +(\nu_{\vdash}\otimes\id)(\nu_{\dashv}(d))
\bullet(\id\otimes\tau)((\theta\otimes\id)(\theta(c)))\\[-1mm]
&\quad +(\nu_{\vdash}\otimes\id)(\nu_{\dashv}(d))
\bullet(\id\otimes\tau)((\tau\otimes\id)((\theta\otimes\id)(\theta(c))))\\[-1mm]
&\quad +(\nu_{\dashv}\otimes\id)(\nu_{\vdash}(d))
\bullet(\tau\otimes\id)((\theta\otimes\id)(\theta(c)))\\[-1mm]
&\quad +(\nu_{\dashv}\otimes\id)(\nu_{\dashv}(d))
\bullet(\tau\otimes\id)((\id\otimes\tau)((\theta\otimes\id)(\theta(c))))\\[-1mm]
&\quad +(\nu_{\dashv}\otimes\id)(\nu_{\dashv}(d))
\bullet(\tau\otimes\id)((\id\otimes\tau)((\tau\otimes\id)
((\theta\otimes\id)(\theta(c)))))\\
&= (\delta\otimes\id)(\delta(d\otimes c)).
\end{align*}
That is, $(\id\otimes\tau)\circ(\delta\otimes\id)\circ\delta=(\delta\otimes\id)\circ\delta$.
Similarly, we can obtain $(\id\otimes\delta)\circ\delta-(\tau\otimes\id)\circ
(\id\otimes\delta)\circ\delta=(\delta\otimes\id)\circ\delta-(\tau\otimes\id)\circ
(\delta\otimes\id)\circ\delta$. Thus, $(D\otimes C, \delta)$ is a Novikov coalgebra.
\end{proof}

Let $(B, \ast)$ be a Novikov algebra, $V$ be a vector space and $\tilde{\kl}, \tilde{\kr}:
B\rightarrow\gl(V)$ be a linear map. Then $(V, \tilde{\kl}, \tilde{\kr})$ is called a
{\bf representation of $(B, \ast)$} if for any $b_{1}, b_{2}\in B$,
\begin{align*}
&\tilde{\kl}(b_{1}\ast b_{2}-b_{2}\ast b_{1})=\tilde{\kl}(b_{1})\circ\tilde{\kl}(b_{2})
-\tilde{\kl}(b_{2})\circ\tilde{\kl}(b_{1}),
\qquad\qquad\qquad\; \tilde{\kl}(b_{1}\ast b_{2})=\tilde{\kr}(b_{2})\circ\tilde{\kl}(b_{1}),\\
&\tilde{\kl}(b_{1})\circ\tilde{\kr}(b_{2})-\tilde{\kr}(b_{2})\circ\tilde{\kl}(b_{1})
=\tilde{\kr}(b_{1}\ast b_{2})-\tilde{\kr}(b_{2})\circ\tilde{\kr}(b_{1}),
\qquad\qquad\tilde{\kr}(b_{1})\circ\tilde{\kr}(b_{2})=\tilde{\kr}(b_{2})\circ\tilde{\kr}(b_{1}).
\end{align*}
Define the left product map and right product map $\tilde{\fl}_{B}, \tilde{\fr}_{B}:
B\rightarrow\gl(V)$ by $\tilde{\fl}_{B}(b_{1})(b_{2})=b_{1}\ast b_{2}=
\tilde{\fr}_{B}(b_{2})(b_{1})$ for any $b_{1}, b_{2}\in B$. Then $(B, \tilde{\fl}_{B},
\tilde{\fr}_{B})$ is a representation of Novikov algebra $(B, \ast)$, which is called the
{\bf regular representation} of $(B, \ast)$. The {\bf coregular representation} of $(B,
\ast)$ is given by $(B^{\ast}, \tilde{\fl}_{B}+\tilde{\fr}_{B}, -\tilde{\fr}_{B})$.

\begin{defi}[\cite{LLB}]\label{def:Nov-bia}
A {\bf Novikov bialgebra} is a triple $(B, \ast, \delta)$ where $(B, \ast)$
is a Novikov algebra and $(B, \delta)$ is a Novikov coalgebra such that
\begin{align}
&\qquad\qquad \delta(b_{1}\ast b_{2})=(\tilde{\fr}_{B}(b_{2})\otimes\id)(\delta(b_{1}))
+(\id\otimes(\tilde{\fl}_{B}+\tilde{\fr}_{B})(b_{1}))(\delta(b_{2})+\tau(\delta(b_{2}))),
\label{Nbialg1}\\
&\qquad\qquad\qquad((\tilde{\fl}_{B}+\tilde{\fr}_{B})(b_{1})\otimes\id)(\delta(b_{2}))
-(\id\otimes(\tilde{\fl}_{B}+\tilde{\fr}_{B})(b_{1}))(\tau(\delta(b_{2})))
\label{Nbialg2}\\[-1mm]
&\qquad\qquad\quad=((\tilde{\fl}_{B}+\tilde{\fr}_{B})(b_{2})\otimes\id)(\delta(b_{1}))
-(\id\otimes(\tilde{\fl}_{B}+\tilde{\fr}_{B})(b_{2}))(\tau(\delta(b_{1}))), \nonumber\\
&(\id\otimes\tilde{\fr}_{B}(b_{1})-\tilde{\fr}_{B}(b_{1})\otimes\id)(\delta(b_{2})
+\tau(\delta(b_{2})))
=(\id\otimes\tilde{\fr}_{B}(b_{2})-\tilde{\fr}_{B}(b_{2})\otimes\id)(\delta(b_{1})
+\tau(\delta(b_{1}))), \label{Nbialg3}
\end{align}
for all $b_{1}, b_{2}\in B$.
\end{defi}

Now, we show that there is a Novikov bialgebra structure on the tensor product of a
diNovikov bialgebra and a quadratic Zinbiel algebra.

\begin{thm}\label{thm:diNbia-Nbia}
Let $(D, \dashv, \vdash, \nu_{\dashv}, \nu_{\vdash})$ be a diNovikov bialgebra,
$(C, \star, \omega)$ be a quadratic Zinbiel algebra and $(D\otimes C, \ast)$ be the
induced Novikov algebra by $(D, \dashv, \vdash)$ and $(C, \star)$ in Propositions
\ref{pro:tensor}. Define a linear map $\delta: D\otimes C\rightarrow(D\otimes C)
\otimes(D\otimes C)$ by
\begin{align}
\delta(d\otimes c)&=\nu_{\vdash}(d)\bullet\theta_{\omega}(c)
+\nu_{\dashv}(d)\bullet\tau(\theta_{\omega}(c))  \label{coNalg}\\
&=\sum_{(d)}\sum_{(c)}(d_{(1)}\otimes c_{(1)})\otimes(d_{(2)}\otimes c_{(2)})
+\sum_{[d]}\sum_{(c)}(d_{[1]}\otimes c_{(2)})\otimes(d_{[2]}\otimes c_{(1)}), \nonumber
\end{align}
for any $d\in D$ and $c\in C$, where $\nu_{\vdash}(d)=\sum_{(d)}d_{(1)}\otimes d_{(2)}$,
$\nu_{\dashv}(d)=\sum_{[d]}d_{[1]}\otimes d_{[2]}$ and $\theta(c)=\sum_{(c)}c_{(1)}
\otimes c_{(2)}$ in the Sweedler notation. Then $(D\otimes C, \ast, \delta)$ is a Novikov
bialgebra, which is called {\bf the Novikov bialgebra induced from $(D, \dashv, \vdash,
\nu_{\dashv}, \nu_{\vdash})$ by $(C, \star, \omega)$}.
\end{thm}

\begin{proof}
By Proposition \ref{pro:nco-dinco} and Lemma \ref{lem:Zinb-dual}, we get that
$(D\otimes C, \delta)$ is a Novikov coalgebra. Following, we show that $(D\otimes C,
\ast, \delta)$ is a Novikov bialgebra.
Similar to the proof of Theorem \ref{thm:permbia-ASI}, if we denote $\Phi_{c_{1}, c_{2}},
\Phi'_{c_{1}, c_{2}}, \Psi_{c_{1}, c_{2}}$, $\Psi'_{c_{1}, c_{2}}, \Omega_{c_{1}, c_{2}}
\in C\otimes C$ for any $c_{1}, c_{2}\in C$ by
\begin{align*}
&\omega(\Phi_{c_{1}, c_{2}},\; e\otimes f)=\omega(c_{1},\; (e\star f)\star c_{2}),\qquad\qquad
\omega(\Phi'_{c_{1}, c_{2}},\; e\otimes f)=\omega(c_{1},\; (f\star e)\star c_{2}),\\
&\omega(\Psi_{c_{1}, c_{2}},\; e\otimes f)=\omega(c_{1},\; (c_{2}\star e)\star f),\qquad\qquad
\omega(\Psi'_{c_{1}, c_{2}},\; e\otimes f)=\omega(c_{1},\; (e\star c_{2})\star f),\\
&\omega(\Omega_{c_{1}, c_{2}},\; e\otimes f)=\omega(c_{1},\; (c_{2}\star f)\star e
+(f\star c_{2})\star e),
\end{align*}
then we obtain
\begin{align*}
&\; \delta((d_{1}\otimes c_{1})\ast(d_{2}\otimes c_{2}))-(\bar{\fr}_{D\otimes C}
(d_{2}\otimes c_{2})\otimes\id)(\delta(d_{1}\otimes c_{1}))\\[-1mm]
&\qquad-(\id\otimes(\bar{\fl}_{D\otimes C}+\bar{\fr}_{D\otimes C})(d_{1}\otimes c_{1}))
(\delta(d_{2}\otimes c_{2})+\tau(\delta(d_{2}\otimes c_{2})))\\
=&\; \nu_{\vdash}(d_{1}\vdash d_{2})\bullet\theta_{\omega}(c_{1}\star c_{2})
+\nu_{\dashv}(d_{1}\vdash d_{2})\bullet\tau(\theta_{\omega}(c_{1}\star c_{2}))\\[-1mm]
&\quad +\nu_{\vdash}(d_{1}\dashv d_{2})\bullet\theta_{\omega}(c_{2}\star c_{1})
+\nu_{\dashv}(d_{1}\dashv d_{2})\bullet\tau(\theta_{\omega}(c_{2}\star c_{1}))\\[-1mm]
&\quad -(\fr_{\vdash}(d_{2})\otimes\id)(\nu_{\vdash}(d_{1}))
\bullet(\tilde{\fr}_{C}(c_{2})\otimes\id)(\theta_{\omega}(c_{1}))
-(\fr_{\dashv}(d_{2})\otimes\id)(\nu_{\vdash}(d_{1}))
\bullet(\tilde{\fl}_{C}(c_{2})\otimes\id)(\theta_{\omega}(c_{1}))\\[-1mm]
&\quad -(\fr_{\vdash}(d_{2})\otimes\id)(\nu_{\dashv}(d_{1}))
\bullet(\tilde{\fr}_{C}(c_{2})\otimes\id)(\tau(\theta_{\omega}(c_{1})))
-(\fr_{\dashv}(d_{2})\otimes\id)(\nu_{\dashv}(d_{1}))
\bullet(\tilde{\fl}_{C}(c_{2})\otimes\id)(\tau(\theta_{\omega}(c_{1}))\\[-1mm]
&\quad -(\id\otimes\fr_{\vdash}(d_{1}))(\nu_{\vdash}(d_{2}))
\bullet(\id\otimes\tilde{\fr}_{C}(c_{1}))(\theta_{\omega}(c_{2}))
-(\id\otimes\fr_{\dashv}(d_{1}))(\nu_{\vdash}(d_{2}))
\bullet(\id\otimes\tilde{\fl}_{C}(c_{1}))(\theta_{\omega}(c_{2}))\\[-1mm]
&\quad -(\id\otimes\fr_{\vdash}(d_{1}))(\nu_{\dashv}(d_{2}))
\bullet(\id\otimes\tilde{\fr}_{C}(c_{1}))(\tau(\theta_{\omega}(c_{2})))
-(\id\otimes\fr_{\dashv}(d_{1}))(\nu_{\dashv}(d_{2}))
\bullet(\id\otimes\tilde{\fl}_{C}(c_{1}))(\tau(\theta_{\omega}(c_{2})))\\[-1mm]
&\quad -(\id\otimes\fl_{\vdash}(d_{1}))(\nu_{\vdash}(d_{2}))
\bullet(\id\otimes\tilde{\fl}_{C}(c_{1}))(\theta_{\omega}(c_{2}))
-(\id\otimes\fl_{\dashv}(d_{1}))(\nu_{\vdash}(d_{2}))
\bullet(\id\otimes\tilde{\fr}_{C}(c_{1}))(\theta_{\omega}(c_{2}))\qquad\qquad\quad
\end{align*}
\begin{align*}
&\quad -(\id\otimes\fl_{\vdash}(d_{1}))(\nu_{\dashv}(d_{2}))
\bullet(\id\otimes\tilde{\fl}_{C}(c_{1}))(\tau(\theta_{\omega}(c_{2})))
-(\id\otimes\fl_{\dashv}(d_{1}))(\nu_{\dashv}(d_{2}))
\bullet(\id\otimes\tilde{\fr}_{C}(c_{1}))(\tau(\theta_{\omega}(c_{2})))\\[-1mm]
&\quad -(\id\otimes\fr_{\vdash}(d_{1}))(\tau(\nu_{\vdash}(d_{2})))
\bullet(\id\otimes\tilde{\fr}_{C}(c_{1}))(\tau(\theta_{\omega}(c_{2})))
-(\id\otimes\fr_{\dashv}(d_{1}))(\tau(\nu_{\vdash}(d_{2})))
\bullet(\id\otimes\tilde{\fl}_{C}(c_{1}))(\tau(\theta_{\omega}(c_{2})))\\[-1mm]
&\quad -(\id\otimes\fr_{\vdash}(d_{1}))(\tau(\nu_{\dashv}(d_{2})))
\bullet(\id\otimes\tilde{\fr}_{C}(c_{1}))(\theta_{\omega}(c_{2}))
-(\id\otimes\fr_{\dashv}(d_{1}))(\tau(\nu_{\dashv}(d_{2})))
\bullet(\id\otimes\tilde{\fl}_{C}(c_{1}))(\theta_{\omega}(c_{2}))\\[-1mm]
&\quad -(\id\otimes\fl_{\vdash}(d_{1}))(\tau(\nu_{\vdash}(d_{2})))
\bullet(\id\otimes\tilde{\fl}_{C}(c_{1}))(\tau(\theta_{\omega}(c_{2})))
-(\id\otimes\fl_{\dashv}(d_{1}))(\tau(\nu_{\vdash}(d_{2})))
\bullet(\id\otimes\tilde{\fr}_{C}(c_{1}))(\tau(\theta_{\omega}(c_{2})))\\[-1mm]
&\quad -(\id\otimes\fl_{\vdash}(d_{1}))(\tau(\nu_{\dashv}(d_{2})))
\bullet(\id\otimes\tilde{\fl}_{C}(c_{1}))(\theta_{\omega}(c_{2}))
-(\id\otimes\fl_{\dashv}(d_{1}))(\tau(\nu_{\dashv}(d_{2})))
\bullet(\id\otimes\tilde{\fr}_{C}(c_{1}))(\theta_{\omega}(c_{2}))\\
=&\; \Big(\nu_{\vdash}(d_{1}\dashv d_{2})-\nu_{\vdash}(d_{1}\vdash d_{2})
+(\fr_{\vdash}(d_{2})\otimes\id)(\nu_{\dashv}(d_{1}))
-(\fr_{\dashv}(d_{2})\otimes\id)(\nu_{\dashv}(d_{1}))
-(\id\otimes\fr_{\vdash}(d_{1}))(\nu_{\vdash}(d_{2}))\\[-1mm]
&\quad +(\id\otimes\fr_{\dashv}(d_{1}))(\nu_{\vdash}(d_{2}))
+(\id\otimes\fl_{\vdash}(d_{1}))(\nu_{\vdash}(d_{2}))
-(\id\otimes\fl_{\dashv}(d_{1}))(\nu_{\vdash}(d_{2}))
-(\id\otimes\fr_{\vdash}(d_{1}))(\tau(\nu_{\dashv}(d_{2})))\\[-1mm]
&\quad +(\id\otimes\fr_{\dashv}(d_{1}))(\tau(\nu_{\dashv}(d_{2})))
+(\id\otimes\fl_{\vdash}(d_{1}))(\tau(\nu_{\dashv}(d_{2})))
-(\id\otimes\fl_{\dashv}(d_{1}))(\tau(\nu_{\dashv}(d_{2})))\Big)\bullet\Phi\\
&\; +\Big(\nu_{\dashv}(d_{1}\dashv d_{2})-\nu_{\dashv}(d_{1}\vdash d_{2})
+(\fr_{\vdash}(d_{2})\otimes\id)(\nu_{\dashv}(d_{1}))
-(\fr_{\dashv}(d_{2})\otimes\id)(\nu_{\dashv}(d_{1}))
-(\id\otimes\fr_{\vdash}(d_{1}))(\nu_{\vdash}(d_{2}))\\[-1mm]
&\quad +(\id\otimes\fr_{\dashv}(d_{1}))(\nu_{\vdash}(d_{2}))
+(\id\otimes\fl_{\vdash}(d_{1}))(\nu_{\vdash}(d_{2}))
-(\id\otimes\fl_{\dashv}(d_{1}))(\nu_{\vdash}(d_{2}))
-(\id\otimes\fr_{\vdash}(d_{1}))(\tau(\nu_{\dashv}(d_{2})))\\[-1mm]
&\quad +(\id\otimes\fr_{\dashv}(d_{1}))(\tau(\nu_{\dashv}(d_{2})))
+(\id\otimes\fl_{\vdash}(d_{1}))(\tau(\nu_{\dashv}(d_{2})))
-(\id\otimes\fl_{\dashv}(d_{1}))(\tau(\nu_{\dashv}(d_{2})))\Big)\bullet\Phi'\\
&\; +\Big(\nu_{\vdash}(d_{1}\dashv d_{2})-(\fr_{\dashv}(d_{2})\otimes\id)(\nu_{\vdash}(d_{1}))
-(\id\otimes\fr_{\vdash}(d_{1}))(\nu_{\vdash}(d_{2}))
+(\id\otimes\fr_{\vdash}(d_{1}))(\nu_{\dashv}(d_{2}))\\[-1mm]
&\quad -(\id\otimes\fl_{\dashv}(d_{1}))(\nu_{\vdash}(d_{2}))
+(\id\otimes\fl_{\dashv}(d_{1}))(\nu_{\dashv}(d_{2}))
+(\id\otimes\fr_{\vdash}(d_{1}))(\tau(\nu_{\vdash}(d_{2})))\\[-1mm]
&\quad -(\id\otimes\fr_{\vdash}(d_{1}))(\tau(\nu_{\dashv}(d_{2})))
+(\id\otimes\fl_{\dashv}(d_{1}))(\tau(\nu_{\vdash}(d_{2})))
-(\id\otimes\fl_{\dashv}(d_{1}))(\tau(\nu_{\dashv}(d_{2})))\Big)\bullet\Psi\\
&\; +\Big(\nu_{\vdash}(d_{1}\dashv d_{2})
+(\fr_{\vdash}(d_{2})\otimes\id)(\nu_{\vdash}(d_{1}))
-(\fr_{\dashv}(d_{2})\otimes\id)(\nu_{\vdash}(d_{1}))
-(\id\otimes\fr_{\vdash}(d_{1}))(\nu_{\vdash}(d_{2}))\\[-2mm]
&\quad -(\id\otimes\fl_{\dashv}(d_{1}))(\nu_{\vdash}(d_{2}))
-(\id\otimes\fr_{\vdash}(d_{1}))(\tau(\nu_{\dashv}(d_{2})))
-(\id\otimes\fl_{\dashv}(d_{1}))(\tau(\nu_{\dashv}(d_{2})))\Big)\bullet\Psi'\\
&\; +\Big(\nu_{\dashv}(d_{1}\dashv d_{2})
+(\id\otimes\fr_{\dashv}(d_{1}))(\nu_{\vdash}(d_{2}))
+(\id\otimes\fr_{\vdash}(d_{1}))(\nu_{\dashv}(d_{2}))
-(\id\otimes\fr_{\dashv}(d_{1}))(\nu_{\dashv}(d_{2}))\\[-1mm]
&\quad -(\id\otimes\fl_{\vdash}(d_{1}))(\nu_{\dashv}(d_{2}))
-(\fr_{\dashv}(d_{2})\otimes\id)(\nu_{\dashv}(d_{1}))
+(\id\otimes\fl_{\dashv}(d_{1}))(\nu_{\dashv}(d_{2}))
+(\id\otimes\fr_{\vdash}(d_{1}))(\tau(\nu_{\vdash}(d_{2})))\\
&\quad -(\id\otimes\fr_{\dashv}(d_{1}))(\tau(\nu_{\vdash}(d_{2})))
+(\id\otimes\fl_{\vdash}(d_{1}))(\tau(\nu_{\dashv}(d_{2})))
-(\id\otimes\fl_{\vdash}(d_{1}))(\tau(\nu_{\vdash}(d_{2})))\\
&\quad +(\id\otimes\fr_{\dashv}(d_{1}))(\tau(\nu_{\dashv}(d_{2})))
+(\id\otimes\fl_{\dashv}(d_{1}))(\tau(\nu_{\vdash}(d_{2})))
-(\id\otimes\fr_{\vdash}(d_{1}))(\nu_{\vdash}(d_{2}))
+(\id\otimes\fl_{\dashv}(d_{1}))(\nu_{\dashv}(d_{2}))\\[-1mm]
&\quad -(\id\otimes\fl_{\dashv}(d_{1}))(\nu_{\vdash}(d_{2}))
-(\id\otimes\fr_{\vdash}(d_{1}))(\tau(\nu_{\dashv}(d_{2})))
-(\id\otimes\fl_{\dashv}(d_{1}))(\tau(\nu_{\dashv}(d_{2})))\Big)\bullet\Omega.
\end{align*}
Suppose $(D, \dashv, \vdash, \nu_{\dashv}, \nu_{\vdash})$ is a diNovikov bialgebra.
Then by Eqs. \eqref{ndba5} and \eqref{ndba6}, we get that the term for $\Phi_{c_{1}, c_{2}}$
in the above equation is zero. By Eqs. \eqref{ndba4} and \eqref{ndba7}, we get that
the term for $\Phi'_{c_{1}, c_{2}}$ in the above equation is zero.
By Eqs. \eqref{ndba3} and \eqref{ndba5}, we get that the term for $\Psi'_{c_{1}, c_{2}}$
in the above equation is zero. By Eq. \eqref{ndba5}, we get that the term
for $\Psi_{c_{1}, c_{2}}$ in the above equation is zero. By Eq. \eqref{ndba7}, we
get that the term for $\Omega_{c_{1}, c_{2}}$ in the above equation is zero.
That is, Eq. \eqref{Nbialg1} holds. Similar to the discussion above, we can also obtain
that Eqs. \eqref{Nbialg2} and \eqref{Nbialg3} hold. Thus, $(D\otimes C, \ast, \delta)$
is a Novikov bialgebra.
\end{proof}

Recall that a Novikov bialgebra $(B, \ast, \delta)$ is called {\bf coboundary} if there
exists $r\in B\otimes B$ such that
\begin{align}
\delta(b)=\delta_{r}(b):=\big(\tilde{\fl}_{B}(b)\otimes\id
+\id\otimes(\tilde{\fl}_{B}+\tilde{\fr}_{B})(b)\big)(r),      \label{cobnov}
\end{align}
for any $b\in B$. In this case, we denote this coboundary Novikov bialgebra by
$(B, \ast, \delta_{r})$.

\begin{defi}\label{def:dNPbi}
Let $(B, \ast)$ be a Novikov algebra, $r=\sum_{i}x_{i}\otimes y_{i}\in B\otimes B$.
$$
\mathbf{N}_{r}:=r_{13}\ast r_{23}+r_{12}\ast r_{23}+r_{23}\ast r_{12}+r_{13}\ast r_{12}=0
$$
is called the {\bf Novikov Yang-Baxter equation} (or $\NYBE$) in $(B, \ast)$,
where $r_{13}\ast r_{23}=\sum_{i,j}x_{i}\otimes x_{j}\otimes(y_{i}\ast y_{j})$,
$r_{12}\ast r_{23}=\sum_{i,j}x_{i}\otimes(y_{i}\ast x_{j})\otimes y_{j}$,
$r_{23}\ast r_{12}=\sum_{i,j}x_{j}\otimes(x_{i}\ast y_{j})\otimes y_{i}$ and
$r_{13}\ast r_{12}=\sum_{i,j}(x_{i}\ast x_{j})\otimes y_{j}\otimes y_{i}$.
\end{defi}

Let $(B, \ast)$ be a Novikov algebra. An element $r\in B\otimes B$ is called
{\bf Nov-invariant} if for any $b\in B$,
$$
\big(\tilde{\fl}_{B}(b)\otimes\id+\id\otimes(\tilde{\fl}_{B}+\tilde{\fr}_{B})(b)\big)(r)=0.
$$

\begin{pro}[\cite{CH}]\label{pro:quasass-Nbia}
Let $(B, \ast)$ be a Novikov algebra, $r\in B\otimes B$ and $\delta_{r}:
B\rightarrow B\otimes B$ be the linear map defined by Eq. \eqref{cobnov}.
\begin{enumerate}\itemsep=0pt
\item[$(i)$] If $r$ is a skew-symmetric solution of the $\NYBE$ in $(B, \ast)$, then
     $(B, \ast, \delta_{r})$ is a Novikov bialgebra, which is called a {\bf triangular
     Novikov bialgebra} associated with $r$.
\item[$(ii)$] If $r$ is a solution of the $\NYBE$ in $(B, \ast)$ and $r-\tau(r)$ is
     Nov-invariant, then $(B, \ast, \delta_{r})$ is a Novikov bialgebra,
     which is called a {\bf quasi-triangular Novikov bialgebra} associated with $r$.
\end{enumerate}
\end{pro}

A quasi-triangular Novikov bialgebra $(B, \ast, \delta_{r})$ is called a {\bf factorizable
Novikov bialgebra} if $\mathcal{I}=r^{\sharp}+\tau(r)^{\sharp}: B^{\ast}\rightarrow B$
is an isomorphism of vector spaces.
For some special diNovikov bialgebra induced by a differential perm bialgebra, we
first need to consider the relationship between the solution of the $\DPYBE$ in an
admissible differential perm algebra and the solution of $\DPYBE$ in the induced
diNovikov algebra.

\begin{pro}\label{pro:DNYBE-NYBE}
Let $(D, \dashv, \vdash)$ be a diNovikov algebra, $(C, \star, \omega)$ be a quadratic
Zinbiel algebra and $(D\otimes C, \ast)$ be the induced Novikov algebra by
$(D, \dashv, \vdash)$ and $(C, \star)$. Suppose that $r=\sum_{i}x_{i}\otimes y_{i}
\in D\otimes D$ is a symmetric solution of the $\DNYBE$ in $(D, \dashv, \vdash)$. Then
\begin{align}
\widehat{r}=\sum_{i, j}(x_{i}\otimes e_{j})\otimes(y_{i}\otimes f_{j})
\in(D\otimes C)\otimes(D\otimes C)   \label{N-rmax}
\end{align}
is a skew-symmetric solution of the $\NYBE$ in $(D\otimes C, \ast)$, where
$\{e_{1}, e_{2}, \cdots, e_{n}\}$ is a basis of $C$ and $\{f_{1}, f_{2},\cdots, f_{n}\}$
is the dual basis of $\{e_{1}, e_{2},\cdots, e_{n}\}$ with respect to $\omega(-,-)$.
\end{pro}

\begin{proof}
For any $1\leq s, u, v\leq n$, since
$$
\omega\Big(\sum_{k,l}e_{k}\otimes(f_{k}\star e_{l})\otimes f_{l},\ \
e_{s}\otimes e_{u}\otimes e_{v}\Big)
=\omega(e_{u},\ \ e_{s}\star e_{v})
=\omega\Big(\sum_{k,l}e_{l}\otimes(f_{l}\star e_{k})\otimes f_{k},\ \
e_{s}\otimes e_{u}\otimes e_{v}\Big),
$$
and the nondegeneracy of $\omega(-,-)$, we get $\sum_{k,l}e_{k}\otimes(f_{k}\star e_{l})
\otimes f_{l}=\sum_{k,l}e_{l}\otimes(f_{l}\star e_{k})\otimes f_{k}$. Similarly,
$\sum_{k,l}e_{k}\otimes(f_{k}\star e_{l})\otimes f_{l}=-\sum_{k,l}(e_{l}\star e_{k})
\otimes f_{l}\otimes f_{k}$, $\sum_{k,l}e_{l}\otimes(e_{k}\star f_{l})\otimes f_{k}
=\sum_{k,l}e_{k}\otimes(e_{l}\star f_{k})\otimes f_{l}=-\sum_{k,l}e_{k}\otimes e_{l}
\otimes(f_{l}\star f_{k})$ and $\sum_{k,l}e_{k}\otimes e_{l}\otimes(f_{k}\star f_{l})
=\sum_{k,l}(e_{k}\star e_{l})\otimes f_{l}\otimes f_{k}=\sum_{k,l}e_{k}\otimes(f_{k}
\star e_{l})\otimes f_{l}+\sum_{k,l}e_{k}\otimes(e_{l}\star f_{k})\otimes f_{l}$.
Therefore, we obtain
\begin{align*}
&\;\widehat{r}_{13}\ast\widehat{r}_{23}+\widehat{r}_{12}\ast\widehat{r}_{23}
+\widehat{r}_{23}\ast\widehat{r}_{12}+\widehat{r}_{13}\ast\widehat{r}_{12}\\
=&\;\sum_{i,j}\sum_{k,l}\Big(\big(x_{i}\otimes x_{j}\otimes(y_{i}\vdash y_{j})\big)\bullet
\big(e_{k}\otimes e_{l}\otimes(f_{k}\star f_{l})\big)
+\big(x_{i}\otimes x_{j}\otimes(y_{i}\dashv y_{j})\big)\bullet
\big(e_{k}\otimes e_{l}\otimes(f_{l}\star f_{k})\big)\\[-4mm]
&\qquad\quad +\big(x_{i}\otimes(y_{i}\vdash x_{j})\otimes y_{j}\big)\bullet
\big(e_{k}\otimes(f_{k}\star e_{l})\otimes f_{l}\big)
+\big(x_{i}\otimes(y_{i}\dashv x_{j})\otimes y_{j}\big)\bullet
\big(e_{k}\otimes(e_{l}\star f_{k})\otimes f_{l}\big)\\
&\qquad\quad +\big(x_{j}\otimes(x_{i}\vdash y_{j})\otimes y_{i}\big)\bullet
\big(e_{l}\otimes(e_{k}\star f_{l})\otimes f_{k}\big)
+\big(x_{j}\otimes(x_{i}\dashv y_{j})\otimes y_{i}\big)\bullet
\big(e_{l}\otimes(f_{l}\star e_{k})\otimes f_{k}\big)\\[-1mm]
&\qquad\quad +\big((x_{i}\vdash x_{j})\otimes y_{j}\otimes y_{i}\big)\bullet
\big((e_{k}\star e_{l})\otimes f_{l}\otimes f_{k}\big)
+\big((x_{i}\dashv x_{j})\otimes y_{j}\otimes y_{i}\big)\bullet
\big((e_{l}\star e_{k})\otimes f_{l}\otimes f_{k}\big)\Big)\\
=&\;\sum_{i,j}\sum_{k,l}\Big(x_{i}\otimes x_{j}\otimes(y_{i}\vdash y_{j})
-x_{i}\otimes x_{j}\otimes(y_{i}\dashv y_{j})
+x_{i}\otimes(y_{i}\dashv x_{j})\otimes y_{j}
+x_{j}\otimes(x_{i}\vdash y_{j})\otimes y_{i}\\[-5mm]
&\qquad\qquad\qquad\qquad\qquad\qquad\qquad\qquad\qquad\qquad
+(x_{i}\vdash x_{j})\otimes y_{j}\otimes y_{i}\Big)
\bullet\big(e_{k}\otimes(e_{l}\star f_{k})\otimes f_{l}\big)\\
&\;+\sum_{i,j}\sum_{k,l}\Big(x_{i}\otimes x_{j}\otimes(y_{i}\vdash y_{j})
+x_{i}\otimes(y_{i}\vdash x_{j})\otimes y_{j}
+x_{j}\otimes(x_{i}\dashv y_{j})\otimes y_{i}
+(x_{i}\vdash x_{j})\otimes y_{j}\otimes y_{i}\\[-5mm]
&\qquad\qquad\qquad\qquad\qquad\qquad\qquad\qquad\qquad\qquad
-(x_{i}\dashv x_{j})\otimes y_{j}\otimes y_{i}\Big)
\bullet\big(e_{k}\otimes(f_{k}\star e_{l})\otimes f_{l}\big).
\end{align*}
Moreover, if $r$ is symmetric, we get
\begin{align*}
&\; \widehat{r}_{13}\ast\widehat{r}_{23}+\widehat{r}_{12}\ast\widehat{r}_{23}
+\widehat{r}_{23}\ast\widehat{r}_{12}+\widehat{r}_{13}\ast\widehat{r}_{12}\\
=&\; (\id\otimes\tau)((\tau\otimes\id)(\mathbf{DN}_{r}))\bullet
\Big(\sum_{k,l}e_{k}\otimes(e_{l}\star f_{k})\otimes f_{l}\Big)
+(\id\otimes\tau)(\mathbf{DN}_{r})\bullet
\Big(\sum_{k,l}e_{k}\otimes(f_{k}\star e_{l})\otimes f_{l}\Big).
\end{align*}
Thus, $\widehat{r}$ is a solution of the $\NYBE$ in $(D\otimes C, \ast)$
if $r$ is a symmetric solution of the $\DNYBE$ in $(D, \dashv, \vdash)$.
Finally, since $\sum_{j}e_{j}\otimes f_{j}$ is skew-symmetric (see the proof of
Proposition \ref{pro:DPYBE-DAYBE}), we get $\widehat{r}$ is a skew-symmetric solution of
the $\NYBE$ in $(D\otimes C, \ast)$ if $r$ is a symmetric solution of the
$\DNYBE$ in $(D, \dashv, \vdash)$.
\end{proof}

Thus, by Theorem \ref{thm:diNbia-Nbia} and Proposition \ref{pro:DNYBE-NYBE},
we can give a construction of triangular Novikov bialgebras from triangular
diNovikov bialgebras.

\begin{thm}\label{thm:indu-sdiNbia}
Let $(D, \dashv, \vdash, \nu_{\dashv}, \nu_{\vdash})$ be a diNovikov bialgebra,
$(C, \star, \omega)$ be a quadratic Zinbiel algebra and $(D\otimes C, \ast)$ be the
induced Novikov algebra by $(D, \dashv, \vdash)$ and $(C, \star)$.
If $\nu_{\dashv}=\nu_{\dashv,r}$ and $\nu_{\vdash}=\nu_{\vdash,r}$ are defined
by Eqs. \eqref{ndadr1} and \eqref{ndadr2} for a symmetric solution $r$ of the
$\DNYBE$ in $(D, \dashv, \vdash)$, then the induced Novikov bialgebra
$(D\otimes C, \ast, \delta)=(D\otimes C, \ast, \delta_{\widehat{r}})$, where
$\delta_{\widehat{r}}$ is given by Eq. \eqref{cobnov}, $\widehat{r}$ is defined
by Eq. \eqref{N-rmax}. That is, the induced Novikov bialgebra $(D\otimes C, \ast,
\delta)$ is triangular. Furthermore, we have the following commutative diagram:
$$
\xymatrix@C=2cm@R=0.6cm{
\txt{$r$ \\ {\tiny a symmetric solution}\\ {\tiny of the $\DNYBE$ in $(D, \dashv, \vdash)$}}
\ar[r]^{{\rm Pro.}~\ref{pro:tri-dinov}} \ar[d]_{{\rm Pro.}~\ref{pro:DNYBE-NYBE}}&
\txt{$(D, \dashv, \vdash, \nu_{\dashv,r}, \nu_{\vdash,r})$ \\
{\tiny a triangular diNovikov bialgebra}}\ar[d]^{{\rm Thm.}~\ref{thm:diNbia-Nbia}}\\
\txt{$\widehat{r}$ \\ {\tiny a skew-symmetric solution}\\ {\tiny of the $\NYBE$ in
$(D\otimes C, \ast)$}}  \ar[r]^{{\rm Pro.}~\ref{pro:quasass-Nbia}} &
\txt{$(D\otimes C, \ast, \delta_{r})$ \\ {\tiny a triangular Novikov bialgebra}}}
$$
\end{thm}

\begin{proof}
Let $r=\sum_{i}x_{i}\otimes y_{i}\in D\otimes D$. First, for any $d\in D$ and
$c\in C$, we have
\begin{align*}
&\; \delta(d\otimes c)\\
=&\; \nu_{\vdash,r}(d)\bullet\theta_{\omega}(c)
+\nu_{\dashv,r}(d)\bullet\tau(\theta_{\omega}(c))\\
=&\;\Big(\sum_{i}\big((d\dashv x_{i})\otimes y_{i}-(d\vdash x_{i})\otimes y_{i}
-x_{i}\otimes(d\dashv y_{i})-x_{i}\otimes(y_{i}\vdash d)\big)\Big)
\bullet\Big(\sum_{(c)}c_{(1)}\otimes c_{(2)}\Big)\\[-2mm]
&\qquad+\Big(\sum_{i}\big(y_{i}\otimes(x_{i}\dashv d)-y_{i}\otimes(d\dashv x_{i})
+y_{i}\otimes(d\vdash x_{i})-y_{i}\otimes(x_{i}\vdash d)+(d\dashv y_{i})\otimes x_{i}\big)\Big)
\bullet\Big(\sum_{(c)}c_{(2)}\otimes c_{(1)}\Big),
\end{align*}
where $\theta_{\omega}(c)=\sum_{(c)}c_{(1)}\otimes c_{(2)}$,
$\nu_{\vdash,r}(d)=((\fl_{\dashv}-\fl_{\vdash})(d)\otimes\id-\id\otimes(\fl_{\dashv}
+\fr_{\vdash})(d))(r)=\sum_{i}\big((d\dashv x_{i})\otimes y_{i}-(d\vdash x_{i})\otimes y_{i}
-x_{i}\otimes(d\dashv y_{i})-x_{i}\otimes(y_{i}\vdash d)\big)$ and
$\nu_{\dashv,r}(d)=(\id\otimes(\fr_{\dashv}-\fl_{\dashv}+\fl_{\vdash}-\fr_{\vdash})(d)
+\fl_{\dashv}(d)\otimes\id)(\tau(r))=\sum_{i}\big(y_{i}\otimes(x_{i}\dashv d)
-y_{i}\otimes(d\dashv x_{i})+y_{i}\otimes(d\vdash x_{i})-y_{i}\otimes(x_{i}\vdash d)
+(d\dashv y_{i})\otimes x_{i}\big)$. On the other hand, we have
\begin{align*}
\delta_{\widehat{r}}(d\otimes c)
&=\big(\tilde{\fl}_{D\otimes C}(d\otimes c)\otimes\id+\id\otimes
(\tilde{\fl}_{D\otimes C}+\tilde{\fr}_{D\otimes C})(d\otimes c)\big)(\widehat{r})\\
&=\sum_{i,j}\Big(
\big((d\vdash x_{i})\otimes y_{i}\big)\bullet\big((c\star e_{j})\otimes f_{j}\big)
+\big((d\dashv x_{i})\otimes y_{i}\big)\bullet\big((e_{j}\star c)\otimes f_{j}\big)\\[-5mm]
&\qquad\quad
+\big(x_{i}\otimes(d\vdash y_{i})\big)\bullet\big(e_{j}\otimes(c\star f_{j})\big)
+\big(x_{i}\otimes(d\dashv y_{i})\big)\bullet\big(e_{j}\otimes(f_{j}\star c)\big)\\[-1mm]
&\qquad\quad
+\big(x_{i}\otimes(y_{i}\vdash d)\big)\bullet\big(e_{j}\otimes(f_{j}\star c)\big)
+\big(x_{i}\otimes(y_{i}\dashv d)\big)\bullet\big(e_{j}\otimes(c\star f_{j})\big)\Big).
\\
&=\Big(\sum_{i}\big((d\dashv x_{i})\otimes y_{i}-(d\vdash x_{i})\otimes y_{i}
-x_{i}\otimes(d\dashv y_{i})-x_{i}\otimes(y_{i}\vdash d)\big)\Big)\bullet
\Big(\sum_{(c)}c_{(1)}\otimes c_{(2)}\Big)\\[-2mm]
&\quad +\Big(\sum_{i}\big((d\dashv x_{i})\otimes y_{i}+x_{i}\otimes(d\vdash y_{i})
-x_{i}\otimes(d\dashv y_{i})-x_{i}\otimes(y_{i}\vdash d)+x_{i}\otimes(y_{i}\dashv d)
\big)\Big)\bullet\Big(\sum_{(c)}c_{(2)}\otimes c_{(1)}\Big).
\end{align*}
Since for any basis element $e_{s}, e_{t}\in P$,
$$
\omega\Big(\sum_{(c)}c_{(1)}\otimes c_{(2)},\; e_{s}\otimes e_{t}\Big)
=\omega(c,\; e_{s}\diamond e_{t})
=-\omega\Big(\sum_{j}(c\star e_{j})\otimes f_{j},\ \ e_{s}\otimes e_{t}\Big)
$$
and $\omega(-,-)$ is nondegenerate, we get $\sum_{(c)}c_{(1)}\otimes c_{(2)}=-
\sum_{j}(c\star e_{j})\otimes f_{j}$. Similarly, $\sum_{(c)}c_{(2)}\otimes c_{(1)}
=\sum_{j}e_{j}\otimes(c\star f_{j})$ and $\sum_{j}(e_{j}\star c)\otimes f_{j}
=-\sum_{j}e_{j}\otimes(f_{j}\star c)=\sum_{(c)}c_{(1)}\otimes c_{(2)}+\sum_{(c)}c_{(2)}
\otimes c_{(1)}$.
Hence we obtain that
\begin{align*}
&\; \delta_{\widehat{r}}(d\otimes c)\\
=&\; \Big(\sum_{i}\big((d\dashv x_{i})\otimes y_{i}-(d\vdash x_{i})\otimes y_{i}
-x_{i}\otimes(d\dashv y_{i})-x_{i}\otimes(y_{i}\vdash d)\big)\Big)\bullet
\Big(\sum_{(c)}c_{(1)}\otimes c_{(2)}\Big)\\[-2mm]
&\quad +\Big(\sum_{i}\big((d\dashv x_{i})\otimes y_{i}+x_{i}\otimes(d\vdash y_{i})
-x_{i}\otimes(d\dashv y_{i})-x_{i}\otimes(y_{i}\vdash d)+x_{i}\otimes(y_{i}\dashv d)
\big)\Big)\bullet\Big(\sum_{(c)}c_{(2)}\otimes c_{(1)}\Big)\\
=&\; \delta(d\otimes c),
\end{align*}
since $r$ is symmetric. Thus, we get that $(D\otimes C, \ast, \delta)=(D\otimes C, \ast,
\delta_{\widehat{r}})$ as Novikov bialgebras, and the diagram is commutative.
\end{proof}

\begin{ex}\label{ex:diN-Nov}
Let $(P={\rm span}_{\Bbbk}\{x_{1}, x_{2}\}, \dashv, \vdash, \nu_{\dashv}, \nu_{\vdash})$
be the $2$-dimensional triangular diNovikov bialgebra associated with $r=x_{2}
\otimes x_{2}$ given in Example \ref{ex:ind-tridN} and $(C={\rm span}_{\Bbbk}\{e_{1},
e_{2}, e_{3}, e_{4}\}, \star, \omega)$ be the $4$-dimensional quadratic Zinbiel algebra
given in Example \ref{ex:qu-zib}. Then we get a $8$-dimensional Novikov bialgebra
$(P\otimes C, \ast, \delta)$, where the nonzero products are given by
\begin{align*}
&(x_{1}\otimes e_{1})\ast(x_{1}\otimes e_{1})=x_{2}\otimes e_{2},\qquad\quad
(x_{1}\otimes e_{1})\ast(x_{1}\otimes e_{4})=2x_{2}\otimes e_{3}-x_{2}\otimes e_{2},\\
&(x_{1}\otimes e_{4})\ast(x_{1}\otimes e_{4})=x_{2}\otimes e_{3},\qquad\quad
(x_{1}\otimes e_{4})\ast(x_{1}\otimes e_{1})=2x_{2}\otimes e_{2}-x_{2}\otimes e_{3},
\end{align*}
and the product $\delta$ is zero. Moreover, one can check that
$$
\widehat{r}=(x_{2}\otimes e_{1})\otimes(x_{2}\otimes e_{3})
+(x_{2}\otimes e_{2})\otimes(x_{2}\otimes e_{4})
-(x_{2}\otimes e_{3})\otimes(x_{2}\otimes e_{1})
-(x_{2}\otimes e_{4})\otimes(x_{2}\otimes e_{2})
$$
is a skew-symmetric solution of the $\NYBE$ in $(P\otimes C, \ast)$ and $\delta_{\hat{r}}=0$.
\end{ex}

Next, we consider the operators forms of solution of the $\NYBE$ in a Novikov algebra.
Let $(B, \ast)$ be a Novikov algebra and $(V, \tilde{\kl}, \tilde{\kr})$ be a representation
of $(B, \ast)$. Recall that a linear map $T: V\rightarrow B$ is called an {\bf
$\mathcal{O}$-operator of $(B, \ast)$ associated to $(V, \tilde{\kl}, \tilde{\kr})$} if
for any $v_{1}, v_{2}\in V$,
$$
T(v_{1})\ast T(v_{2})=T\big(\tilde{\kl}(T(v_{1}))(v_{2})+\tilde{\kr}(T(v_{2}))(v_{1})\big).
$$
For the solution of the $\NYBE$ in a Novikov algebra and
$\mathcal{O}$-operator, we have

\begin{pro}[\cite{HBG}]\label{pro:o-nov}
Let $(B, \ast)$ be a Novikov algebra and $r\in B\otimes B$ be skew-symmetric. Then $r$
is a solution of the $\NYBE$ in $(B, \ast)$ if and only if $r^{\sharp}: B^{\ast}\rightarrow
B$ is an $\mathcal{O}$-operator of $(B, \ast)$ associated to the coregular representation
$(B^{\ast}, \tilde{\fl}_{B}+\tilde{\fr}_{B}, -\tilde{\fr}_{B})$.
\end{pro}

For the $\mathcal{O}$-operator of a diNovikov algebra and
the $\mathcal{O}$-operator of the induced Novikov algebra, by Proposition
\ref{pro:o-nov} and \ref{pro:o-dinov}, we have the following corollary.

\begin{cor}\label{cor:o-N-diN}
Let $(D, \dashv, \vdash)$ be a diNovikov algebra, $(C, \star, \omega)$ be a quadratic
Zinbiel algebra and $(D\otimes C, \ast)$ be the induced Novikov algebra.
Suppose $r$ is a symmetric solution of the $\DNYBE$ in $(D, \dashv, \vdash)$. Then we have
the following commutative diagram:
$$
\xymatrix@C=3cm@R=0.5cm{
\txt{$r$ \\ {\tiny a symmetric solution} \\ {\tiny of the $\DNYBE$ in $(D, \dashv, \vdash)$}}
\ar[d]_-{{\rm Pro.}~\ref{pro:DNYBE-NYBE}}\ar[r]^-{{\rm Pro.}~\ref{pro:o-dinov}} &
\txt{$r^{\sharp}$\\ {\tiny an $\mathcal{O}$-operator of $(D, \dashv, \vdash)$} \\
{\tiny associated to coregular representation}}
\ar[d]^-{\mbox{$-\otimes\kappa^{\sharp}$}} \\
\txt{$\widehat{r}$ \\ {\tiny a skew-symmetric solution} \\ {\tiny of the $\NYBE$ in
$(D\otimes C, \ast)$}} \ar[r]^-{{\rm Pro.}~\ref{pro:o-nov}}
& \txt{$\widehat{r}^{\sharp}=r^{\sharp}\otimes\kappa^{\sharp}$ \\
{\tiny an $\mathcal{O}$-operator of $(D\otimes C, \ast)$} \\
{\tiny associated to coregular representation}}}
$$
where $\widehat{r}$ is defined by Eq. \eqref{N-rmax} and $\kappa:=\sum_{j}e_{j}\otimes
f_{j}\in P\otimes P$ is given in the proof of Theorem \ref{thm:indu-asssdibia}.
\end{cor}

\begin{proof}
Similarly to the proof of Theorem \ref{thm:indu-asssdibia}, we get
$\widehat{r}^{\sharp}=r^{\sharp}\otimes\kappa^{\sharp}$. Thus, by Propositions
\ref{pro:DNYBE-NYBE}, \ref{pro:o-nov} and \ref{pro:o-dinov}, we obtain the
diagram above is commutative.
\end{proof}

\subsection{From differential infinitesimal bialgebras to Novikov bialgebras}
\label{subsec:ASI-Nov}
The Novikov bialgebras induced by differential infinitesimal bialgebras have been
studied in \cite{HBG1} and \cite{CH}. Let $(A, \cdot, \Delta)$ be a infinitesimal bialgebra
and $\fd: A\rightarrow A$ be a derivation on $(A, \cdot, \Delta)$. Then $(A, \cdot,
\Delta, \fd, -\fd)$ is a differential infinitesimal bialgebra. In this subsection,
we consider the Novikov bialgebras induced by the differential infinitesimal bialgebras
in the form of $(A, \cdot, \Delta, \fd, -\fd)$. First, we consider the Novikov coalgebras
induced by codifferential coalgebras.

\begin{pro}[\cite{HBG1}]\label{pro:ind-conov}
Let $(A, \Delta, -\fd)$ be a codifferential coalgebra.
Define a coproduct $\delta_{\fd}: A\rightarrow A\otimes A$ by
\begin{align}
\delta_{\fd}=(\id\otimes\fd)\circ\Delta.     \label{ind-conov}
\end{align}
Then $(A, \delta_{\fd})$ is a Novikov coalgebra, which is called the {\bf Novikov
coalgebra induced from $(A, \Delta, -\fd)$}.
\end{pro}

In \cite{HBG1}, Hong, Bai and Guo have provided some conditions that a differential
infinitesimal bialgebra can induce a Novikov bialgebra. Specifically, for differential
infinitesimal bialgebra of the form $(A, \cdot, \Delta, \fd, -\fd)$, we have
the following conclusion.

\begin{thm}[\cite{HBG1}]\label{thm:ind-novbia}
Let $(A, \cdot, \Delta, \fd, -\fd)$ be a differential infinitesimal bialgebra.
Define a binary operation $\diamond_{\fd}$ and a coproduct $\delta_{\fd}$ on $A$ by
Eqs. \eqref{ind-nov} and \eqref{ind-conov} respectively. Then $(A, \diamond_{\fd},
\delta_{\fd})$ is a Novikov bialgebra, which is called {\bf the Novikov bialgebra
induced by the differential infinitesimal bialgebra $(A, \cdot, \Delta, \fd, -\fd)$}.
\end{thm}

Let $(P, \diamond, \vartheta, \fd, -\fd)$ be a differential perm bialgebra and
$(C, \star, \omega)$ be a quadratic Zinbiel algebra. On the one hand, this
differential perm bialgebra induces a diNovikov bialgebra $(P, \dashv_{\fd},
\vdash_{\fd}, \nu_{\dashv,\fd}, \nu_{\vdash,\fd})$. By considering the tensor product
of $(P, \dashv_{\fd}, \vdash_{\fd}, \nu_{\dashv,\fd}, \nu_{\vdash,\fd})$
and $(C, \star, \omega)$, we get a Novikov bialgebra $(P\otimes C, \ast, \delta)$.
On the other hand, there is a differential infinitesimal bialgebra $(P\otimes C,
\cdot, \Delta, \hat{\fd}, -\hat{\fd})$. This differential infinitesimal bialgebra
induces a Novikov bialgebra $(P\otimes C, \ast_{\hat{\fd}}, \delta_{\hat{\fd}})$.
In fact, we can prove that the two diNovikov bialgebras obtained are the same.

\begin{thm}\label{thm:commdig}
Let $(P, \diamond, \vartheta, \fd, -\fd)$ be a differential perm bialgebra and
$(C, \star, \omega)$ be a quadratic Zinbiel algebra. With the above notations,
we have the following commutative diagram:
$$
\xymatrix@C=2cm@R=0.7cm{
\txt{$(P, \diamond, \vartheta, \fd, -\fd)$ \\ {\tiny a differential perm bialgebra}}
\ar[d]_{{\rm Thm.}~\ref{thm:ind-dinovbia}} \ar[r]^{{\rm Thm.}~\ref{thm:permbia-ASI}}
&\txt{$(P\otimes C, \cdot, \Delta, \hat{\fd}, -\hat{\fd})$\\
{\tiny a differential infinitesimal bialgebra}}\ar[d]^{{\rm Thm.}~\ref{thm:ind-novbia}} \\
\txt{$(P, \dashv_{\fd}, \vdash_{\fd}, \nu_{\dashv,\fd}, \nu_{\vdash,\fd})$ \\
{\tiny a diNovikov bialgebra}}
\ar[r]^{{\rm Thm.}~\ref{thm:diNbia-Nbia}\qquad}
& \txt{$(P\otimes C, \ast_{\hat{\fd}}, \delta_{\hat{\fd}})
=(P\otimes C, \ast, \delta)$ \\ {\tiny a Novikov bialgebra}}}
$$
\end{thm}

\begin{proof}
Here we only show that $(P\otimes C, \ast_{\hat{\fd}}, \delta_{\hat{\fd}})
=(P\otimes C, \ast, \delta)$ as Novikov bialgebras. In Section \ref{sec:prel}, we have
shown that $(P\otimes C, \ast_{\hat{\fd}})=(P\otimes C, \ast)$ as Novikov algebras.
For any $c\in C$ and $p\in P$,
\begin{align*}
\delta(p\otimes c)&=\nu_{\vdash,\fd}(p)\bullet\theta_{\omega}(c)
+\nu_{\dashv,\fd}(p)\bullet\tau(\theta_{\omega}(c))\\
&=(\id\otimes\fd)(\vartheta(p))\bullet\theta_{\omega}(c)
+\tau((\fd\otimes\id)(\vartheta(p)))\bullet\tau(\theta_{\omega}(c))\\
&=(\id\otimes\hat{\fd})(\vartheta(p)\bullet\theta_{\omega}(c)+\tau(\vartheta(p))
\bullet\tau(\theta_{\omega}(c)))\\
&=\delta_{\hat{\fd}}(p\otimes c).
\end{align*}
That is, $\delta=\delta_{\hat{\fd}}$, $(P\otimes C, \ast_{\hat{\fd}}, \delta_{\hat{\fd}})
=(P\otimes C, \ast, \delta)$ as Novikov bialgebras.
\end{proof}

In \cite{CH}, we have discussed in detail the relationship between the solutions of
the $\DAYBE$ in an admissible differential algebra and the solutions of the $\NYBE$
in the induced Novikov algebra. Specifically, for differential bialgebras
of the form $(A, \cdot, \fd, -\fd)$ and the induced Novikov algebras, and we have
the following conclusion.

\begin{pro}[\cite{CH}]\label{pro:indu-NYBE}
Let $(A, \cdot, \fd, -\fd)$ be an admissible differential algebra, $r=\sum_{i}x_{i}
\otimes y_{i}\in A\otimes A$ and $(A, \ast_{\fd})$ be the induced Novikov algebra.
Then $r$ is a solution of the $\NYBE$ in $(A, \ast_{\fd})$ if $r$ is a solution
of the $\DAYBE$ in $(A, \cdot, \fd, -\fd)$. Furthermore, $r+\tau(r)$ is also Nov-invariant
if $r+\tau(r)$ is ass-invariant.

In particular, each skew-symmetric solution of the $\DAYBE$ in $(A, \cdot, \fd, -\fd)$
is also a skew-symmetric solution of the $\NYBE$ in the induced Novikov algebra
$(A, \ast_{\fd})$.
\end{pro}

Therefore, by Proposition \ref{pro:indu-NYBE}, we get that the the induced Novikov
bialgebra given in Theorem \ref{thm:ind-novbia} by a quasi-triangular (resp.
triangular, factorizable) differential infinitesimal bialgebra is also
quasi-triangular (resp. triangular, factorizable).

\begin{thm}[\cite{CH}]\label{thm:indu-spNbia}
Let $(A, \cdot, \Delta, \fd, -\fd)$ be a differential infinitesimal bialgebra,
where $\Delta=\Delta_{r}$ is given by Eq. \eqref{cobass} for some $r\in A\otimes A$.
Suppose linear maps $\delta_{r}, \delta_{\fd}: A\rightarrow A\otimes A$ are defined
by Eqs. \eqref{cobnov} and \eqref{ind-conov} for $\Delta=\Delta_{r}$ respectively.
If $r$ is a solution of the $\DAYBE$ in $(A, \cdot, \fd, -\fd)$, then
$\delta_{r}=\delta_{\fd}$. Therefore, we obtain
\begin{enumerate}\itemsep=0pt
\item[$(i)$] $(A, \ast_{\fd}, \delta_{\fd})$ is a quasi-triangular Novikov bialgerba if
     $(A, \cdot, \Delta, \fd, -\fd)$ is quasi-triangular;
\item[$(ii)$] $(A, \ast_{\fd}, \delta_{\fd})$ is a triangular Novikov bialgerba if
     $(A, \cdot, \Delta, \fd, -\fd)$ is triangular;
\item[$(iii)$] $(A, \ast_{\fd}, \delta_{\fd})$ is a factorizable Novikov bialgerba if
     $(A, \cdot, \Delta, \fd, -\fd)$ is factorizable.
\end{enumerate}
\end{thm}

By this theorem, we obtain the following commutative diagram:
$$
\xymatrix@C=2cm@R=0.6cm{
\txt{$r$ \\ {\tiny a solution of the $\DAYBE$ in $(A, \cdot, \fd, \partial)$}\\
{\tiny such that $r+\tau(r)$ is ass-invariant}}
\ar[d]_{{\rm Pro.}~\ref{pro:indu-NYBE}}\ar[r]^{\quad{\rm Pro.}~\ref{pro:diff-bia}} &
\txt{$(A, \cdot, \Delta_{r}, \fd, \partial)$ \\ {\tiny a quasi-triangular differential}\\
{\tiny infinitesimal bialgebra}} \ar[d]^{{\rm Thm.}~\ref{thm:indu-spNbia}} \\
\txt{$r$ \\ {\tiny a solution of the $\NYBE$ in $(A, \ast_{\fd})$}\\
{\tiny such that $r+\tau(r)$ is Nov-invariant}}
\ar[r]^{\qquad {\rm Pro.}~\ref{pro:quasass-Nbia}\qquad} &
\txt{$(A, \ast_{\fd}, \delta_{r})=(A, \ast_{\fd}, \delta_{\fd})$ \\
{\tiny a quasi-triangular}\\ {\tiny Novikov bialgebra}}}
$$

\begin{ex}\label{ex:ind-comm}
Let $(P\otimes C, \cdot, \Delta, \hat{\fd}, -\hat{\fd})$ be the $8$-dimensional
differential infinitesimal bialgebra given in Example \ref{ex:ind-triasbi}. By using
Eqs. \eqref{ind-nov} and \eqref{ind-conov}, we can directly calculate:
\begin{align*}
&(x_{1}\otimes e_{1})\ast_{\hat{\fd}}(x_{1}\otimes e_{1})=x_{2}\otimes e_{2},\qquad\quad
(x_{1}\otimes e_{1})\ast_{\hat{\fd}}(x_{1}\otimes e_{4})=2x_{2}\otimes e_{3}
-x_{2}\otimes e_{2},\\
&(x_{1}\otimes e_{4})\ast_{\hat{\fd}}(x_{1}\otimes e_{4})=x_{2}\otimes e_{3},\qquad\quad
(x_{1}\otimes e_{4})\ast_{\hat{\fd}}(x_{1}\otimes e_{1})=2x_{2}\otimes e_{2}
-x_{2}\otimes e_{3},
\end{align*}
and $\delta_{\hat{\fd}}=0$. Thus, the Novikov bialgebra $(P\otimes C, \ast_{\hat{\fd}},
\delta_{\hat{\fd}})$ induced by $(P\otimes C, \cdot, \Delta, \hat{\fd}, -\hat{\fd})$
is not only the Novikov bialgebra given in Example \ref{ex:diN-Nov}, but also
triangular Novikov bialgebra associated with the skew-symmetric solution $\widehat{r}$
of the $\NYBE$ in $(P\otimes C, \ast_{\hat{\fd}})$.
Thus, we provide a specific example of the commutative graph in Theorem \ref{thm:commdig}
for triangular bialgebras.
\end{ex}

Moreover, for the $\mathcal{O}$-operator of a differential algebra and
the $\mathcal{O}$-operator of the induced Novikov algebra, we have the following
corollary.

\begin{cor}\label{cor:r-dass-N}
Let $(A, \cdot, \fd, -\fd)$ be an admissible differential algebra and $(A, \ast_{\fd})$
be the Novikov algebra induced from $(A, \cdot, \fd)$. Suppose $r$ is a
skew-symmetric solution of the $\DAYBE$ in $(A, \cdot, \fd, -\fd)$. Then we have the
following commutative diagram:
$$
\xymatrix@C=3cm@R=0.5cm{
\txt{$r$ \\ {\tiny a skew-symmetric solution} \\ {\tiny of the $\DAYBE$ in
$(A, \cdot, \fd, -\fd)$}}
\ar[d]_-{{\rm Pro.}~\ref{pro:indu-NYBE}}\ar[r]^-{{\rm Pro.}~\ref{pro:o-dass}} &
\txt{$r^{\sharp}$\\ {\tiny an $\mathcal{O}$-operator of $(A, \cdot, \fd)$} \\
{\tiny associated to $(A^{\ast}, -\fu_{A}^{\ast}, -\fd^{\ast})$}} \ar[d] \\
\txt{$r$ \\ {\tiny a skew-symmetric solution} \\ {\tiny of the $\NYBE$ in
$(A, \ast_{\fd})$}} \ar[r]^-{{\rm Pro.}~\ref{pro:o-nov}}
& \txt{$r^{\sharp}$ \\ {\tiny an $\mathcal{O}$-operator of $(A, \ast_{\fd})$ } \\
{\tiny associated to $(A^{\ast}, \tilde{\fl}_{A}+\tilde{\fr}_{A}, -\tilde{\fr}_{A})$}}}
$$
\end{cor}

Thus, by the commutative diagrams obtained in Theorems \ref{thm:indu-sdiNbia} and
\ref{thm:indu-spdiNbia}, Corollaries \ref{cor:o-dass-dperm}, \ref{cor:r-dass-N},
\ref{cor:o-N-diN} and \ref{cor:o-dass-dN}, and the commutative diagrams after Theorem
\ref{thm:indu-asssdibia} and \ref{thm:indu-spNbia}, we can obtain the
three-dimensional commutative diagram given in Section \ref{sec:intr}.

\begin{rmk}\label{rmk:Leibconfor}
Let $(B, \ast)$ be a Novikov algebra and $\Bbbk[\partial]B$ be the free $\Bbbk[\partial]$
module. Define a $\lambda$-bracket $[-_{\lambda}-]$ on $\Bbbk[\partial]B$ by $[b_{\lambda}
b']=\partial(b'\ast b)+\lambda(b\ast b'+b'\ast b)$. Then $(\Bbbk[\partial]B,
[-_{\lambda}-])$ is a Lie conformal algebra. In \cite{HBG2}, the authors have given a
construction of Lie conformal bialgebras from Novikov bialgebras. By
using our results in this paper, we can get a Lie conformal bialgebra
from a differential infinitesimal bialgebra.
\end{rmk}

\bigskip
\noindent
{\bf Acknowledgements. } This work was financially supported by National
Natural Science Foundation of China (No.11771122).

\smallskip
\noindent
{\bf Declaration of interests.} The authors have no conflicts of interest to disclose.

\smallskip
\noindent
{\bf Data availability.} Data sharing is not applicable to this article as no new data were
created or analyzed in this study.

 \end{document}